\documentclass[10pt, a4paper, twoside]{amsart}
\usepackage[utf8]{inputenc}

\usepackage{graphicx}
\usepackage[shortlabels]{enumitem}
\usepackage{array}
\usepackage[noadjust]{cite}
\usepackage{url}

\usepackage{anysize}
\usepackage{fancyheadings}
\usepackage{fancyhdr}
\usepackage{xfrac}
\usepackage{tikz}
\usetikzlibrary{arrows}

\usepackage{chngcntr}
\counterwithin{figure}{section}
\numberwithin{equation}{section}
\usepackage{xcolor}
\usepackage{graphicx}
\usepackage{subcaption}

\usepackage{float}
\usepackage{hyperref}
\usepackage{amssymb}
\usepackage{amsmath,accents}
\usepackage{amsfonts}
\usepackage{amsthm}
\usepackage{mathrsfs} %special cool script

\usepackage{subcaption}
\usepackage{caption}
\newtheorem{theorem}{Theorem}%[chapter]
\newtheorem{lemma}{Lemma}%[chapter]
\newtheorem{deff}{Definition}%[chapter]
\usepackage{xfrac}

\usepackage{enumitem, hyperref}
\makeatletter
\def\namedlabel#1#2{\begingroup
    #2%
    \def\@currentlabel{#2}%
    \phantomsection\label{#1}\endgroup
}
\makeatother
\usepackage{anysize}
\marginsize{3.25cm}{3.25cm}{2.5cm}{2.5cm}
\newtheorem{remark}{Remark}[section]

\usepackage{listings}%opening

\title{Neural ODE Control for classification, approximation and transport}
\author[Domènec Ruiz-Balet]{Domènec Ruiz-Balet$^{\dagger}$}
\address{}
\curraddr{}
\email{domenec.ruiz@deusto.es\\domenec.ruizi@uam.es}
\thanks{}

\author[Enrique Zuazua]{Enrique Zuazua$^{\star\dagger}$}
\address{}

\curraddr{}
\email{enrique.zuazua@fau.de}
\thanks{
\textbf{Funding}: Authors were funded by the European Research Council (ERC) under the European Union’s Horizon 2020 research and innovation programme (grant agreement NO. 694126-DyCon).
the work of the second author was also funded by the European Union’s Horizon 2020 research and innovation programme under the Marie Sklodowska-Curie grant agreement No.765579-ConFlex.D.P., the Alexander von Humboldt-Professorship program, the Transregio 154, Mathematical Modelling, Simulation and Optimization using
 the Example of Gas Networks,
  of the German DFG, project C08, grant MTM2017-92996-C2-1-R COSNET of MINECO (Spain) and by the Air Force Office of Scientific Research (AFOSR) under Award NO. FA9550-18-1-0242.\newline\textcolor{white}{.}\newline
  $\dagger$Departamento de Matemáticas, Universidad Autónoma de Madrid, 28049 Madrid, Spain.\newline
\textcolor{white}{..}Chair of Computational Mathematics, Fundación Deusto Av. de las Universidades 24, 48007 Bilbao, BasqueCountry, Spain.\newline
$\star$Chair in Applied Analysis, Alexander von Humboldt-Professorship, Department of Data Science Friedrich-Alexander-Universität, Erlangen-Nürnberg, 91058 Erlangen, Germany.\newline
}

\begin{document}

\maketitle
 \begin{abstract}
We analyze Neural Ordinary Differential
Equations (NODEs) from a control theoretical perspective to address some of the main properties and paradigms of Deep Learning (DL), in particular, data classification and universal approximation. These objectives are  tackled and
achieved from the perspective of the simultaneous control of systems of NODEs. For instance, in the context of classification, each item to be classified corresponds to a different initial datum for the control problem of the NODE, to be classified, all of them by the same common control, to the location (a subdomain of the euclidean space) associated to each label.  
Our
proofs are genuinely nonlinear and constructive, allowing us to estimate the complexity of
the control strategies we develop. The nonlinear nature of the activation functions
governing the dynamics of NODEs under consideration plays a key role in our
proofs, since it allows deforming half of the phase space while the other half remains invariant,
a property that classical models in mechanics do not fulfill. This very property allows to build elementary controls inducing specific dynamics and transformations whose concatenation, along with properly chosen hyperplanes, allows achieving our
goals in finitely many steps. 
The nonlinearity of the dynamics is assumed to be Lipschitz. Therefore, our results apply also in the particular case of the ReLU activation function.
We also present the counterparts in the context of the
control of neural transport equations, establishing a link between optimal transport and
deep neural networks.

\end{abstract}

\section{Introduction}

In this paper, we prove that the dynamical control properties of Neural Ordinary Differential Equations (NODEs) allow for the understanding of some of the main Deep Neural Networks' properties. In particular, we will focus on data classification and universal approximation.

Given a dataset $\{(x_i,y_i)\}_{i=1}^N\subset\mathbb{R}^d\times\mathbb{R}^d$,  $x_i$ being the input and $y_i$ the corresponding label, Supervised Learning (SL) aims to find a map $\phi:\mathbb{R}^d\to\mathbb{R}^d$  that associates $x$ to $y$ for later use on the allocation of a label to new unknown data. In classification, the labels take values in a finite set of $\mathbb{R}^d$ (typically associated with a subset of the natural numbers), while in regression, the labels correspond to continuous values of $\mathbb{R}^d$.

Neural networks (NN) are widely used in this context. A shallow neural network  can be written as a linear combination of the form:
\begin{equation}\label{nn}
  y(x)=\sum_{i=1}^n W_i\boldsymbol{\sigma}(A_ix+b_i)
\end{equation}
where $\boldsymbol{\sigma}:\mathbb{R}^d\to\mathbb{R}^d$ is the so called activation function (more details will be given later on), and $W_i,A_i\in\mathbb{R}^{d\times d}$, $b_i\in\mathbb{R}^d, \, i\in\{1,...,n\}$.

 In the seminal paper by Cybenko \cite{cybenko1989approximation}, it was proved that, under certain assumptions on $\boldsymbol{\sigma}$,  any continuous function in a compact set can be approximated by NN of the form \eqref{nn} in the supremum norm as $n\to \infty$. The product $dn$ is  the width of the network. This density result is the so called Universal Approximation Theorem and it has been the object of intensive study and
several improvements and generalizations (see \cite{pinkus1999approximation,leshno1993multilayer,hornik1989multilayer}).

A deep neural network is obtained from the composition of shallow neural networks of the type \eqref{nn} and in recent years they have played an important role in diverse applications,
\cite{lecun2015deep}.  Each of this composition steps is refereed to as a {\it layer} and the depth of the network is the total number of compositions operated. Deep neural networks can also be interpreted as discrete dynamical systems and can be employed, in particular,  to prove universal approximation theorems, \cite{daubechies2019nonlinear}. 
The parameters $W, A$ and $b$ entering in the design of the networks are often found via optimization, using, for instance, a least-squares error type 
functionals \cite{chizat2020implicit}, in the so-called {\it learning} process.

Residual neural networks (ResNets) (see \cite{he2016deep}) refer to the variant in which an ``{\it inertia}`` term is added in each layer, leading to a discrete dynamical system of the form
\begin{equation}\label{Euler} x_{l+1}=x_l+hW_l\boldsymbol{\sigma}(A_lx_{l}+b_l),\end{equation}
where $h$ is a positive real number. 
ResNets can also be understood as an Euler discretization scheme of an ODE, the so called  Neural ODEs (NODEs) (\cite{weinan2017proposal,haber2017stable,chen2018neural}):
\begin{equation}\label{CTNNo}
\dot{x}(t)=W(t)\boldsymbol{\sigma}(A(t)x(t)+b(t)).
\end{equation}

This paper aims to develop a dynamical control theoretical analysis of NODEs \eqref{CTNNo}. The time-dependent coefficients $A(t)$, $W(t)$ and $b(t)$ are not chosen using least squares and optimization approaches  (as in \cite{esteve2020large} and the references therein). Rather, we adopt a dynamic controllability approach according to which the time-dependent parameters $A(t)$, $W(t)$ and $b(t)$ are viewed as controls and are built in a constructive way, being defined in a piecewise-constant switching manner, adapted to the distribution of data to be classified. This allows for algorithmic constructions that permit to measure the complexity of the controls needed to achieve the main goals, in particular, classification and universal approximation. Our approach does not require either the regularity of the nonlinearity that classical ODE control methods need, often based on linearization and Lie brackets. In this way our results apply also in the case of the Rectified Linear Unit  (ReLU)  activation function which is Lipschitz but fails to be $C^1$.

The properties and problems we study are addressed in an ordered and systematic manner, according to their complexity.

The first one we address is the classification problem. From a control perspective, it can be formulated as a simultaneous control problem. The data to be classified are viewed as the initial data of the Cauchy problem associated to \eqref{CTNNo}. Classification  requires to build controls  $A(t)$, $W(t)$ and $b(t)$ that simultaneously drive all initial data to their corresponding final destination, allocated  depending on their label. The final targets that we prefix are disjoint sets corresponding to a partition of the space $\mathbb{R}^d$. The choice of the partition is not unique, and the specific control problem to be addressed depends on it. Although our methods allow to handle arbitrary partitions, for the sake of simplicity,  we consider the particular one in which the space is partitioned in parallel strips, which suffices to achieve the goal of classification. In Theorem \ref{TH1}, Section \ref{SCLASS}, we develop our new constructive proof of classification. As mentioned above, our construction allows to master the complexity of the required classifying controls, in terms of the distribution and structure of the dataset to be classified.

Theorem \ref{TH1} assures that each  initial data can be driven simultaneously to the strip corresponding to its label. But,  in fact, our constructions allow driving the system to much more precise targets so that each trajectory approaches arbitrarily a prescribed target point in $\mathbb{R}^d$. This is proved in  Section \ref{SSCONTROL}, Theorem \ref{TH2}.

Both results hold in an arbitrary time horizon $T>0$. This can be easily achieved by scaling arguments, the general case $T>0$ being reduced to $T=1$. Obviously, the $L^\infty$-norm of the controls employed depends then monotonically and linearly on $1/T$, $T$ being the length of time-horizon.

From a control perspective, these results can be viewed as particular instances of simultaneous or ensemble controllability properties. This topic has been considered in the control community for a variety of systems \cite{lions1988exact,loheac2016averaged,tucsnak2000simultaneous}.   From that perspective, the results in the present paper are new since the ODE system under consideration is the same for all initial data and the number $N$ of the controlled initial data  is arbitrarily large, while the number of available controls is limited: $2d^2+d$. This kind of simultaneous control result would be impossible for linear finite or infinite-dimensional systems, where the simultaneous control never occurs when different trajectories solve the same system. The specific nature of the nonlinearity of the NODEs and, as explained above,  the very structure of the activation function is the key of our construction.

Our proofs are genuinely nonlinear and strongly rely on the structure of the activation function $\boldsymbol{\sigma}$ that models the nonlinearity of the NODE \eqref{CTNNo} under consideration. The key ingredient of our proofs is that the activation function allows splitting the phase-space into two half-spaces, so that one remains invariant under the action of the control, while the other one evolves in the direction we wish, approaching the prescribed final target. Such construction is not possible for the classical nonlinear ODE systems in Mechanics in which all trajectories evolve simultaneously, driven by the nonlinear flow.  The proofs are constructive, allowing us to estimate the complexity of the needed piecewise-constant controls for achieving such results. Moreover, our techniques are of dynamic nature, and they are not based on optimization or linearization arguments,  that  would not allow  to obtain the global results we prove. 

Once the classification problem is solved and its upgraded simultaneous controllability version is proved in Theorems  \ref{TH1} and \ref{TH2}, our second main objective is to show that these techniques, when refined, allow also to  prove the Universal Approximation Theorem (Section \ref{SUA}). More precisely,  in Theorem \ref{THsimple} we prove that, given a simple function in $f: \Omega \to \mathbb{R}^d$ (a linear finite combination of characteristic functions of measurable sets), under the assumption that characteristic sets have finite perimeter, for any    bounded set $\Omega\subset\mathbb{R}^d$ and $\epsilon >0$ arbitrary being given, we can find controls $A(t), W(t)$ and $b(t)$ (depending in particular on $\epsilon$) such that the associated flow $\phi_T$ of \eqref{CTNNo} satisfies:
$$ \|f-\phi_T(\cdot;A,W,b)\|_{L^2(\Omega)}<\epsilon. $$ The proof, which is of constructive nature, yields a connection between the complexity of the function to be approximated and the cost of control in terms of the box-counting dimension (see \cite[Chapter 2]{falconer2004fractal}).   The control complexity is measured in terms of its $L^\infty$-norm and the number of switches and it depends on, from one side, the level of approximation $\epsilon$ we desire, and from the other, on the complexity of the target function. In fact,  any simple function can be approximated and its control cost approximated provided that the box-counting dimension of the boundaries of its characteristic sets is smaller than the dimension of the ambient space.

Our third main objective is to formulate these results in terms of control of transport equations, where the dynamical control perspective is made more evident. In Section \ref{Transport} (see Theorems \ref{TH5} and \ref{TH6}), we prove that the Neural Transport equation 
  \begin{equation}\label{neuraltransport}
  \partial_t\rho+\mathrm{div}_x\big[\left(W(t)\boldsymbol{\sigma}(A(t)x+b(t))\right)\rho\big]=0
\end{equation}
is approximately controllable. Furthermore, we will see that one can achieve a simultaneous control result for a Neural Transport system of equations of the form \eqref{neuraltransport} under certain conditions on the initial data and targets (roughly, having disjoint supports), which mimic the configuration of classification problems.

\subsection{Roadmap and related work.}
This paper is organized as follows. In Section \ref{Smatset}, we introduce the notation and formulate the main problems more precisely, presenting some of the fundamental tools that will systematically be used in our proofs. Section \ref{SCLASS} is devoted to the classification problem and in Section \ref{SSCONTROL} we reinterpret and further develop it in the simultaneous control context.  Section \ref{SUA} is devoted to proving the Universal Approximation Theorem, while Section \ref{Transport} to analyze neural transport equations. Section \ref{Sconcl} is devoted to present some open problems and further directions of research.

For the sake of simplicity of the exposition most of our results are presented in the particular case of the ReLU activation function, although the methods we develop apply for larger classes of activation functions that will be made  clear below.

 We conclude this introduction with a discussion of some of the related literature.
 
 Control issues for NODEs have been addressed from different  perspectives.  In \cite{sontag1997complete}, for instance, the authors analyze the classical control problem of NODEs. Note however that this is done in the context in which one typically considers one single trajectory and not the simultaneous control problems that classification requires.
 
 %\textcolor{red}{
 In the context of simultaneous control of Neural ODEs, in \cite{cuchiero2020deep} a Lie bracket technique is developed for certain smooth activation functions that allow to achieve universal interpolation results, similar to those  in Section \ref{SSCONTROL}. We also refer to \cite{agrachev2020control} for further developments on the geometric Lie bracket interpretation of these concepts in the context of classification, also using smooth vector fields. Note however that our approach can be implement for Lipschitz continuous activation functions as the ReLU as in Figure \ref{fig:RELU}.
 %}
 
 In \cite{esteve2020large}, the authors show a simultaneous control result for sufficiently regular activations under smallness conditions on the targets, and employing linearization arguments, to later derive Turnpike results for NODEs. Our techniques, of constructive nature, are of independent interest and pave the path to analyze Turnpike properties of NODEs in a much broader context.

 Recently, in \cite{li2019deep}, universal approximation results were proved by means of NODEs, see also \cite{teshima2020universal,tabuada2020universal}. In  Section \ref{SUA} we provide a simpler and more comprehensive proof of those results which allows, in particular, to handle the complexity of the controls in terms of the data to be classified or the target functions.%}

 The results in Section \ref{Transport}, which provide a  transport formulation of Deep Learning, making use of the classical link between ODEs and transport equations, can be seen as  bilinear-type simultaneous control results for Neural transport equations and systems (Theorems \ref{TH5} and \ref{TH6}). The bilinear  control of transport equations has also been considered (see, for instance, \cite{duprez2019approximate,duprez2020minimal}) with controls localized in space motivated by collective behavior problems \cite{control-kineticCS}. The nature of the controls we employ here and the dynamics that the systems under consideration, exploiting the very features of activation functions, are completely different. Furthermore, the use of Neural transport equations allows a probabilistic formulation of the classification problem.

 %\textcolor{red}{
 Our results for Neural Transport equations can also be related to optimal transport, that can be formulated as a minimal kinetic energy controllability problem as in \cite{benamou2000computational}, \cite{benamou1998optimal} and the references therein.% }
 
%\medskip
\subsection{Notation.} Throughout this article we will use the following notation. 
\begin{itemize}
 \item[-]We understand by $E\lesssim_{q}F$  the existence of a constant $C(q)$ depending on $q$ such that $E\leq C(q) F$. Analogously, we denote by $E\sim_q F$ the existence of $C(q)$ such that $E=C(q)F$.
 \item[-]Let $x\in \mathbb{R}^d$, we denote by  $x^{(k)}$ the $k$-th component of the vector $x$.
 \item[-]Let $\omega\subset\mathbb{R}^d$ we denote by 
$\omega^{(k)}$ the following set $$\omega^{(k)}:=\{x\in\mathbb{R}: \exists y\in\omega\quad \text{ such that } y^{(k)}=x\}.$$
\item[-] Given $\omega\subset\mathbb{R}^d$, we denote by $\mathrm{diam}$ the quantity $$\mathrm{diam}(\omega):=\sup_{x_1,x_2\in\omega} |x_1-x_2|$$ and by $\mathrm{diam}_{(k)}$ the quantity $$\mathrm{diam}_{(k)}(\omega):=\sup_{x_1,x_2\in\omega} |x_1^{(k)}-x_2^{(k)}|.$$
\end{itemize}

\section{Mathematical setting}\label{Smatset}
\subsection{Setting}

Consider a Lipschitz activation function $\sigma:\mathbb{R}\to\mathbb{R}$. For the clarity of the exposition, most of the presentation will be developed in the context of the particular activation function
\begin{equation}\label{relu}\sigma(x)=\max\{x,0\},\end{equation}
the so called ReLU displayed in Figure \ref{fig:RELU}. Later, in  Remark \ref{remact}, we shall discuss the extension to more general activation functions (roughly to functions which are Lipschitz, vanishing in $(-\infty,0]$ and positive in $(0,+\infty)$.)
\begin{figure}
 \includegraphics[scale=0.5]{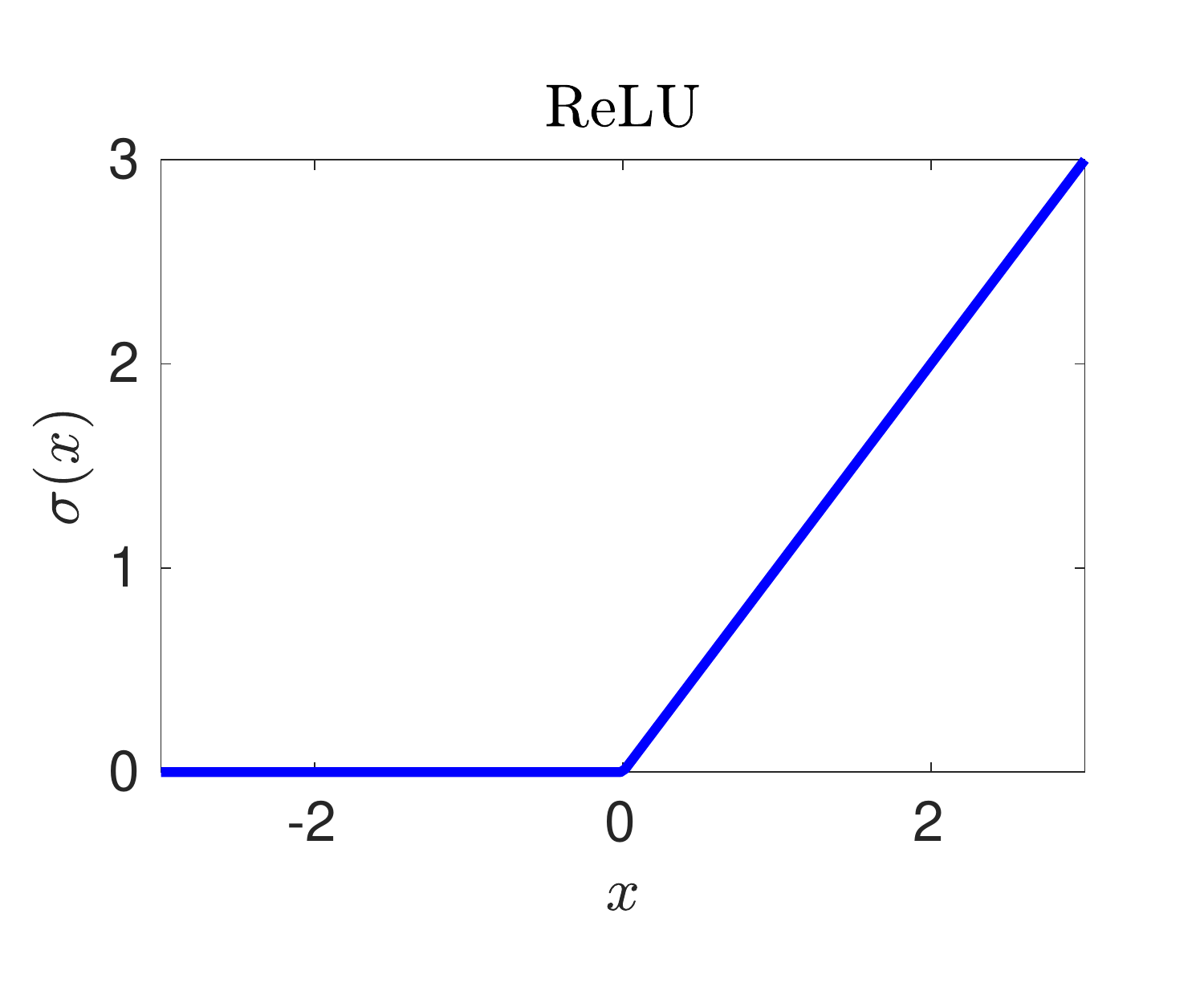}
 \caption{The Rectified Linear Unit (ReLU) activation function \eqref{relu}}\label{fig:RELU}
\end{figure}
We shall also employ the vector-valued version of the activation function, denoted by $\boldsymbol{\sigma}$, built applying $\sigma$ to every component of a vector:
\begin{equation}\label{struc}
 \boldsymbol{\sigma}:\mathbb{R}^d\to\mathbb{R}^d,\qquad
\boldsymbol{\sigma}(x)=\begin{pmatrix}
                        \sigma(x^{(1)})\\
                        \sigma(x^{(2)})\\
                        \vdots\\
                        \sigma(x^{(d)})
                       \end{pmatrix}.
\end{equation}
%Here and in what follows we  denote  by $x=(x^{(1)},x^{(2)},...,x^{(k)},...,x^{(d)})$ the components of the vector $x$. 
Consider now the Neural ODE% defined by $A,w$ and $b$:
\begin{equation}\label{CTNN}
 \begin{cases}
  \dot{x}=W(t)\boldsymbol{\sigma}(A(t)x+b(t))\\
  x(0)=x_0.
 \end{cases}
\end{equation}
The coefficients $A,W$ and $b$ play the role of controls, $A,W\in L^\infty((0,T),\mathbb{R}^{d\times d})$, $b\in L^\infty((0,T),\mathbb{R}^{d})$. With the notations above  \eqref{CTNN} is a system of $d$ coupled differential equations.

The first question we address is the classification problem, which, in the context of the NODEs \eqref{CTNN},  can be reformulated as follows. We fix a time horizon $[0, T]$. Consider $N$ distinct data $\{x_i\}_{i=1}^N\subset \mathbb{R}^d$ to be classified by their labels $\{y_i\}_{i=1}^N\subset \mathbb{R}^d$ using \eqref{CTNN}. %:  $\{(x_i,y_i)\}_{i=1}^N\subset \mathbb{R}^d\times \mathbb{R}^d$.
In fact $\{x_i\}$, the data to be classified,  will be taken to be the initial data at time $t=0$  for the Cauchy problem \eqref{CTNN}, while $\{y_i\}$ are the corresponding labels that will allow us to define the targets at the final time $t=T$. 

To reformulate this problem we consider a partition of $\mathbb{R}^d$ into $M$ (the number of distinct labels) disjoint open sets with nonempty interior.  We allocate one of these sets to each label. From a control theoretical  perspective the classification  problem can be reformulated as follows: to find controls $W, A, b$ bringing each initial data to the corresponding set along the trajectories of \eqref{CTNN}.

Note that this problem can be considered as a simultaneous or ensemble control one, in the sense that the same controls are aimed to control all trajectories.

\medskip
\noindent \textbf{The Partition of the Euclidean Space.} Our techniques can be applied for an arbitrary partition of the euclidean space. But the complexity of the control maps will depend on the structure, geometry and topology  of the partition. To simplify the construction of the controls and our presentation we shall consider a particular partition of the space into parallel strips that we introduce now.

Let $M$ be the number of different  labels $y_i$, i.e.
$$M=\left|\{y_i:\quad i\in\{1,...,N\}\}\right|. $$
We split the space $\mathbb{R}^d$ into $M$ disjoint sets $\mathcal{S}:=\{S_m\}_{m=1}^M$, each set  associated with one label. For simplicity of the presentation we set the partition into strips

\begin{equation}\label{strips}
 S_m:=\{ \alpha_{m-1}<x^{(1)}\leq \alpha_{m} \}\qquad 1\leq m\leq M
\end{equation}
with $\alpha_{0}=-\infty$, $\alpha_{M}=+\infty$ and $\alpha_{m}<\alpha_{m+1}$.

\begin{figure}
 \includegraphics[scale=0.40]{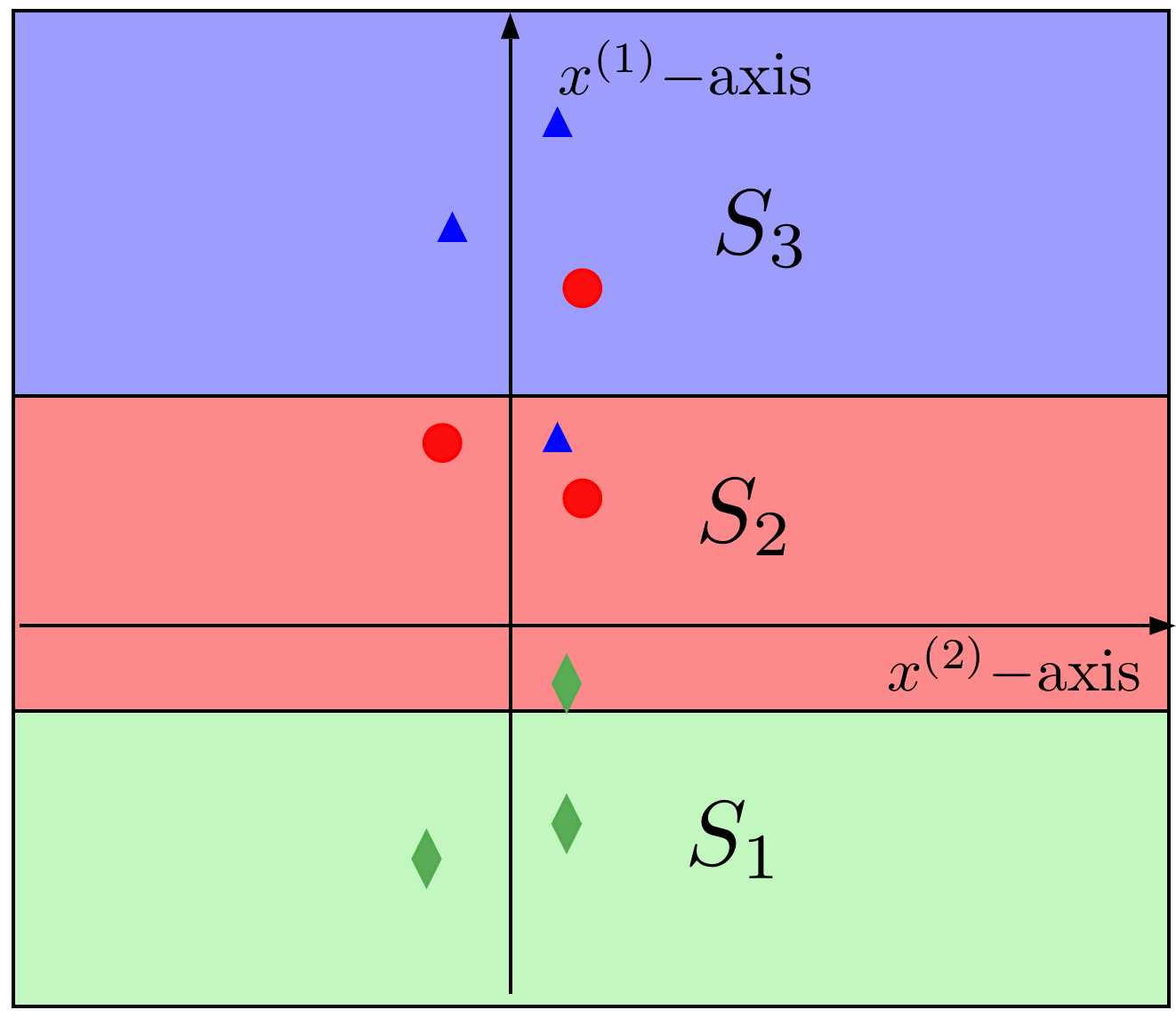}
 \caption{Partition of the Euclidean plane ($d=2$) into three parallel strips ($M=3$).  Nine points to be classified  are denoted by three different colors and shapes corresponding to their labels (black triangles, red circles and green diamonds) and they are distributed in an unclassified manner. The goal of the approximation process, that we aim to conduct by controlling the flow of the Neural ODE, is to drive each subclass of three points to the corresponding label, i.e. to the strip corresponding to its color.}
\end{figure}

\medskip
\noindent \textbf{The Control Process.} Given controls $W,A,b$ we denote the solution of \eqref{CTNN} by $\phi_T(x;A,W,b)$.
Given initial data $\{x_i\}_{i=1}^N$ with labels $\{y_i\}_{i=1}^N$, the goal is to find controls $W,A,b$ such that
$$\phi_T(x_i;A,W,b)\in \mathcal{S}_{m(i)}\qquad 1\leq i\leq N$$
where  $m(i)\in \{1,..., M\}$ denotes the index of the label corresponding to the input $x_i$.

\medskip
\noindent \textbf{The Prediction.} Once this map $\phi_T$ has been built out of the corresponding controls $A, W, b$, it can be used for prediction purposes, assigning to each new datum $x$ the corresponding label $y_m$ associated with the subset $S_m$ of the partition satisfying $\phi_T(x;A,W,b)\in S_m.$

\subsection{The fundamental operations: Place, freeze, compress/expand and translate}\label{funop}

All our results will be proved by the same methodology. It consists of building piece-wise constant controls which, concatenated, will assure that the NODE fulfills the needed requirements. Each of these constant actions exploits the activation function's properties, the ReLU, in this presentation. In each step the available control parameters $W, A$, and $b$ are chosen so that the ODE induces a transformation of the Euclidean space $\mathbb{R}^{d}$ which leaves invariant half of the space while  the other evolves towards the desired direction. Obviously, the choice of the correct  controls is essential to assure that  the invariant and the moving half-spaces are the appropriate ones, together with the direction of motion. %Note that we are not exploiting the full capacity of the system

\smallskip
\noindent For simplicity of the proofs we will also choose separating hyperplanes parallel to the Cartesian axes. The algorithm could be optimized by choosing oblique hyperplanes adapted to the datasets to be classified (in the context of the classical problem) or the function to be approximated (in the context of universal approximation). But, for simplicity of the presentation, we will only consider hyperplanes parallel to the Cartesian axes. 
\begin{enumerate}
 \item Consider a hyperplane of the form:
 $$\left\{x^{(k)}-c=0\right\}.$$
 This hyperplane can be represented as $Ax+b=0$ with the choice of the controls $A, b$ as below
$$A_{ij}=\pm \delta_{ik}\delta_{jk},\quad b_i=\mp c\delta_{ik}.$$
The hyperplane divides the Euclidean space in two half-spaces.
 Here and in the sequel $\delta_{jk}$ stands for the Kronecker delta.

\smallskip
 \item When  applying the vector-valued activation function  $\boldsymbol{\sigma}$ to $Ax+b$, taking into account that the ReLU activation function $\sigma$ vanishes for negative inputs, the choice of the sign of $A$ and $b$ is crucial to determine which half-space remains invariant;
 $$ \sigma(Ax+b)^{(i)}=\delta_{ik}\max\{\pm(x^{(k)}-c),0\}.$$
 See Figure \ref{Funstep}.A.
 
\smallskip
 \item Choosing $W$ to be a rotation matrix, when applied to $\boldsymbol{\sigma}(Ax+b)$, we can change the direction of the vector field in the direction we desire, see Figure \ref{Funstep}.B. Note that, in particular, we can choose $W$ so that the hyperplane is attractive or repulsive with respect to the moving half-space, but also we can generate a  movement parallel  to the hyperplane. 
 
  Throughout the proofs, we will only consider controls $W$ such that the resulting vector fields point in some of  the canonical euclidean directions. As it occurs for the choice of the hyperplanes, proofs could be implemented with a lower number of concatenations if $W$ were chosen in oblique directions, depending on the dataset. But, for simplicity of the exposition, we shall only employ these elementary $W$'s.

\end{enumerate}

Summarizing, we can:
\begin{enumerate}
 \item \textbf{Place} and orient the dividing hyperplane arbitrarily.
 
\smallskip
 \item  Choose the invariant and the active half-spaces. In particular we may choose the half-space that we aim to \textbf{freeze}.
 
\smallskip
 \item Choose the orientation of  the vector-field $\boldsymbol{\sigma}(Ax+b)$ so that either
 \begin{itemize}
 \item[-] it induces an \textbf{expansion or contraction} of the active half-space.
 
 or
 
 \item[-] it induces a \textbf{translation} on the active half-space parallel to the dividing hyperplane.
 \end{itemize}
\end{enumerate}

As we shall see, the concatenation of these elementary transformations will suffice to achieve our classification and  approximation goals.

\begin{figure}%[h!]
\hspace{-2.5cm}\begin{subfigure}[b]{0.1\textwidth}
 \includegraphics[scale=0.4]{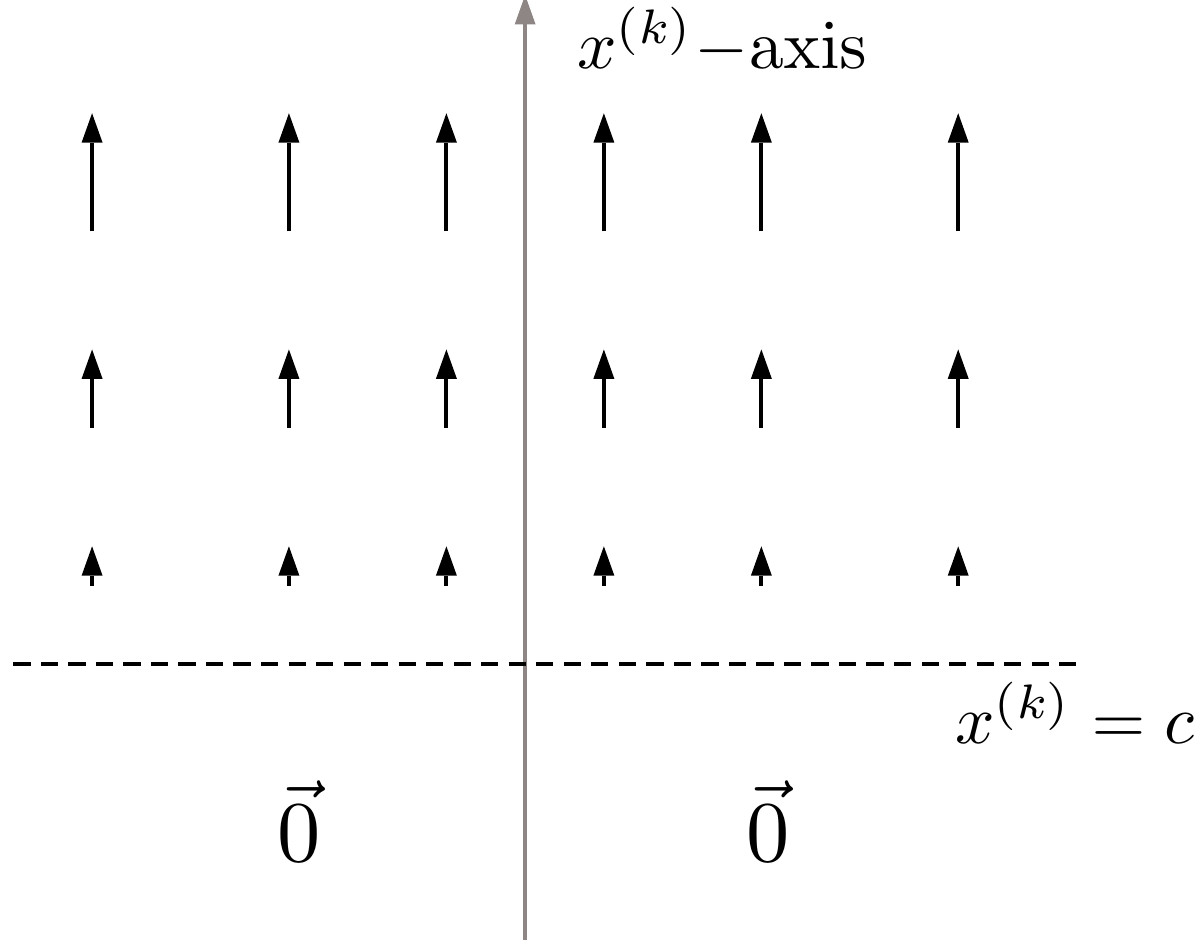}
\caption{}
\end{subfigure}
\hspace{5cm}
\begin{subfigure}[b]{0.1\textwidth}
 \includegraphics[scale=0.4]{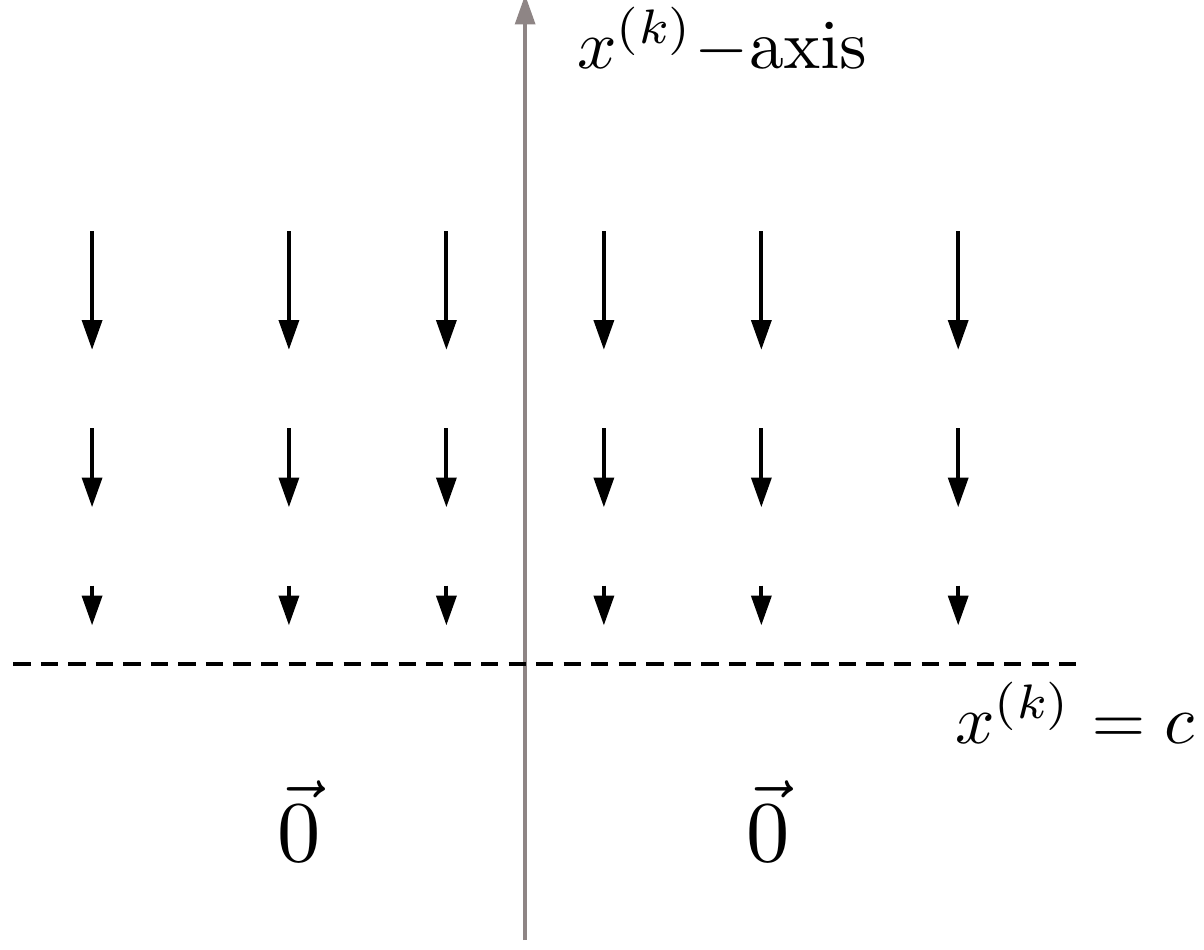}
\caption{}
\end{subfigure}

\hspace{-2.5cm}\begin{subfigure}[b]{0.1\textwidth}
 \includegraphics[scale=0.4]{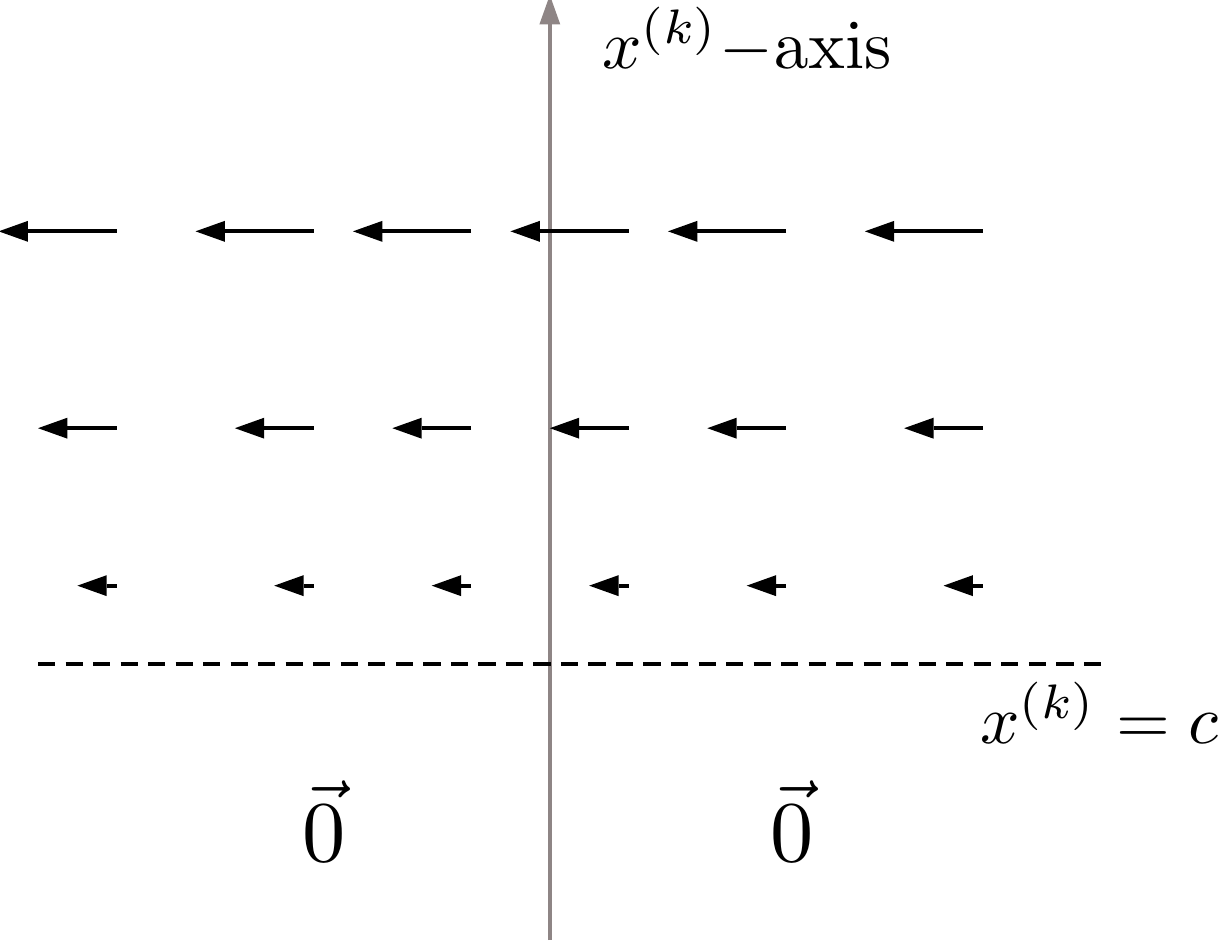}
\caption{}
\end{subfigure}
\hspace{5cm}
\begin{subfigure}[b]{0.1\textwidth}
 \includegraphics[scale=0.4]{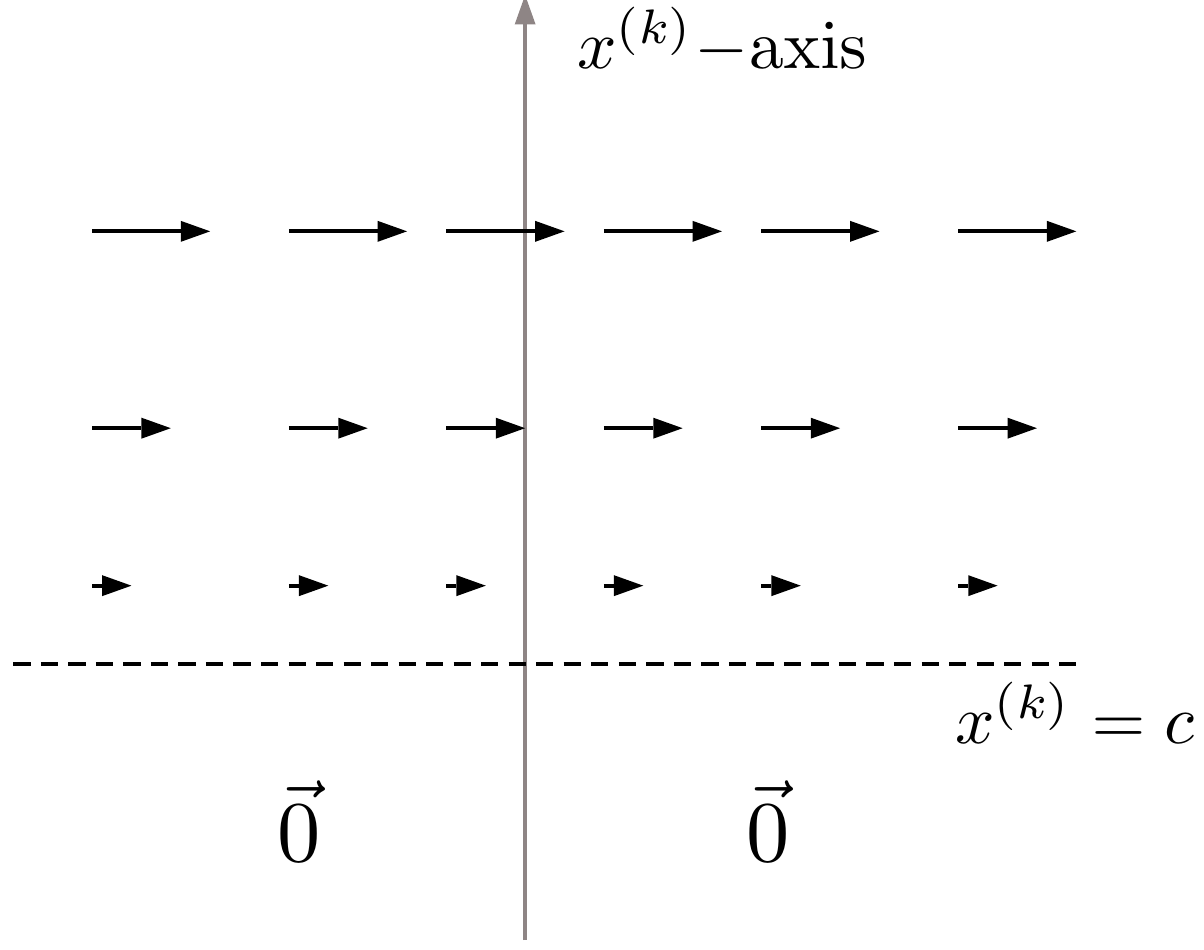}
\caption{}
\end{subfigure}

 \caption{Representation of the elementary vector-fields achieved with $\boldsymbol{\sigma}(Ax+b)$ as in in Subsection \ref{funop}. While the lower half-space remains invariant, the other one experiences a transformation of the form: (A) Expansion; (B) Contraction;  (C) and (D) correspond to translations in opposite directions.}\label{Funstep}
\end{figure}

\section{Classification}\label{SCLASS}
This section is devoted to present Theorem \ref{TH1} and its proof, which allows to classify an arbitrary finite  dataset in strips, according  to their label. The proof is based on combining the essential movements explained in subsection \ref{funop}. For simplicity we state and prove the Theorem for the ReLU \eqref{relu} activation function. However our arguments apply to more general activation functions $\sigma$, as we will make precise later in Remark \ref{remact}.

\begin{theorem}\label{TH1}
 Let $d\geq 2$ and $M\geq 2$ be natural numbers and let $\sigma$ be as in \eqref{relu}.
Let $\{x_i,y_i\}_{i=1}^N\subset \mathbb{R}^d\times\mathbb{R}^d $ be the dataset to be classified. Assume that $x_i\neq x_j$ if $i\neq j$.
 Then, for every $T>0$, there exist control functions $A,W\in L^\infty \left((0,T); \mathbb{R}^{d\times d}\right)$ and $b\in L^\infty\left((0,T),\mathbb{R}^d\right)$ such that the flow associated to \eqref{CTNN}, when applied to all initial data $\{x_i\}_{i=1}^N$, classifies the trajectories by strips \eqref{strips}, according to their labels $\{y_i\}_{i=1}^N$, i.e. $$\phi_T(x_i;A,W,b)\in S_{m_i},$$ $S_{m_i}$ being the subset corresponding to the  label $y_i$, with $i\in\{1,..., M\}$.
 Furthermore, controls are piecewise constant with a finite number of switches of the order of $\mathcal{O}(N)$. Therefore, they also lie in $BV$.
 
\end{theorem}

\begin{remark}[$d=1$]\label{d=1}
Theorem \ref{TH1} requires that $d\geq 2$. If $d=1$ and $x_1<x_2$, then for all choices of the controls $W, A ,b$, the associated flow will fulfill that $\phi_T(x_1;A,W,b)<\phi_T(x_2;A,W,b)$ for all $T>0$. This  is due to the ODE formulation and the uniqueness of solutions for the corresponding Cauchy problem.
Note that this restriction does not necessarily limit time-discrete NNs as, for instance, the one in  \eqref{Euler}, since it does not necessarily fulfill this monotonicity property when the step-size is large.
Our proofs apply in an ODE setting and therefore are limited to $d\geq 2$. As a corollary, they apply as well for the time-discrete ResNets \eqref{Euler}  provided the step-size $h$ is small enough.

\end{remark}

\begin{proof}
As we mentioned above, the actual value of the final time $T>0$ is irrelevant since, by scaling, it can be set to be, in particular, $T=1$.

Therefore we present a strategy that in a finite number of steps allows to achieve the classification in a finite time $T^*>0$. The scheme of the proof is the following:
\begin{enumerate}[font=\bfseries]
 \item[(1)]  \textbf{Preparation of  the dataset}. The goal of this first step is to assure that the data to be classified fulfill the condition:
 \begin{equation}
 \label{separation}
 \forall i\neq j,\quad x_i^{(1)}\neq x_j^{(1)}.%\quad \forall k\in\{1,...,d\}.
 \end{equation}
 Obviously, for that to hold, one needs to exploit the dynamics of the system by properly choosing the controls. 
  This condition is relevant to assure that data can be efficiently separated by hyperplanes parallel to the Cartesian axes.
 
\smallskip
 \item[(2)] \textbf{Classification}  of the points so that the corresponding trajectories reach the allocated  strips.
 
\smallskip
 \item[(3)]  \textbf{Time rescaling} allows to assure that the goal is achieved in the given final time $T>0$.
 
\end{enumerate}

\smallskip

We now present in more detail the process of each of these steps.
\begin{enumerate}[font=\bfseries]

\item\textbf{Preparation of the dataset.} We need the  data to be classified to fulfill the condition \eqref{separation}.
 If the data $\{x_i\}_{i=1}^N$ do not fulfill this property, by a proper choice of the controls $A, W, b$ in a time $\tau >0$, we may assure that \eqref{separation} is satisfied by $\phi_\tau (x_i; A,W,b),$ $i\in\{1,...,N\}$. 

\smallskip
 \noindent The construction of the controls guaranteeing that this separation property is fulfilled can be made by induction. 
\begin{itemize}
 \item Suppose we have only two points, $x_i$ and $x_j$, sharing  the first component $x^{(1)}$. Since $x_i \ne x_j$, there is a coordinate $k$ for which $x_i^{(k)}\neq x_j^{(k)}$. We then take the mean value  $r=(x_i^{(k)}+x_j^{(k)})/2$ and consider the matrix $A$ and the vector $b$ defined by
 $$A_{ij}=\delta_{jk}\delta_{ik}\quad i, j\in\{1,...,d\}, \hspace{1cm} b_j=r\delta_{jk}\quad j\in\{1,...,d\}.$$
 The hyperplane $\{Ax+b=0\}$ separates the two points $x_i$ and $x_j$.
We choose $W$ such that the resulting vector field $\boldsymbol{\sigma}(Ax+b)$ can change the $x^{(1)}$-component of one of the elements. For doing so, we set 
$$W=\begin{pmatrix}
   1&1&...&1\\
   0&0&...&0\\
   \vdots&\vdots& &\vdots\\
      0&0&...&0                                                                                                                                                                      \end{pmatrix}
$$ acting as a translation in the half-space. Solving the NODE in a sufficiently small time $\tau>0$, we can assure that 
\begin{equation}
\label{separationafter}
\phi_\tau (x_j;A,W,b)^{(1)}\neq \phi_\tau(x_i;A,W,b)^{(1)},\quad \forall i\neq j
\end{equation}
as shown in Figure \ref{fig:prep}.qquad

Note that $\tau>0$ needs to be taken small enough so to guarantee that when assuring that \eqref{separationafter} holds, the other points do not lose the distinction of all their coordinates. 
\begin{figure}

\begin{subfigure}[b]{0.1\textwidth}
 \includegraphics[scale=0.4]{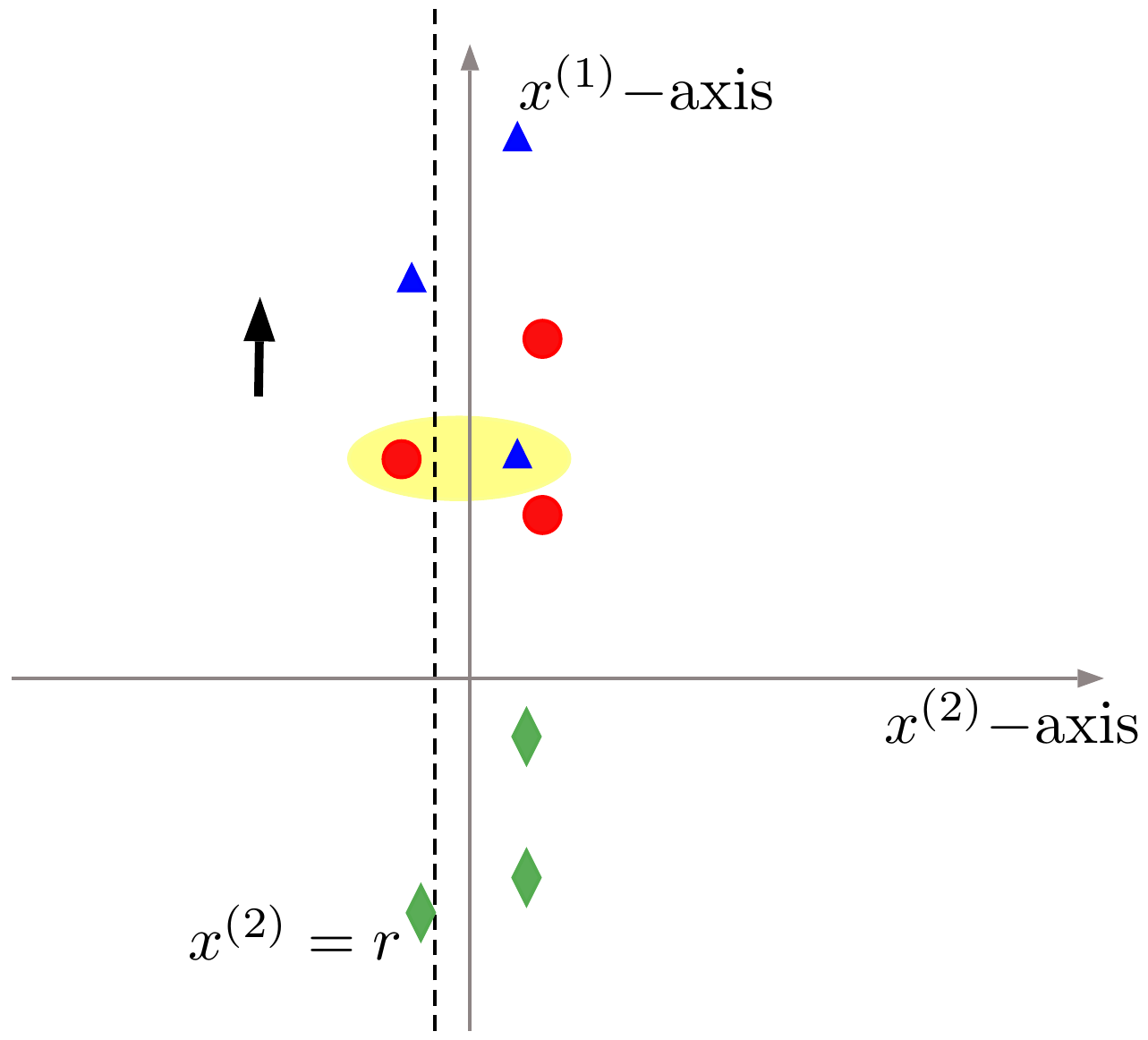}
 \caption{}
\end{subfigure}
\hspace{5cm}
\begin{subfigure}[b]{0.1\textwidth}
 \includegraphics[scale=0.4]{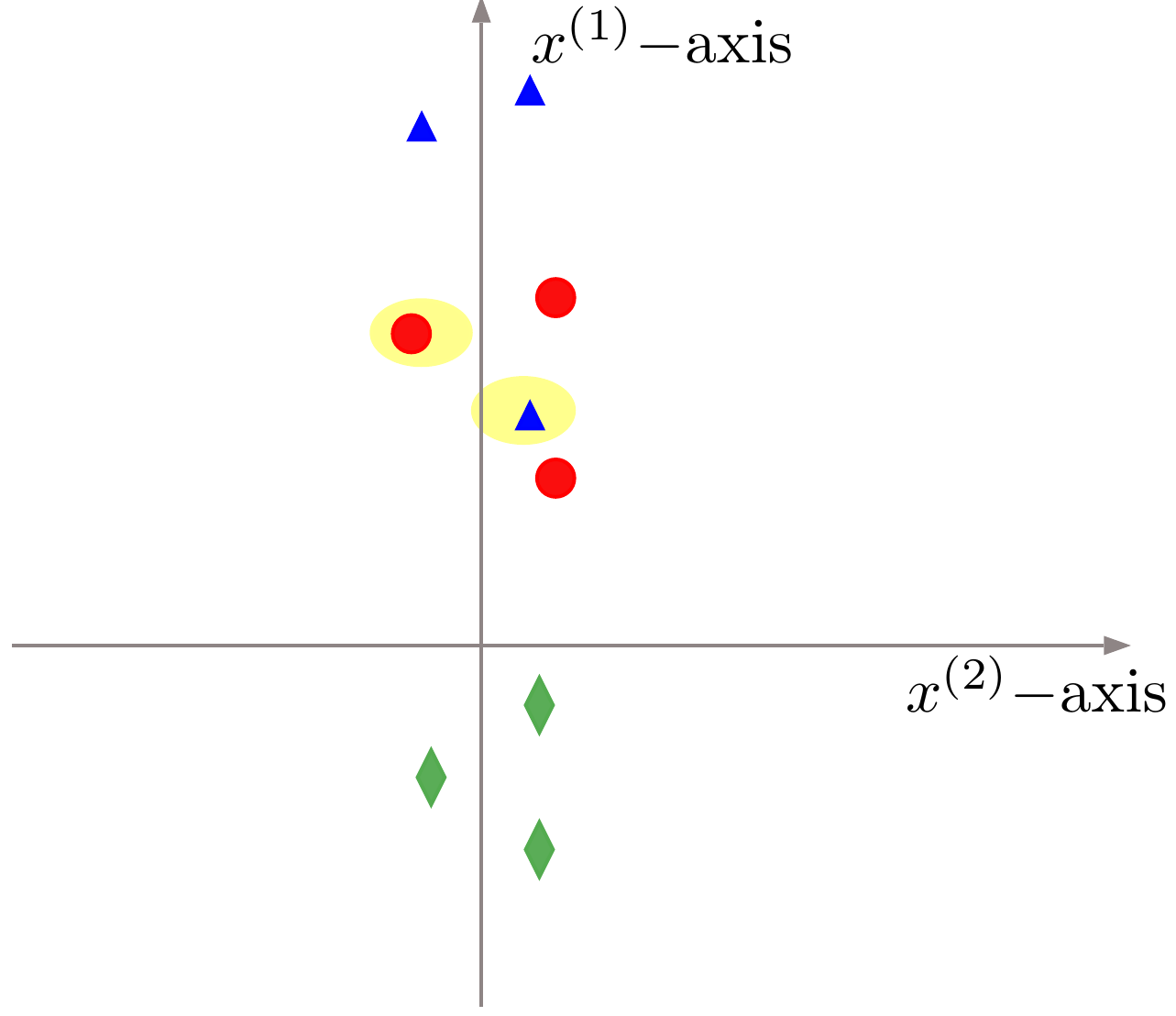}
  \caption{}
\end{subfigure}
\hspace{5cm}

 \caption{(A) Initial configuration in which two points within the yellow cloud share the same component $x^{(1)}$. The hyperplane parallel to the $x^{(1)}$-axis is chosen separating the two points. The vector field induces a vertical upwards translation of the left half-space. (B) Represents the final distribution of points. The yellow clouds indicate the points that initially shared the same $x^{(1)}$-component, a fact that is avoided after the application of the dynamic deformation.}\label{fig:prep}
 %\textcolor{red}{No me parece que ayude el que el eje $x^(1)$ de la figura sea el vertical. No habria sido mejor hacerlo horizontal? Tambien habria ayudado que en en la figura B se pusiera una nubecita amarilla oblicua, como en A, para indicar los dos puntos que se han conseguido mover.}
\end{figure}

\item This argument can be applied recursively on all pairs of data sharing the component $x^{(1)}$.
\end{itemize}

 \noindent Once the separation property is guaranteed we can reset $t=0$ to proceed to classification.
 This first step shows that, without loss of generality, we can assume that the dataset to be classified fulfills the separation property \eqref{separation}.

\smallskip
\item\textbf{Classification.}
        By the previous step we have a set of points that, in particular,  do not share the $x^{(1)}$-component. 
        
\smallskip
 \noindent Recall also that to each initial datum $x_i$, $i\in\{1,...,N\}$, it corresponds the strip $S_{m(i)}$ as a target, associated to its label $m(i)\in \{1,...,M\}$.
 Whenever a datum $x_j$ lies in the set corresponding to its target, i.e. $x_j\in S_{m(j)}$, no action is needed, which corresponds to simply taking $W=0$, i.e. the trivial dynamics. But for each $j\in\{1,...,N\}$ such that $x_j\notin S_{m(j)}$ one needs to force the allocation of the points to the corresponding strip by a suitable choice of the controls.

\smallskip
 \noindent We do it iteratively as follows. We start with any point $x_j\notin S_{m(j)}$. To simplify the notation, let us assume that this is the point $x_1$ corresponding to the index $j=1$. 
Then:
     \begin{enumerate}
      \item we choose the hyperplane $\{Ax+b=0\}$ with $A$ and $b$ of the form
         \begin{equation}\label{tria1}
     A=\begin{pmatrix}
      1 & 0 & ... & 0\\
      0 & 0 & ... & 0\\
      \vdots &\vdots & &\vdots\\
      0 &0 &... & 0
     \end{pmatrix},\quad b=-\begin{pmatrix}
     x_1^{(1)}+r\\
     0\\
     \vdots\\
     0
     \end{pmatrix}
    \end{equation}
where $r>0$ is such that $x_1^{(1)}+r< x_i^{(1)}$ for every $i$ such that $x_1^{(1)} < x_i^{(1)}$.
   % \textcolor{red}{
   We choose $W$ so that the field is parallel to the hyperplane, for simplicity:
         \begin{equation}\label{tria2}
      W=\begin{pmatrix}0 & 0 & \dots &0\\
                       -1& 0 & \dots &0\\
                       0& 0& \dots &0\\
                       \vdots & \vdots& &\vdots\\
                       0& 0& \dots & 0\end{pmatrix}%\|W\boldsymbol{\sigma}(Ax+b) \| 
     \end{equation}
     %\begin{equation}\label{tria2}
      %\langle (0,-1,0,...,0),W\boldsymbol{\sigma}(Ax+b) \rangle=\|W\boldsymbol{\sigma}(Ax+b) \| 
     %\end{equation}
     as in Figure \ref{fig:Induction}.  
          \begin{figure}%[h!]
                    \begin{subfigure}[b]{0.1\textwidth}
                 \includegraphics[scale=0.4]{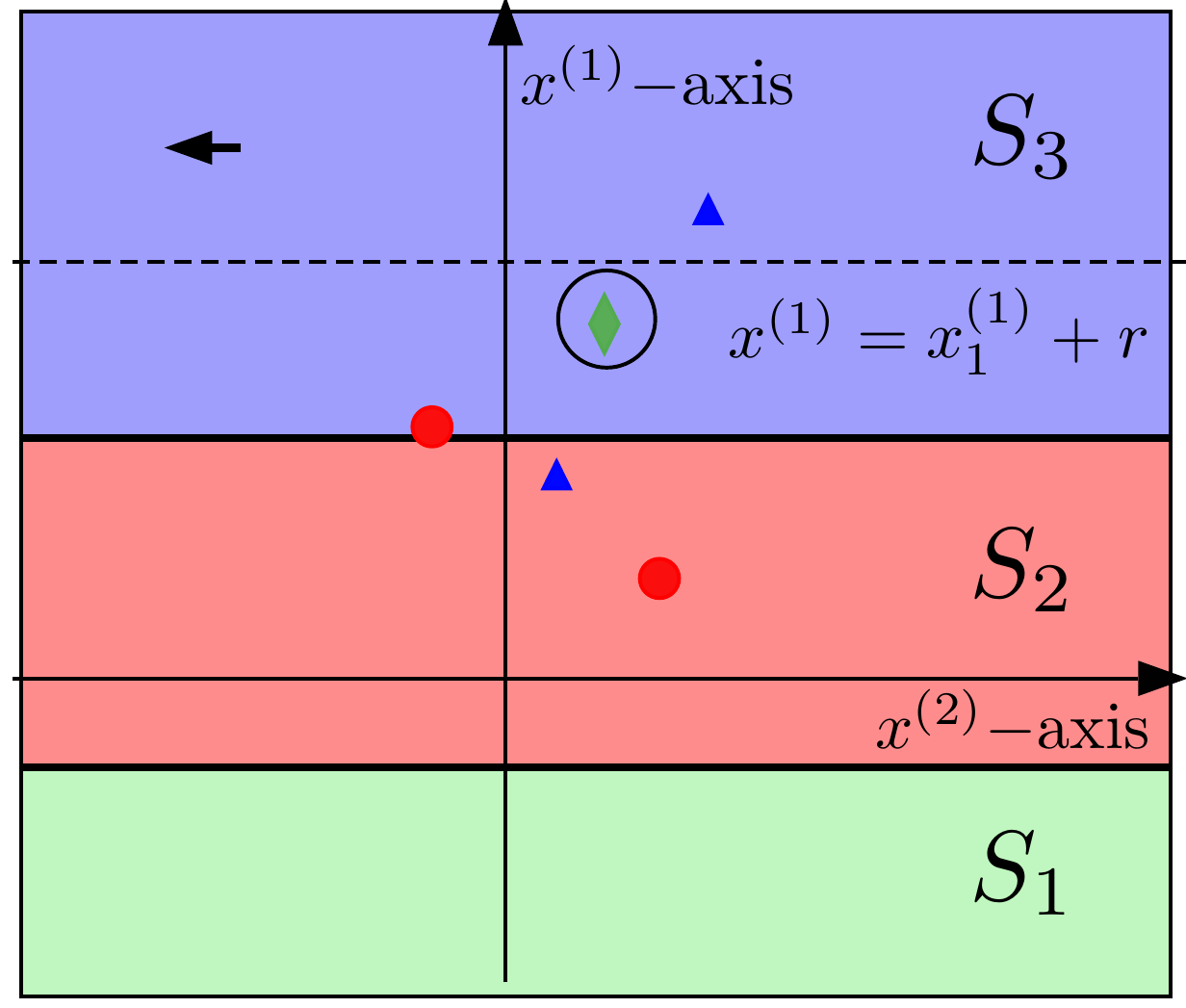}
                 \caption{ }
          \end{subfigure}\hspace{4cm}     
                    \begin{subfigure}[b]{0.1\textwidth}
                 \includegraphics[scale=0.4]{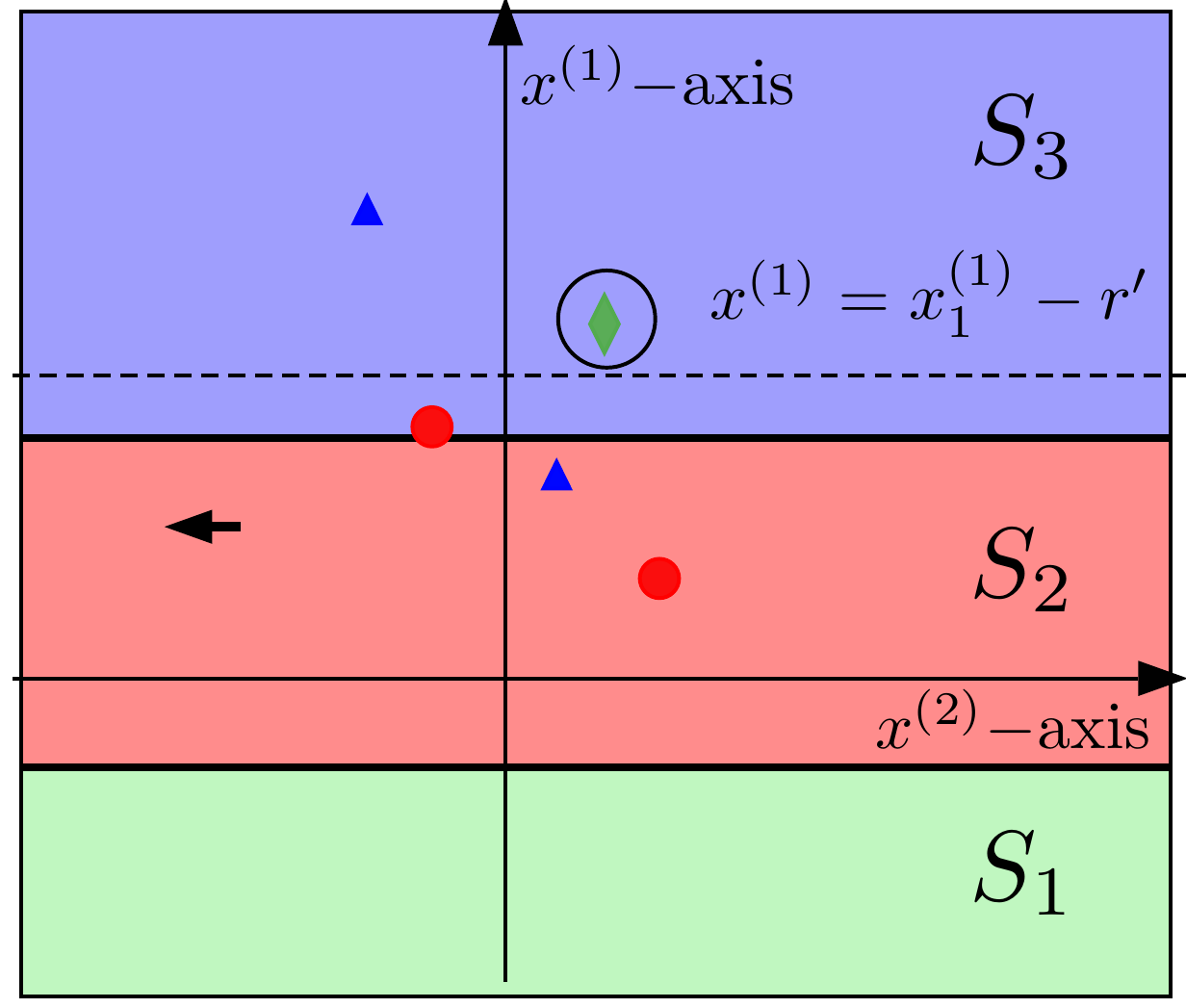}
                     \caption{ }
          \end{subfigure}
          \hspace{2cm}

                    \begin{subfigure}[b]{0.1\textwidth}
                 \includegraphics[scale=0.4]{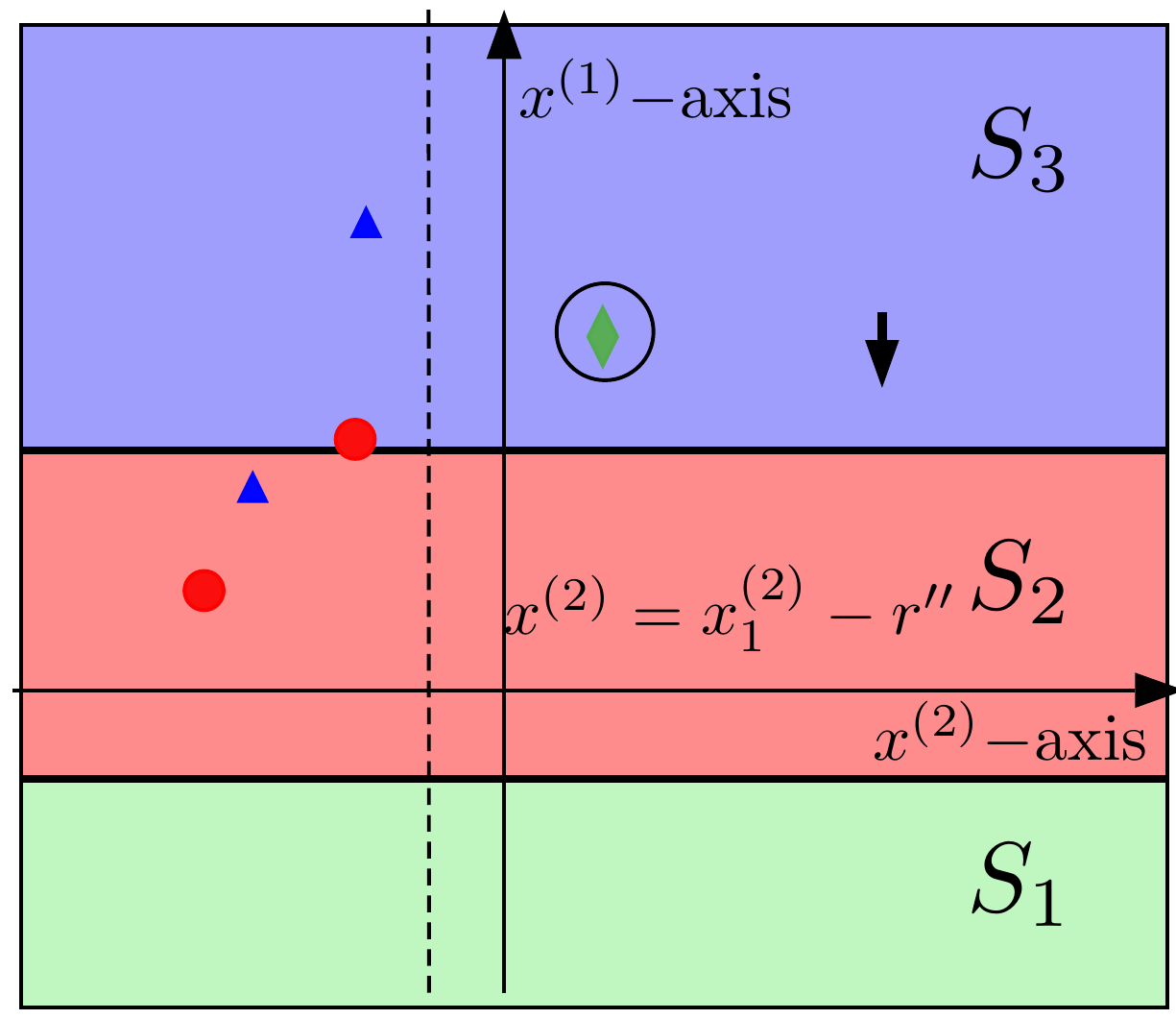}
                \caption{ }
          \end{subfigure}\hspace{4cm} 
                    \begin{subfigure}[b]{0.1\textwidth}
                 \includegraphics[scale=0.4]{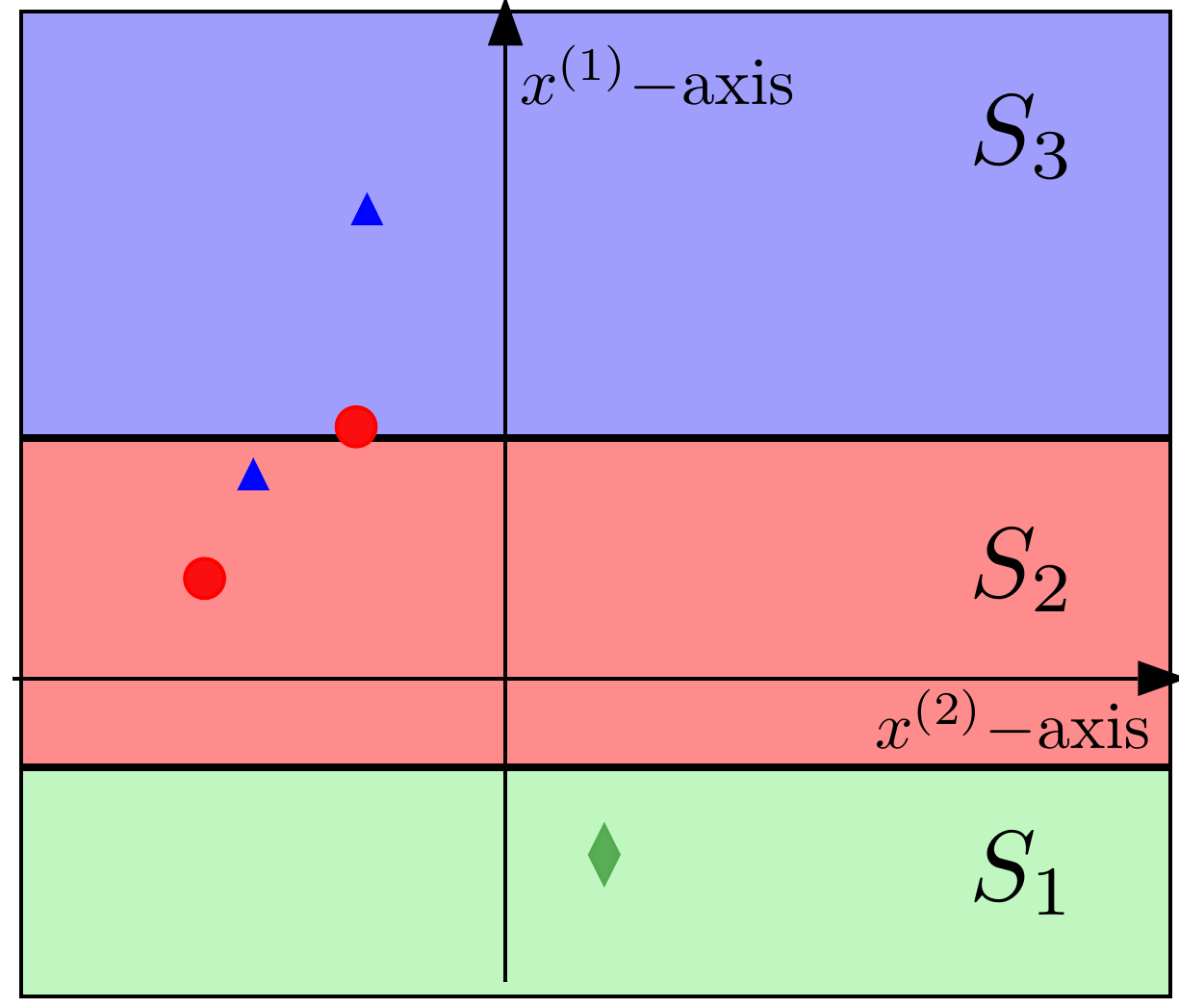}
                 \caption{}
          \end{subfigure}
\hspace{2cm}
       \caption{ (A), (B) and (C) represent the three main steps (a), (b) and (c) and the concatenation  leading to the final configuration of the dataset in (D). In this manner the green diamond that was originally located in the strip $S_3$ ends up lying in the one corresponding to its label (green color) $S_1$. }\label{fig:Induction}
     \end{figure}
     We solve \eqref{CTNN} up to a time $T_1$  so that 
     \begin{equation}\label{criteriaT1}
     \phi_{T_1}(x_i,A,W,b)^{(2)}<x_1^{(2)},\quad\forall i\text{ s.t. } x_1^{(1)}<x_i^{(1)}.
     \end{equation}
     %for any $i$ such that $x_N^{(1)}< x_i^{(1)}$.
        \item  Then, we set $r'$ to be such that $x_1^{(1)}-r'> x_i^{(1)}$ for any $i$ such that $x_1^{(1)}> x_i^{(1)}$ and we choose our controls $A'$ and $b'$ as 
         $$ A'=\begin{pmatrix}
      -1 & 0 & ... & 0\\
      0 & 0 & ... & 0\\
      \vdots &\vdots & &\vdots\\
      0 &0 &... & 0
     \end{pmatrix},\quad b'=\begin{pmatrix}
     x_1^{(1)}-r'\\
     0\\
     \vdots\\
     0.
     \end{pmatrix}
     $$ 
     %\textcolor{red}{EN EL DIBUJO PONES $x_1^{(1)}+r'$ EN LUGAR DE $x_1^{(1)}-r'$.}
     
In this way $\boldsymbol{\sigma}(Ax+b) $ yields a null field ($=\boldsymbol{0}$) in all points that satisfy $x^{(1)}_i\geq x_1^{(1)}+r'$.

     Now we choose $W'=W$ as before  and we solve \eqref{CTNN} up to a time $T_2$ so that:
     $$\phi_{T_2}(\phi_{T_1}(x_i;A,W,b);A',b',W')^{(2)}<x_1^{(2)}.$$
     %for very $i\neq j$.

     \item 
     Then, we can choose a hyperplane represented by 
         $$ A''=\begin{pmatrix}
      0 & 0 & ... & 0\\
      0 & 1 & ... & 0\\
      \vdots &\vdots & &\vdots\\
      0 &0 &... & 0
     \end{pmatrix},\quad b''=\begin{pmatrix}
     0\\
     -x_1^{(2)}+r''\\
     0\\
     \vdots\\
     0
     \end{pmatrix}
     $$      with $r''$ so that the image of $x_1$ under the previous transformations, the one to be properly classified, lies alone to one side of the hyperplane.
We can then choose $W''$ so that the corresponding vector field pushes this point towards $S_{m(1)}$. Note that this procedure not only assures that the image of $x_1$ gets into $S_{m(1)}$ but also allows  to determine  exactly the location of  the first coordinate of that point. We choose a target location $z_1^{(1)}$ in $S_{m(1)}$ such that there is no point $x_i$ fulfilling $x_i^{(1)}=z_1^{(1)}$.
     Moreover, note that his process keeps the $x^{(1)}$ coordinate of all points $i\neq 1$ unaltered while allocating $x_1$ to $S_{m(1)}$ (see Figure \ref{fig:clasfinal} for the final result). 
     Observe that the process above will not change the proper allocation of the points that already lie in the corresponding strip. 
     Applying this construction recursively, in finitely many steps, we can guarantee the correct collocation all points that initially did not fulfill the classification criterion, i. e. such that $x_j\notin S_{m(j)}$.
     
     %\textcolor{red}

     \begin{figure}
      \includegraphics[scale=0.4]{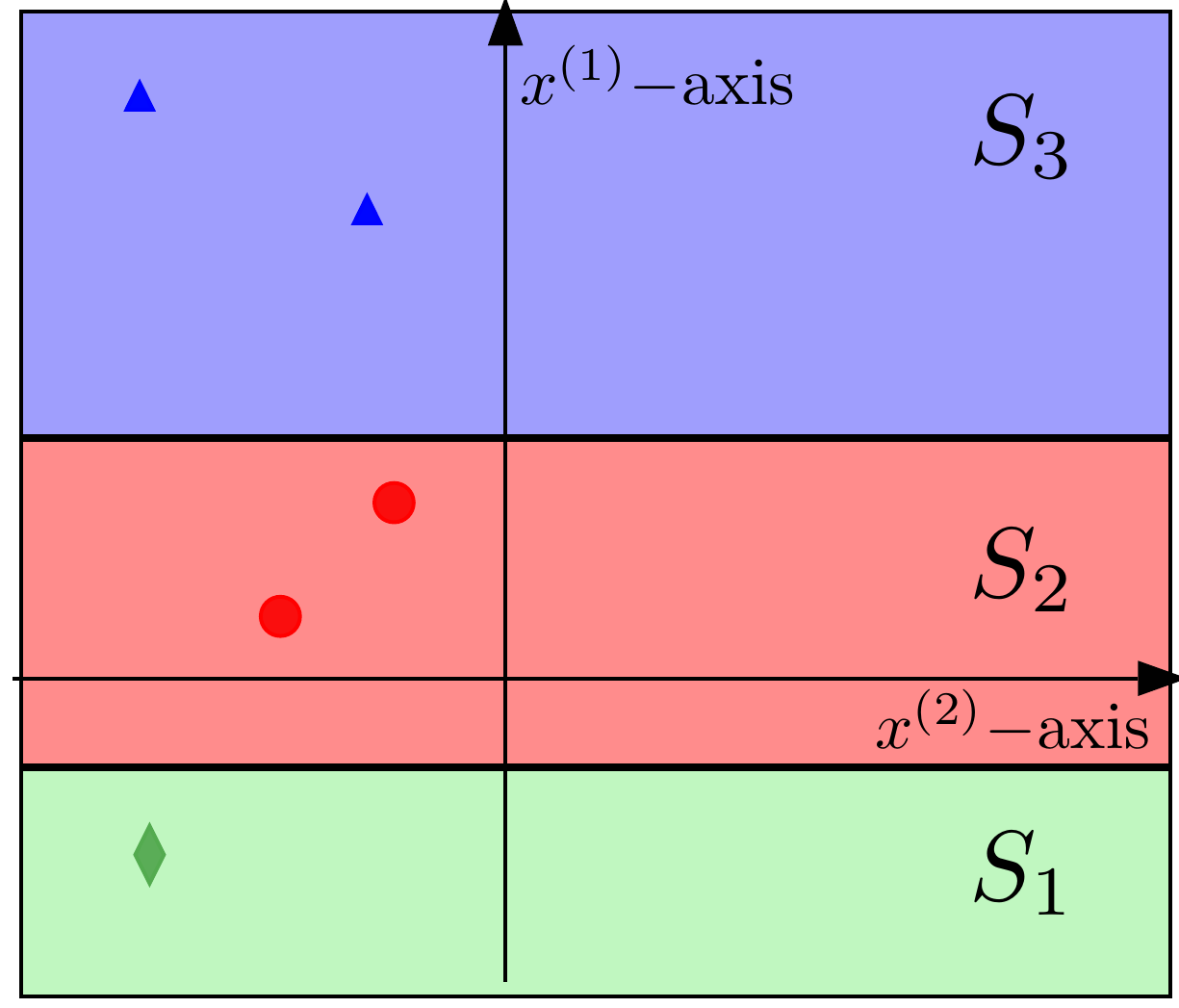}
      \caption{Final configuration after applying recursively the steps shown in Figure \ref{fig:Induction}.}\label{fig:clasfinal}

     \end{figure}

     \end{enumerate}

\item\textbf{ Time rescaling.}\label{rescalingth1} The argument above assures that the complete classification can be achieved in some time $T^*>0$. Classification can be assured in any time $T>0$, by time-scaling, $s=(tT)/T^*$. Indeed, we then have
$$\frac{d}{ds}x=\frac{T^*}{T}W(s)\boldsymbol{\sigma}\left({A}(s)x(s)+{b}(s)\right)\quad s\in(0,T) $$
and the multiplicative  term $T^*/T$ can be absorbed within the control ${W}$, keeping $A$ and $b$ unchanged (other than time rescaling).
\end{enumerate}
\end{proof}

\begin{remark}[On the activation function]\label{remact}
 Although the proof was presented in the particular case of the ReLU, the properties of the activation function $\sigma$  that we used in the proof are  the following:
 \begin{enumerate}
  \item $\sigma$ is Lipschitz continuous,
  \item $\sigma$ vanishes in $(-\infty,0]$,
  \item $\sigma$ is positive in $(0,+\infty)$.
 \end{enumerate}
The above properties of $\sigma$ suffice to assure that the main steps of our proof apply as represented in Figure \ref{fig:Induction}. 
 In fact, these conditions could also be further relaxed to consider activation functions that do not necessarily vanish in $(-\infty,0]$ but simply decay as $s$ tends to $-\infty$. But this would increase the cost of control. We shall not pursue this case here but the main steps of the proof would be essentially the same. \end{remark}

\begin{remark}[Complexity of the controls I]\label{complexity1}
 The controls are piecewise constant, and their complexity can be measured by the number and amplitude of the required switches.
The number of required switches for the classification task is $4N$, where $N$ is the number of points: $N$ for the preparation and $3N$ for the classification.  

Our proof shows that the role of the control $b$ is very much related to the distribution and spreading of the data along the Euclidean space. On the other hand, the controls $W$ and $A$ are of norm one in $L^\infty$. Their complexity can be estimated in terms of the number of switchings, which is reflected on the $BV$-norm. 

In our proof, the cost of the control is measured in terms of the needed time of control, which is the accumulation of the time spent in each iteration.  
The length of the $n$-th time interval $T_n$ for which the control is constant is larger when points whose ordering needs to be swapped are very close to each other, and this leads to a higher cost. Note also that in our proof, we did not exploit the possible structuring of data into clusters. Indeed, when a given point to be classified is embedded into a cluster corresponding to the same label, one may take advantage of that fact to classify the whole cluster simultaneously and, in this manner, reduce the number of switchings.

By the rescaling step \ref{rescalingth1} of Theorem \ref{TH1}, we have seen that classification can be achieved in any time horizon, in particular in $T=1$. A natural way of measuring the cost of control is to fix $T=1$ and then to estimate the $L^\infty$ and $BV$ norm of the controls.

A complete analysis of the cost of classification in terms of the complexity of the data-set is a challenging problem that requires substantial further work. However, we will give an insight of this issue relating it with the control cost in Section \ref{SUA}, Remark \ref{fractalremark}.
\end{remark}

\section{Simultaneous Control}\label{SSCONTROL}

The proof of the  previous classification Theorem shows that, other than classifying the data to the corresponding strips,  we can also,  for instance, completely control the first coordinate $x^{(1)}$ of all points to be classified at the final time, with the same controls.  

Based on that observation, in this section, we explore to which extent the classification can be enhanced to allocate to each datum to be classified a given point of destination. As we shall see, this can be done in an approximate manner, finding controls driving the ensemble of data to be classified as close as we wish from the given final locations chosen. This is an approximate simultaneous or ensemble controllability result. In fact, as we shall see, the exact simultaneous control can be achieved when initial data and targets are well separated. 
\begin{theorem}[Simultaneous Control]\label{TH2}
 Suppose $d\geq 2$. Fix  $T>0$ and let the activation function $\sigma$ be the ReLU \eqref{relu}. 
Let  $\{x_i\}_{i=1}^N\subset \mathbb{R}^d$  be distinct initial data and $\{z_i\}_{i=1}^N\subset \mathbb{R}^d$ be distinct final data. Then there exist piecewise-constant control functions $A,W\in L^\infty \left((0,T); \mathbb{R}^{d\times d}\right)$ and $b\in L^\infty\left((0,T),\mathbb{R}^d\right)$  such that the associated flow to \eqref{CTNN} with the controls $A,W,b$ fulfills:
 $$ \phi_T(x_i;A,W,b)=z_i\quad \forall i\in\{1,...,N\}.$$
 
\end{theorem}

\begin{remark}Several remarks are in order:
\begin{itemize}
\item Note that it is impossible to bring two different points to the same target by means of a Lipschitz vector field. Therefore, the control to distinct targets is the sharpest result one can achieve. 
\item By continuity, this also implies the approximate simultaneous control when the initial data or targets are not completely distinct. 
Indeed, let us assume for instance that  $z_i=z_j$. We can then slightly modify these targets to new points  points $ z_j'=z_j+\boldsymbol{\epsilon}$, with $|\boldsymbol{\epsilon}|\leq \epsilon$ arbitrarily small.
%, in a way that the new points 
 %$z_j'$ fulfill the distinction condition  $z_i'^{(1)}\ne z_j'^{(1)}$.
 The exact simultaneous control of the system to the targets  $\{z_i'\}_{i=1}^N\subset \mathbb{R}^d$ is then feasible and assures the control to a distance 
$\epsilon$ of the original targets  $\{z_i\}_{i=1}^N\subset \mathbb{R}^d$.

In other words, even when the targets $\{z_i\}_{i=1}^N\subset \mathbb{R}^d$ are not distinct, given any $\epsilon>0$, there exist control functions $A, W\in L^\infty \left((0,T); \mathbb{R}^{d\times d}\right)$ and $b\in L^\infty\left((0,T),\mathbb{R}^d\right)$ (depending on $\epsilon$) such that the associated flow to \eqref{CTNN} with the controls $A,W,b$ fulfills:
 $$ \phi_T(x_i;A,W,b)=z_i+\boldsymbol{\epsilon}\quad \forall i\in\{1,...,N\}, \quad \|\boldsymbol{\epsilon}\|=\epsilon.$$
 
\end{itemize}
\end{remark}

\begin{proof}
We proceed in several steps.
\begin{enumerate}[font=\bfseries]
 {\color{black}\item \textbf{Preparing the target.}
 The backward NODE
 \begin{equation*}
  \begin{cases}
   \dot{x}=W'\boldsymbol{\sigma}(A'x+b')\\
   x(T)=z
  \end{cases}
 \end{equation*}
is also a solution of the forward NODE \eqref{CTNN} by the controls $W(t)=-W'(T-t)$, $A(t)=A'(T-t)$ and $b(t)=b'(T-t)$. Therefore, we can apply the same argument done in Step 1 of Theorem \ref{TH1} for preparing the targets. At the price of $N$ extra switches, we can find controls $W,A$ and $b$ such that the targets fulfill
$$ z_i^{(1)}\neq z_{j}^{(1)}\qquad \forall j\neq i.$$
 }
 \item \textbf{Control the first component.} Applying the arguments in the proof of Theorem \ref{TH1}, we may control the first component  $x^{(1)}$ of each trajectory. In this way we can deduce  there existence of  $A,W$ and $b$ such that:
$$\phi_T(x_i;A,W,b)^{(1)}=z_i^{(1)}\quad \forall i\in\{1,...,N\}.$$

\begin{figure}

\hspace{-3cm}\begin{subfigure}[b]{0.1\textwidth}
 \includegraphics[scale=0.4]{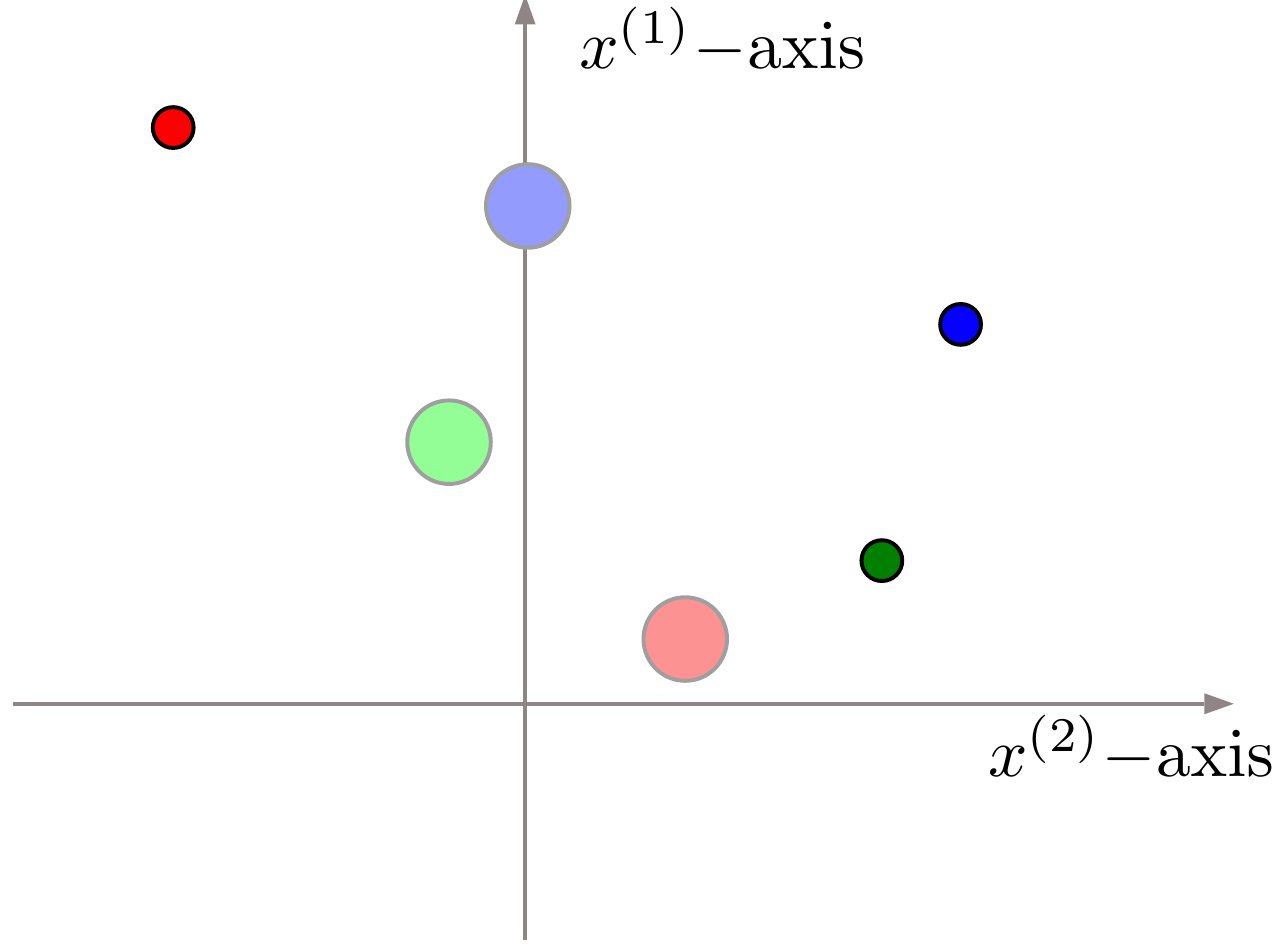}
  \caption{}
\end{subfigure}
\hspace{7cm}
\begin{subfigure}[b]{0.1\textwidth}
  \includegraphics[scale=0.4]{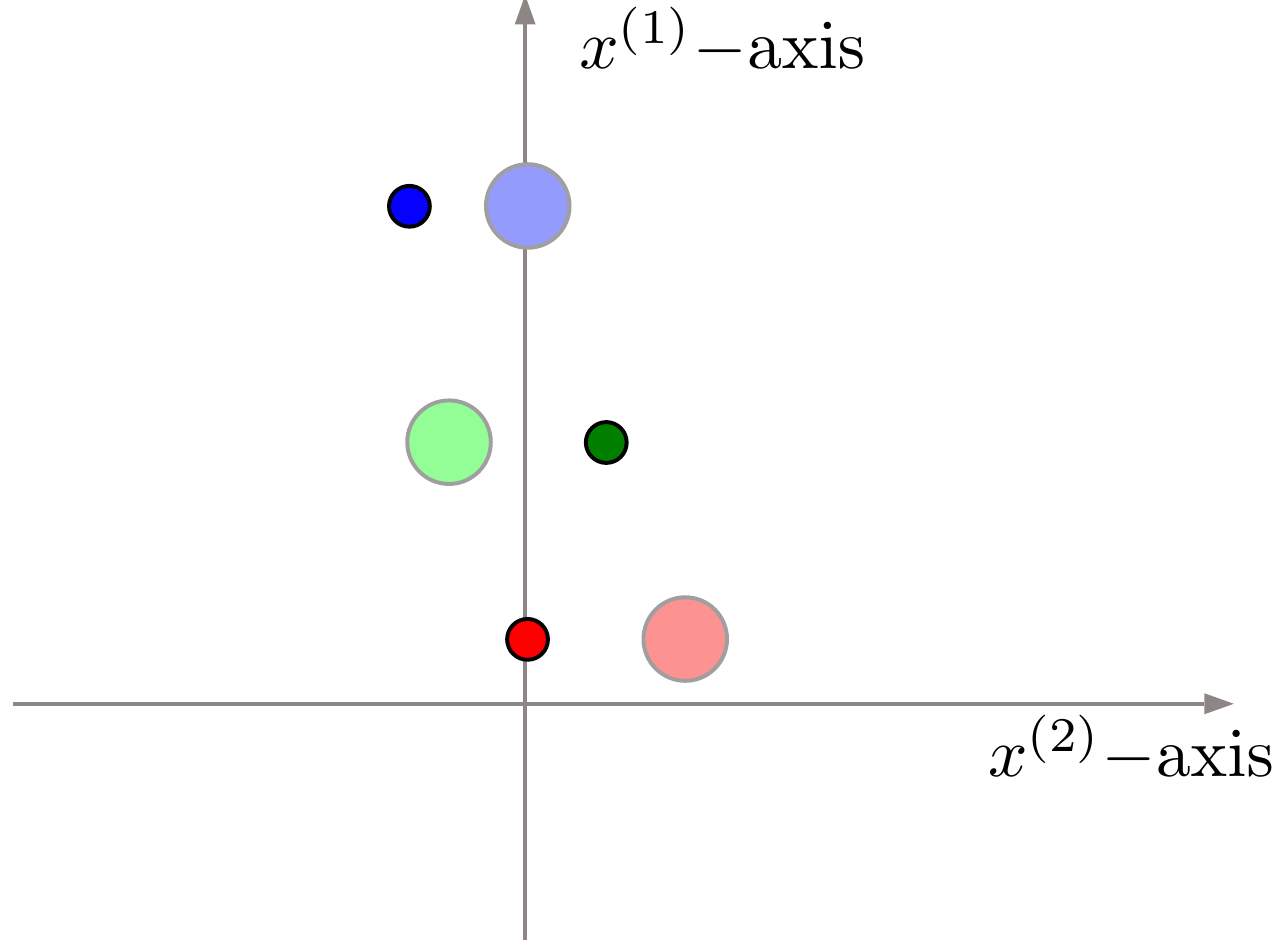}
  \caption{}
\end{subfigure}

 \caption{(A) Initial configuration for the simultaneous control before applying Theorem \ref{TH1}. The big colored circles represent the target locations for each point, denoted by small circles. (B) Configuration after controlling the $x^{(1)}$-component of each point. }\label{initsimcon}
\end{figure}

\item \textbf{Controlling to the targets.}\label{contrtotarget}
Now, we consider the hyperplanes $\{x^{(1)}=c\}$, for $c$ negative enough.
Let us take
$$ A'=\begin{pmatrix}
      1 & 0 & ... & 0\\
      0 & 0 & ... & 0\\
      \vdots &\vdots & &\vdots\\
      0 &0 &... & 0
     \end{pmatrix},\quad b'=\begin{pmatrix}
                              -c\\
                              0\\
                              \vdots\\
                              0
                             \end{pmatrix}
$$
%and remind that $dim(Ker(A'))=d-1$.
We use $W'$ to rotate the field in the direction to control the point $\phi_T(x_1;A,W,b)$ to its target $z_1$ with time $T'>0$.

 Then we set $c=z_1^{(1)}+\delta$ for $\delta$ small enough and we consider the hyperplane $x^{(1)}=c$ and apply the same argument, finding $W''$ that controls $\phi_{T'}(\phi_T(x_2;A,W,b);A',W',b')$ to $z_2$. Note that $z_1$ remains static along this deformation (see Figure \ref{fig:simcon}).

 Since $z_i^{(1)}\ne z_j^{(1)}$ for all $i, j$ holds, we can proceed inductively.  After a time-rescaling, as in Theorem \ref{TH1}, we conclude the proof of the simultaneous controllability. 

\begin{figure}%[h!]
\hspace{-3cm}\begin{subfigure}[b]{0.1\textwidth}
  \includegraphics[scale=0.4]{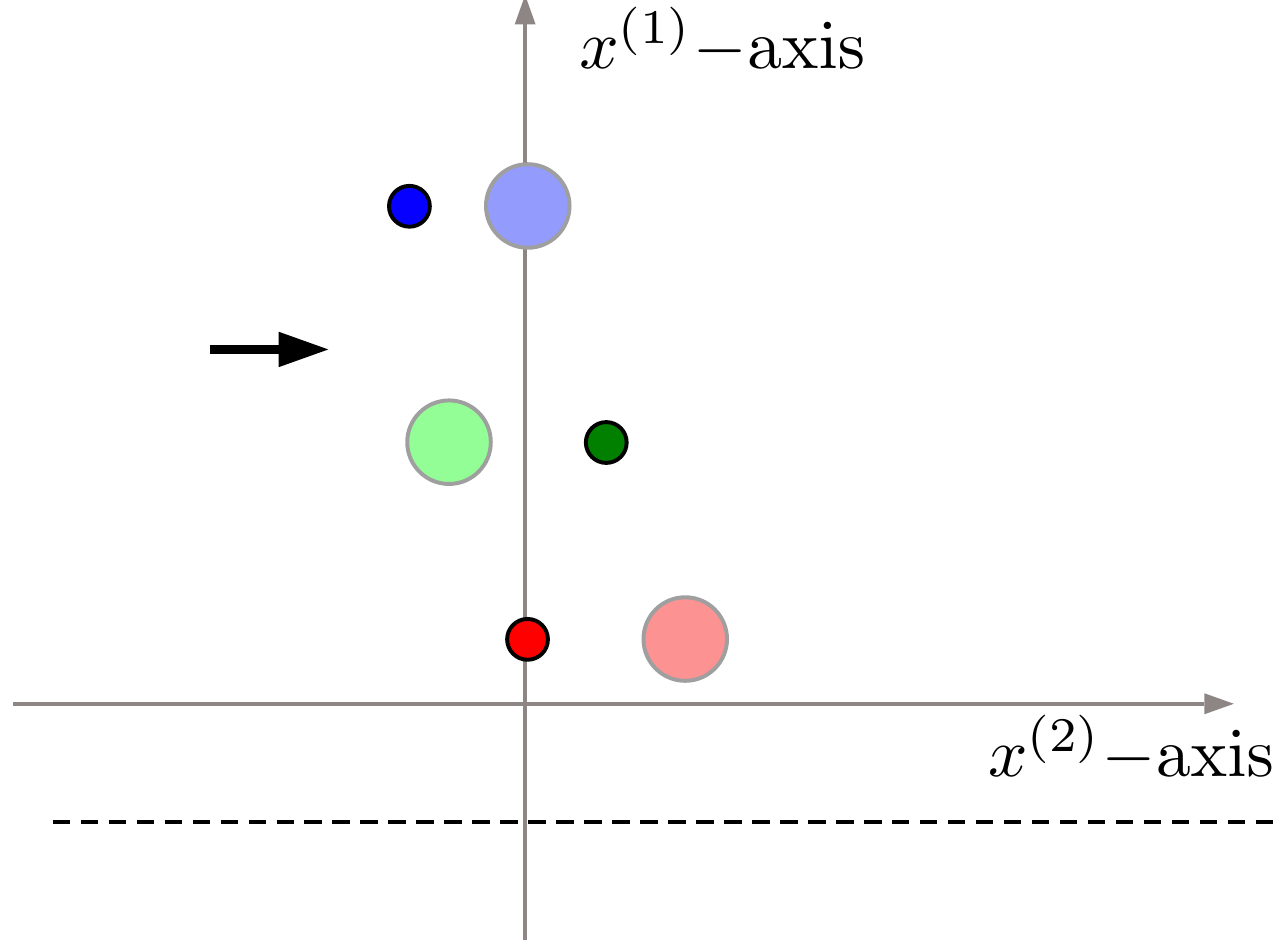}
  \caption{}
\end{subfigure}\hspace{7cm}
\begin{subfigure}[b]{0.1\textwidth}
  \includegraphics[scale=0.4]{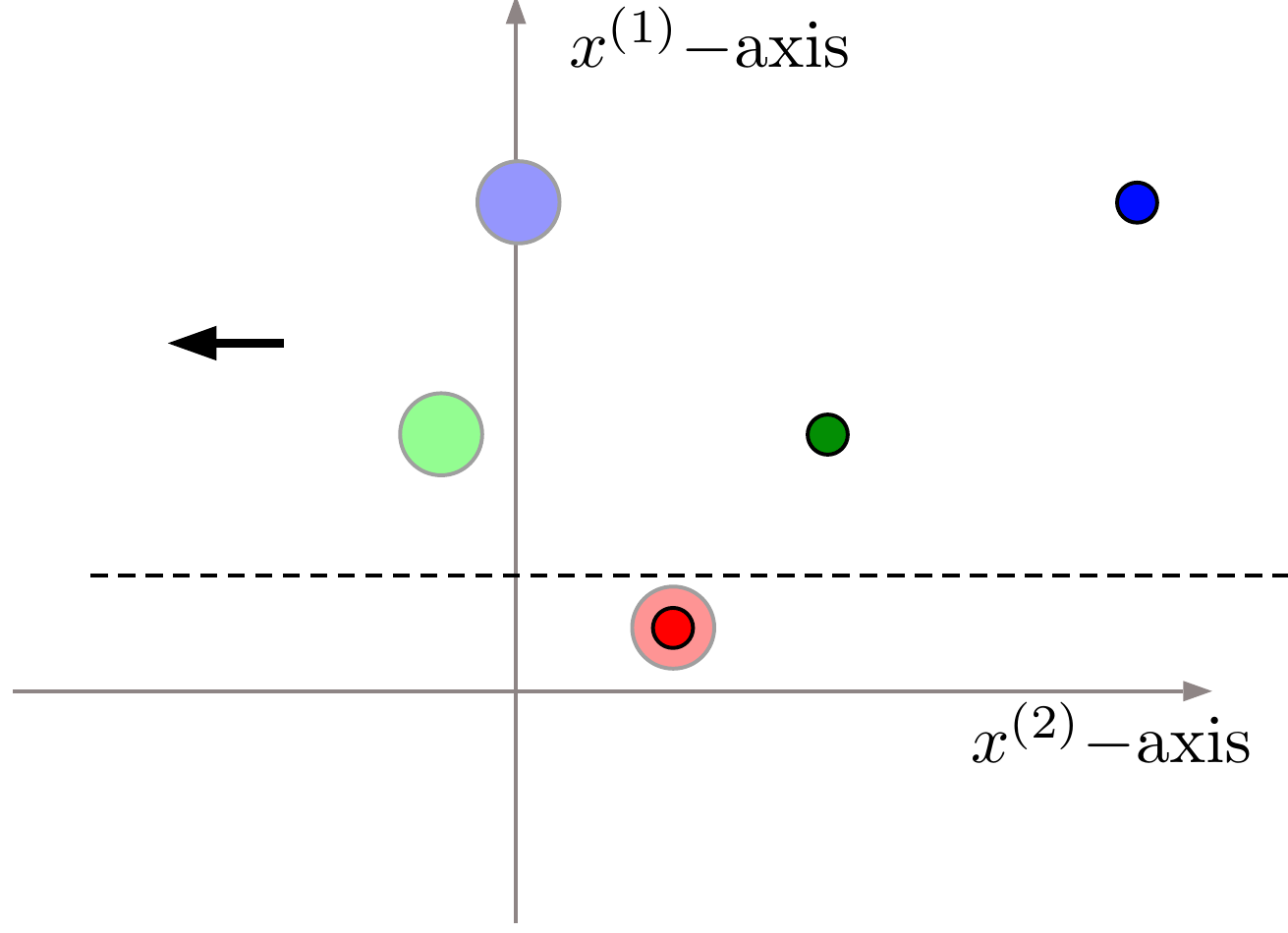}
    \caption{}
\end{subfigure}

\hspace{-3cm}\begin{subfigure}[b]{0.1\textwidth}
  \includegraphics[scale=0.4]{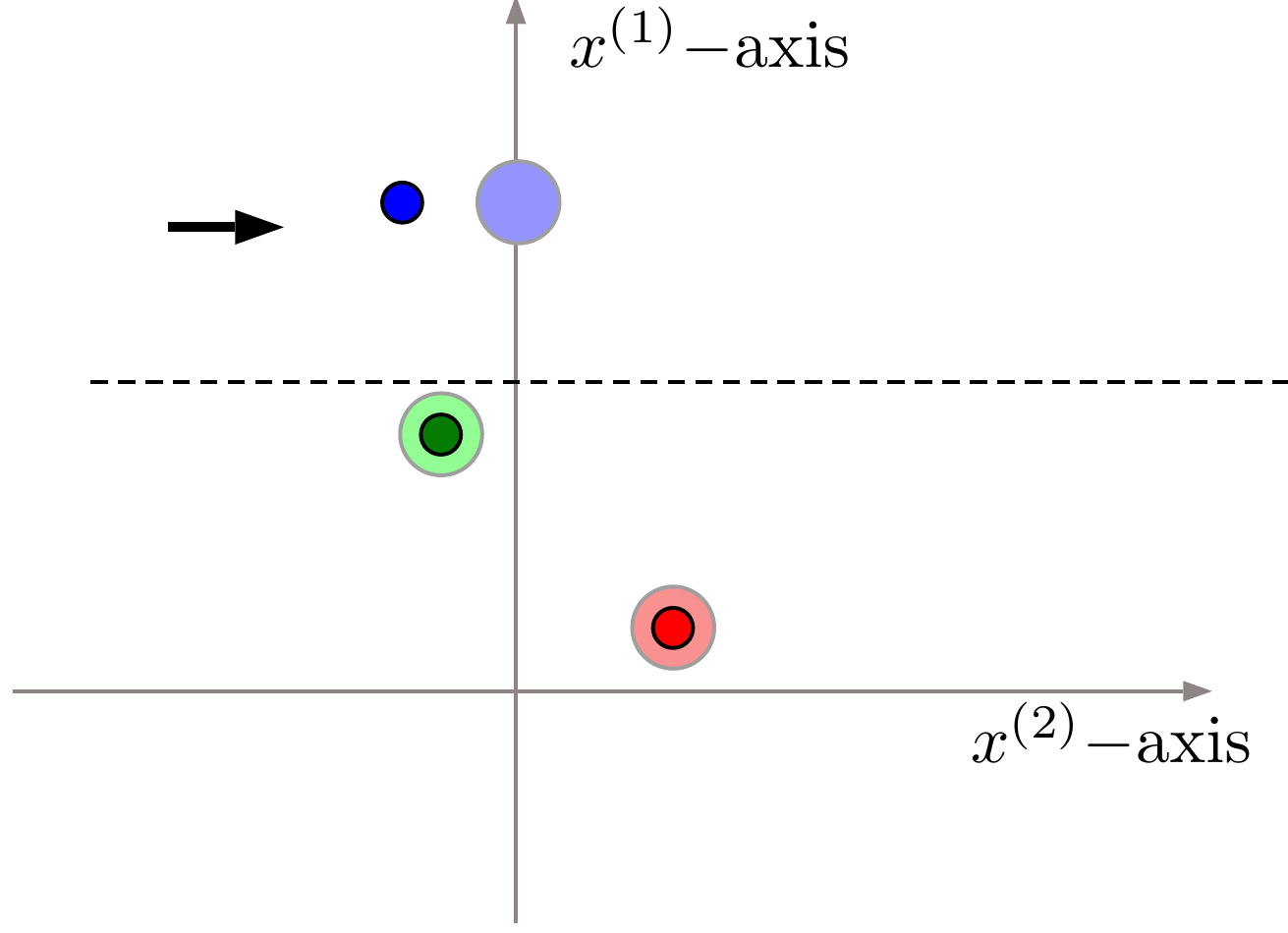}
    \caption{}
\end{subfigure}\hspace{7cm}
\begin{subfigure}[b]{0.1\textwidth}
  \includegraphics[scale=0.4]{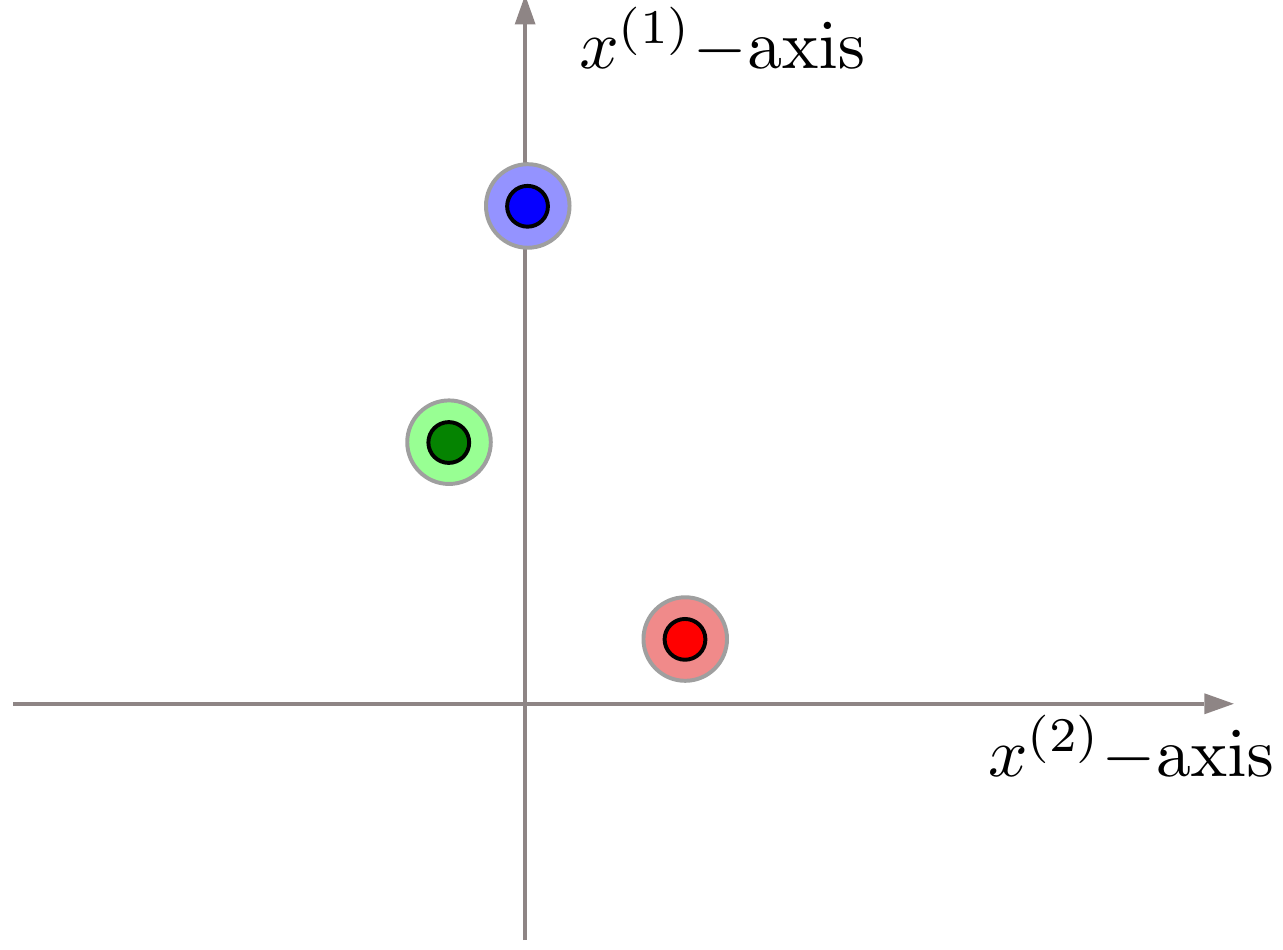}
    \caption{}
\end{subfigure}

\caption{Iterative process leading to the simultaneous controllability. The colored big circles represent the targets, the small ones the data to e controlled, the dashed line is the chosen hyperplanes.}\label{fig:simcon}
\end{figure}
\end{enumerate}

\end{proof}

 \begin{remark}[On the activation function] As in the context of Theorem \ref{TH1}, although Theorem \ref{TH2} was stated for the ReLU \eqref{relu} activation function,  it can be extended to more general activation functions as in Remark \ref{remact}.
 \end{remark}
\begin{remark}
The classification result in Theorem \ref{TH1} is a direct consequence of the simultaneous controllability result in Theorem \ref{TH2}. In fact, as a corollary of Theorem \ref{TH2}, the classification can be achieved for an arbitrary partition  $\mathcal{S}$  of the Euclidean space and any allocation of labels.
\end{remark}

\begin{remark}[Complexity of the controls II]\label{complexity2}
Note that, with respect to the classification result and Remark \ref{complexity1}, the number of switches in this simultaneous control result has increased by $2N$, remaining of the order of $N$. However, its norm  will depend also on the chosen targets. 
The reason is that the gap for the choice of a  hyperplane in step \ref{contrtotarget} of the proof of Theorem \ref{TH2} can be small, depending on the target configuration. This will imply the smallness of  the vector field chosen, and therefore the needed time will be larger.

 Note however, that this discussion is specific for the controls we have constructed. The analysis of optimal control strategies is a challenging open problem.

\end{remark}

    \begin{remark}[Nonlinearities with drift]  Note that the simultaneous control result we have proved  would also hold in an approximate manner for systems $$\dot{x}=f(x)+W(t)\boldsymbol{\sigma}(A(t)x+b(t))$$ with $f$ bounded. Indeed, the control $W$ can enhance the effect of the nonlinearity making $f$ negligible.  Whether an exact simultaneous control result in the spirit of Theorem \ref{TH2} can be obtained in this setting is an interesting open problem. %However, it would be more difficult to extend these results to systems that do not have similar properties to activation functions.
    \end{remark}
    
    \begin{remark}[Sparsity]
    The proof above is built-up in consecutive steps. In each step we employ only an hyperplane and a direction of motion for the  flow. The proof does not require to activate at all time the whole ensemble of components of the controls $A, W, b$.
    
To be more precise, let us consider a hyperplane with normal vector  $a$ and displaced by $b$  from the origin.
Then, the sign of scalar product $\langle a,x-b\rangle$ determines  whether $x$ is located at one side or the other of the hyperplane. Therefore, applying the scalar function $\sigma$ to $\langle a,x-b\rangle$, one obtains either $0$ if $x$ is in one side of the hyperplane or a positive one if $x$ is in the other one.
    
    One can then set the direction of motion of the active half-space by choosing a vector $w\in \mathbb{R}^d$, generating a dynamics of the form:
    $$\dot{x}=w\sigma\left(\langle a,x-b\rangle\right)$$
     In the expression above only $3d$ controls are activated. But, taking into account that $\langle a,b\rangle$ is a scalar value,  the total number of active controls is $2d+1$.
    
  As a conclusion of this discussion we see that our proofs hold with controls  $W, A$ and $b$ with the following particular structure:   
    $$ W=\begin{pmatrix}
          w_1(t)\\
          w_2(t)\\
          \vdots\\
          w_d(t)
         \end{pmatrix}(1,1,\cdots,1),\quad A=\mathrm{diag}\begin{pmatrix}a_1(t)\\\ a_2(t)\\ \vdots\\a_d(t)\end{pmatrix},\quad b(t)= \tilde{b}(t)\begin{pmatrix}
                                    a_1(t)\\
                                    \vdots\\
                                    a_d(t)\end{pmatrix}$$
    where $w_i,a_i,\tilde{b}\in L^\infty((0,T),\mathbb{R})$, involving a total of $2d+1$ controls.
    
    The form of dynamics \eqref{CTNN} is richer since it allows for more general matrices $W, A$. But the arguments above show that the actual dimension of the needed controls can be drastically diminished. This is a manifestation of the fact that the results own this paper can be achieved assuring further sparsity or polarization conditions on the controls.

    \end{remark}

\section{Universal Approximation}\label{SUA}

In this section we prove the Universal Approximation Theorem using the methods developed to treat simultaneous controllability.

We mainly consider simple functions of the form 
\begin{equation}\label{simple}
 f=\sum_{m=1}^M \alpha_m\chi_{\Omega_m},
 \end{equation}
as in Figure \ref{fig:UA1}. 
\begin{figure}%[h!]
 \includegraphics[scale=0.35]{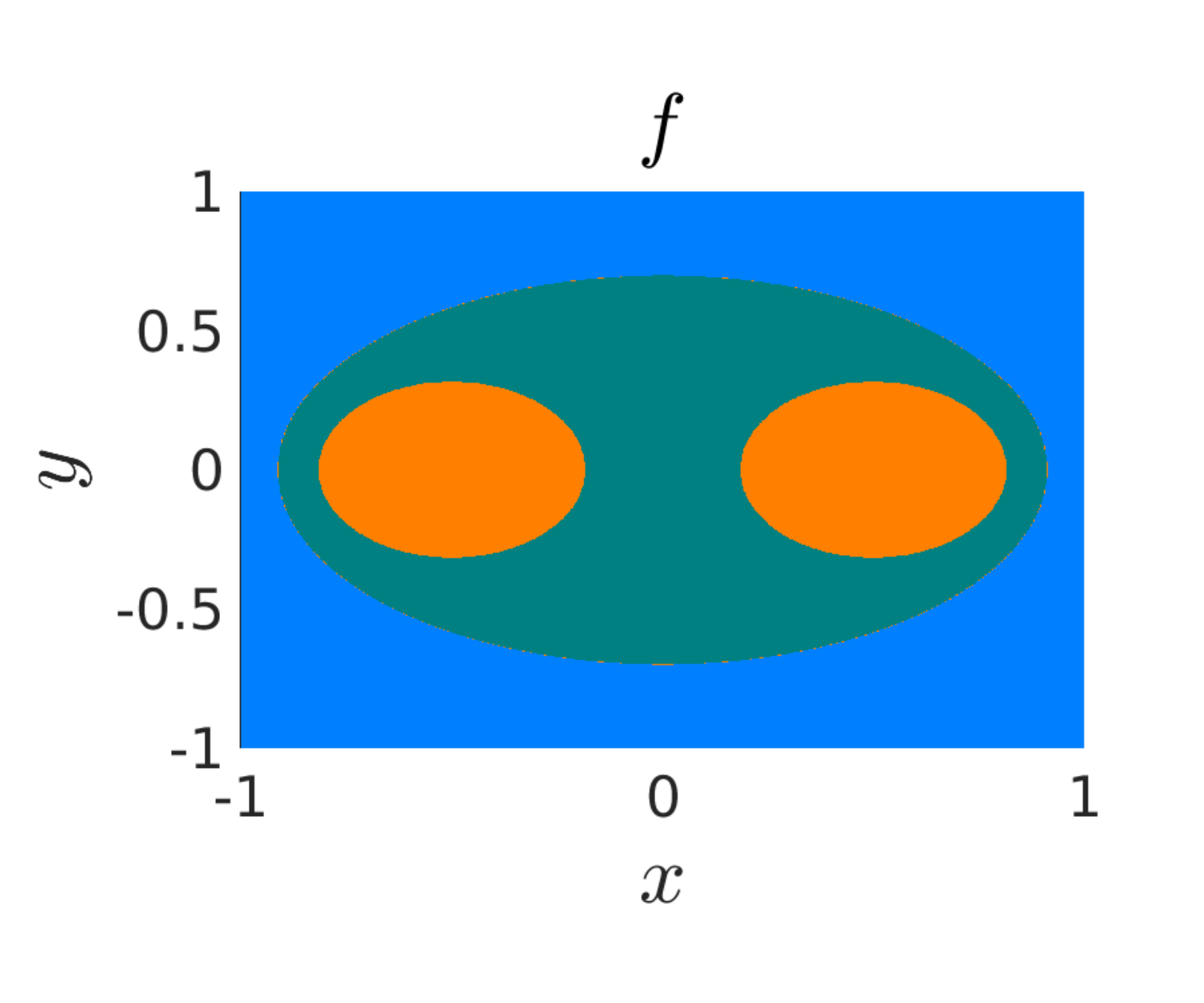}
 \caption{Simple function in the square $\Omega = [-1, 1] \times [-1, 1]$, $M=3$, each color represents a set $\Omega_m$.}\label{fig:UA1}
\end{figure}
Our goal is to show how, controlling a NODE, such a function can be approximated in $L^2_{\mathrm{loc}}$. Obviously, by density, the result extends immediately to any bounded measurable function. 
But, as we shall see, the complexity of the controls we will build depends on the structure of the simple function and, more precisely, on the box-counting dimension of the boundaries of  the sets $\Omega_m$.

\smallskip
\noindent In the context of NODEs the goal of approximating  a function can be interpreted as a simultaneous control problem for an uncountable ensemble of trajectories. We do it by viewing the values $f$ at each point $x$ as the target or label that each $x$ have to be associated with along the trajectories of the NODE. Consequently, the proofs of the main results of this section will employ the ideas of the previous ones, but with substantial added work.

\smallskip
\noindent For simplicity, we assume that all the boundaries of  the sets $\Omega_m$ constituting the simple target function $f$, have a finite perimeter.  Later on in Remark \ref{fractalremark} we shall show that the same proof works under weaker conditions, to the price of increasing the complexity of the needed controls. We will see how the box-counting dimension of the boundary of the characteristic sets arises in the control cost. The approximation can only be guaranteed over bounded sets $\Omega$, as we shall see in  Section \ref{Transport} (Remark \ref{impUA}). Otherwise the use of unbounded controls would be required.

More precisely, the main result of this section is as follows:

\begin{theorem}[Universal Approximation]\label{TH3}
 Let $d\geq 2$ and $T>0$.  Let $\Omega\subset\mathbb{R}^d$ be a bounded set, and $\sigma$ be the ReLU activation function in \eqref{relu}. 
 Then, for any $f\in L^2(\Omega;\mathbb{R}^d)$ and $\epsilon>0$ there exist $A,W\in L^\infty((0,T);\mathbb{R}^d)$ and $b\in L^\infty((0,T);\mathbb{R}^d)$ such that the flow generated by \eqref{CTNN}, $\phi_T(\cdot;A,W,b)$,  satisfies:
\begin{equation}\label{aproxepsilon}
\|\phi_T(\cdot;A,W,b)-f(\cdot)\|_{L^2(\Omega)}<\epsilon.
\end{equation}
 \end{theorem}

\smallskip 
\noindent Theorem \ref{TH3} is a direct consequence of the following theorem:

\begin{theorem}[Approximation for simple functions]\label{THsimple}
 Let $d\geq 2$ and $T>0$.  Let $\Omega\subset\mathbb{R}^d$ be a bounded set, and $\sigma$ be the ReLU activation function in \eqref{relu}. Consider any  $f$ in the form of \eqref{simple} with characteristic sets $\Omega_m$ of finite perimeter.
 Then, for any $\epsilon>0$ there exist piecewise-constant controls $A,W\in L^\infty((0,T);\mathbb{R}^d)$ and $b\in L^\infty((0,T);\mathbb{R}^d)$ such that the flow generated by \eqref{CTNN}, $\phi_{T}(\cdot;A,W,b)$,  satisfies:
\begin{equation*}%\label{aproxepsilon}
\|\phi_{T}(\cdot;A,W,b)-f(\cdot)\|_{L^2(\Omega)}<\epsilon.
\end{equation*}

%Furthermore, the norms of the controls can be estimated: there exists a constant $\mathcal{K}(f)$ depending on the target simple function $f$

  Setting $T=1$, the norm the controls and their switches satisfy:
 \begin{equation}\label{FinalT}
 \|W\|_{L^\infty((0,1);\mathbb{R}^{d\times d})} \lesssim_{\Omega,f} \epsilon^{-4d(d-1)} \quad \text{ as }\epsilon \to 0
\end{equation}
\begin{equation}\label{Finalb}
 \|A\|_{L^\infty((0,1);\mathbb{R}^{d\times d})}=1,\quad \|b\|_{L^\infty((0,1);\mathbb{R}^{d})} \lesssim_{\Omega,f} \epsilon^{-2d(d-1)}\quad \text{ as }\epsilon \to 0
\end{equation}
   \begin{equation}\label{FinalS}
      \text{The number of switches of the controls $A,W,b$ is of the order of $\epsilon^{-2d(d-1)}$.}
     \end{equation}
 %By $\lesssim_{\Omega,f}$ we understand that there exists a constant depending on $\Omega$ and $f$ such that
 \end{theorem}

\begin{proof}

Let us first explain the main ideas of the proof. 

In the NODE formulation of this result,  each  input $x\in\mathbb{R}^d$ has a target $ f(x) \in\mathbb{R}^d$. The actual value of $f(x)$ depends on the characteristic set in which $x$ lies. Ideally, we would like to find controls $A,W$ and $b$ such that for every $x\in \Omega$ we have that
\begin{equation*}
 \begin{cases}
  \dot{x}=W(t)\boldsymbol{\sigma}(A(t)x(t)+b(t))\\
  x(0)=x\\
  \phi_T(x;A,W,b)= f(x).
 \end{cases}
\end{equation*}
But, due to the discontinuity of $f(\cdot)$ and the continuous dependence  of the  solutions of the ODE with respect to the initial data,  this task is impossible. Thus, we have to relax the problem so to formulate our goal in an approximate manner.

%{\color{green}

First, we will define for any $h>0$ small enough, a partition of $\Omega$ in a way that
$$\Omega=\Omega_c\sqcup \Omega_h,\qquad \mathrm{meas}(\Omega_h) \sim h$$
%$$.$$%the measure of $\Omega_h$ is of the order of $h$.
Then, the control problem we seek to solve is the following. For all $ \nu>0$, find $A,W$ and $b$  (that depend on $h$) satisfying that the associated controlled flow $\phi_T(\cdot,A,W,b)$ satisfies:
\begin{align*}
\begin{cases}
 |\phi_T(x;A,W,b)-f(x)|<\nu\quad &x\in\Omega_c\\
|\phi_T(x;A,W,b)|<K\quad &x\in\Omega_h.
\end{cases}
\end{align*}
with $K$ independent of $h$ and $\nu$.

%}

This result can be achieved by means of a careful implementation and further development of the ideas of  the proof of the simultaneous controllability result  of Theorem \ref{TH2}. The proof is structured as follows.

In \ref{UAStep1}  we construct the partition  $(\Omega_h, \Omega_c)$. Furthermore, we will split $\Omega_c$ in $d$-dimensional hyperrectangles, each of them having a single target associated in $\mathbb{R}^d$.
\ref{UAStep2} concerns the compression of $\Omega_c$ and of its hyperrectangles plus an adaptation of Theorem \ref{TH1}.  The goal is to arrive to a configuration in which one will be able to apply the simultaneous controllability, Theorem \ref{TH2}. 
Finally, in \ref{UAStep3}, we apply the simultaneous controllability, Theorem \ref{TH2}.

\begin{description}[leftmargin=0cm]
\item[\namedlabel{UAStep1}{\textbf{Step 1}}]\textbf{Construction of $\Omega_h$.}

    \begin{enumerate}[leftmargin=0.5cm,font=\bfseries]
     \item[1.1] \textbf{Covering of  boundaries of the characteristic sets.} 
     Thanks to the compactness of the boundary
        $\Gamma=\bigcup_{m=1}^M \partial \Omega_m$
        it can be covered with a finite number of hypercubes of side $h$. Let us denote this cover by $\Gamma_h$. The cover $\Gamma_h$ can be chosen in a way that each edge of each hypercube is oriented in the direction of an element of the canonical basis. 
        
        Since the length of $\Gamma$  is finite, and its dimension is $d-1$, the number, $N_\Gamma$, of needed hypercubes of size $h$  in the cover $\Gamma_h$ is:
        $$N_\Gamma\lesssim_{\Gamma} h^{-(d-1)}\quad \text{ as }\quad h\to0. $$
        Since each hypercube has volume $h^d$, the volume of the cover $\Gamma_h$ is of the order of $h$. 
        Note that,  $\Gamma_h$ (represented by in white in Figure \ref{fig:UA2}) contains the jumps of the simple function $f$ to be approximated. %Therefore, the approximation will not very accurate in this subset.

        \begin{figure}[h!]
        \centering
        \begin{subfigure}[t]{0.4\textwidth}
        %\centering
         \includegraphics[scale=0.35]{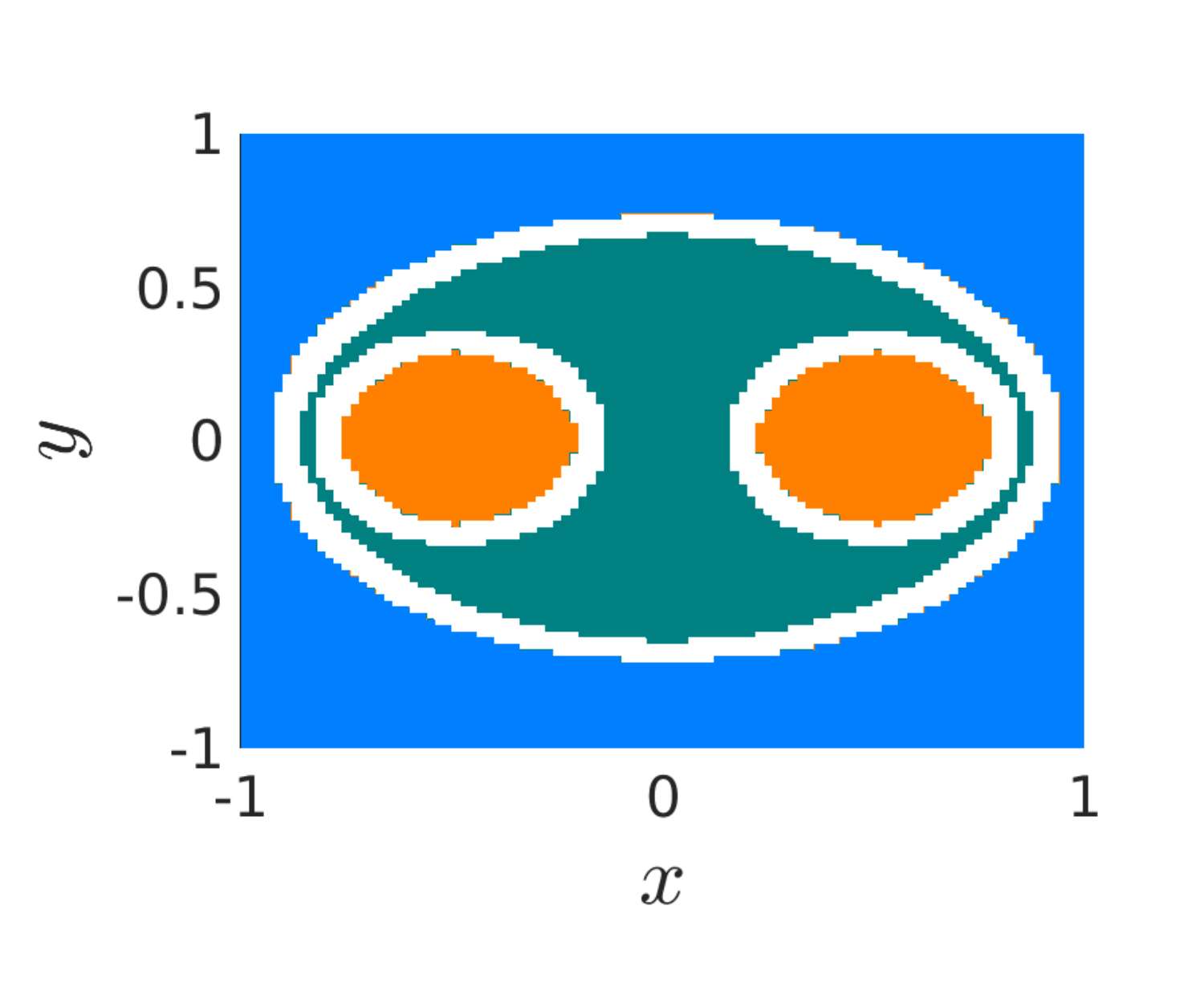}%\hspace{3cm}
\caption{}\label{fig:UA2} 
%\hspace{5cm}
\end{subfigure}
%\hspace{3cm}
                        \begin{subfigure}[t]{0.4\textwidth}
                        \centering
               \includegraphics[scale=0.6]{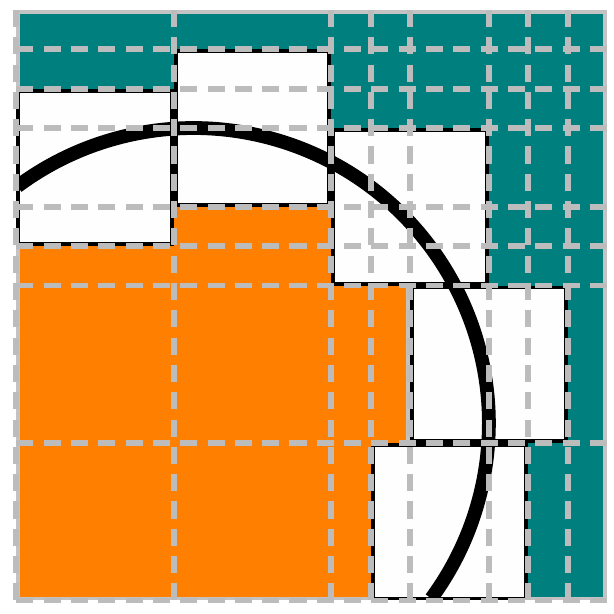}
\caption{}\label{fig:hyprec}                           
\end{subfigure}
\caption{(A) The resulting function after not considering the set $\Gamma_h$ showed in white. (B) Qualitative representation $d$-dimensional hyperrectangles' generation. The black curve represents part of the boundary of a characteristic set. The white hypercubes are part of the cover of such boundary and the grey dotted lines represent the hyperplanes \eqref{hyperplanes} generated out of the sides of the hypercubes of the cover. The colors orange and dark green represent the values of the target function in coherence with Figure \ref{fig:UA2}.}
        \end{figure}

\smallskip
     \item[1.2] \textbf{Determine the final $d$-dimensional hyperrectangles.}
     {\color{black}
        From the cover $\Gamma_h$ we want to find a partition of $\Omega\setminus\Gamma_h$ made out of $d$-dimensional hyperrectangles $\{\mathcal{H}_l\}_{l=1}^N$ satisfying the following:
        $$\forall x,y\in \mathcal{H}_l, \quad f(x)=f(y),$$
        %in a way that each hyperrectangle has at most a target associated depending on the value of the simple function on such hyperrectangle.
        i.e. each hyperrectangle has at most a single associated value $\alpha_m$.
        
        For finding such partition we proceed as follows. We consider all hyperplanes that define the hypercubes in $\Gamma_h$ (see Figure \ref{fig:hyprec} for a qualitative illustration). Altogether we have        
        \begin{equation}\label{hyperplanes}
         \left\{x^{(k)}=c_{k,n}\right\}\qquad 1\leq n\leq N_{\Gamma,k}\qquad 1\leq k\leq d, \qquad c_{k,n}<c_{k,n+1}
        \end{equation}
        where $N_{\Gamma,k}$ is the number of hyperplanes orthogonal to the $k$-th vector of the canonical basis and let $c_{k,N_{\Gamma,k}+1}=+\infty $ . By construction one has that $ N_{\Gamma,k}\leq 2N_{\Gamma}$ and therefore       
        \begin{equation}\label{hypdim}
        N_{\Gamma,k}\lesssim_{\Gamma} h^{-(d-1)}\quad \text{ as } h\to 0\qquad \text{for every }1\leq k\leq d.
        \end{equation}

         %By convenience, we order the $c_{k,n}$ in a way that $c_{k,n}<c_{k,n+1}$.
         
        All these hyperplanes generate  $d$-dimensional hyperrectangles $\{\mathcal{H}_l\}_{l=1}^{N}$ with the property that every $\mathcal{H}_l$ has a associated target in $\alpha_{m(l)}\in\mathbb{R}^d$. Then, the number of final hyperrectangles $N$ can be estimated by:
        $$N\leq\prod_{k=1}^d N_{\Gamma,k}\leq 2^dN_\Gamma^d.$$
        Hence
        \begin{equation}\label{enH}
                   N\lesssim_{\Gamma} h^{-d(d-1)} \quad \text{ as } h\to 0.
        \end{equation}

     }

        \medskip
       
    \end{enumerate}

         We consider sufficiently thin $d$-dimensional strip around each hyperplane $\{x^{(k)}=c_{k,n}\}$ of side $\zeta=h^d$ given by
         \begin{equation}\label{varsigmastrip}
         \mathscr{S}_{n,k}:=\left(\mathbb{R}\times\mathbb{R}\times\cdots \times [c_{n,k}-\zeta,c_{n,k}+\zeta]\times\cdots \times\mathbb{R}\right)\cap \Omega.
         \end{equation}
         
         These strips, of measure of the order of $h^d$, belong to $\Omega_h$ together with $\Gamma_h$ (see Figure \ref{fig:removed}.).
         The total number of hyperplanes is of the order of $h^{-(d-1)}$, \eqref{hypdim}, and therefore also the number of the thin strips. Hence, one can see that the measure of $\Omega_h$ is of the order of $h$. 
         After removing the $d$-dimensional strips, we are left with at most $N$ connected sets.
         
         Then, making an abuse of notation, we denote the number by $N$ the number of connected components of $\Omega\setminus \cup_{k,n}\mathscr{S}_{n,k}$ and its connected components by  $\{\mathcal{H}_l\}_{l=1}^N$.

         Now the $L^2$ norm can be split in two:%, and imposing that we want the norm to be less than $\epsilon$ we get
        \begin{align*}
            \|f-\phi\|_{L^2(\Omega)}^2&=\sum_{i=1}^N \int_{\mathcal{H}_i} |f-\phi|^2dx+\int_{\Omega_h} |f-\phi|^2dx\\
            &\lesssim_{\Gamma} N\nu^2+h^{-(d-1)}h^d.%\lesssim  \epsilon^2
        \end{align*}
        We see that it suffices to pick:%This means that one must require that:
        \begin{subequations}
         \begin{align}
        h&\lesssim_{\Gamma} \epsilon ^2\label{enh}\\
         \nu&\lesssim_{\Gamma}\epsilon N^{-\frac{1}{2}}\lesssim_{\Gamma} \epsilon h^{\frac{d(d-1)}{2}}\lesssim_{\Gamma}  \epsilon^{1+d(d-1)}\label{ennu}
         \end{align}
        \end{subequations}
        to ensure that
        $$\|f-\phi\|_{L^2(\Omega)}\lesssim_{\Gamma} \epsilon.$$

      \begin{figure}%[h!]
   \includegraphics[scale=0.4]{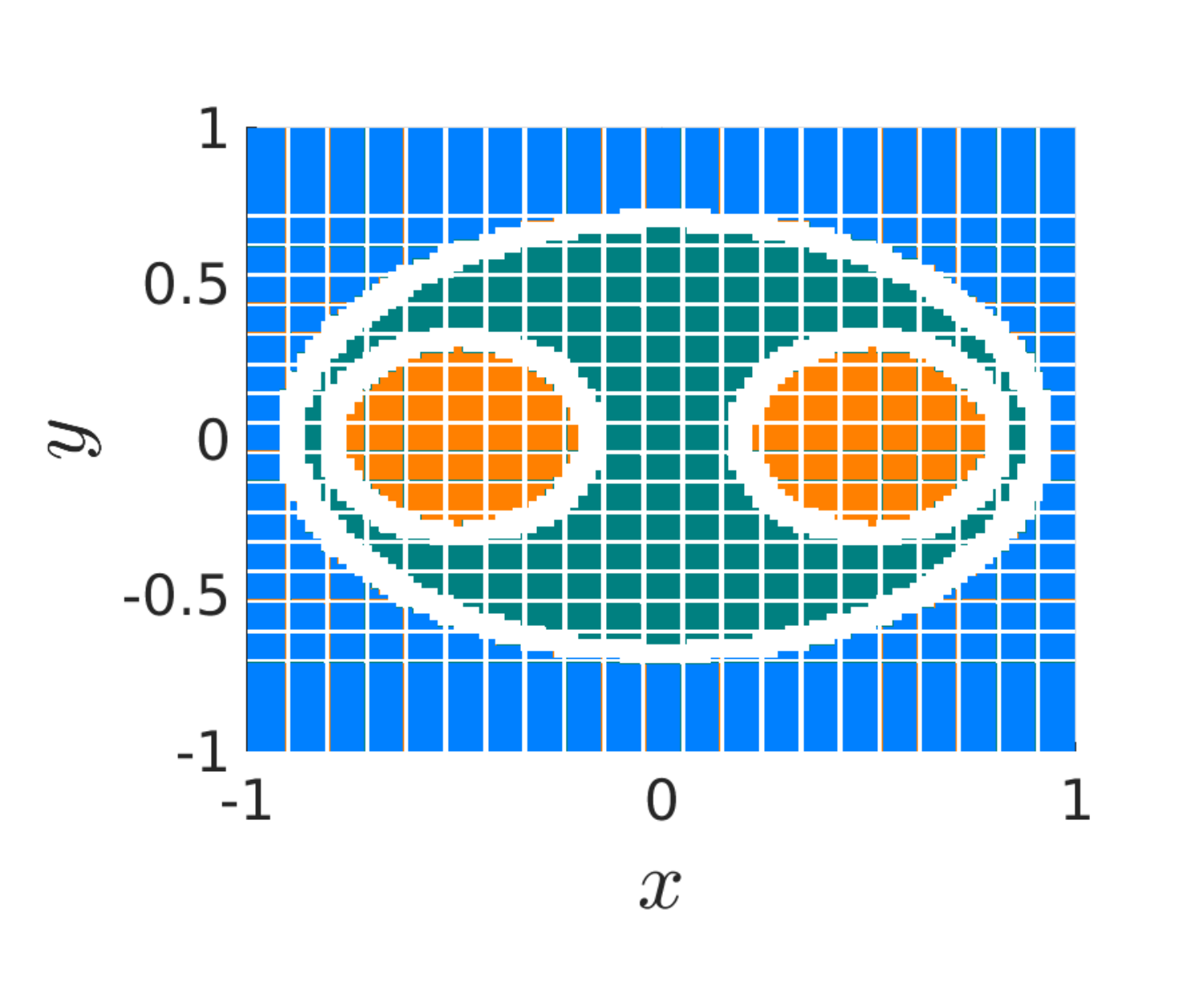}
   \caption{Qualitative representation of $f$ with removed strips around the meshing hyperplanes. The white regions correspond to $\Omega_h$.}\label{fig:removed}
    \end{figure}

    \medskip
    
 \item[\namedlabel{UAStep2}{\textbf{Step 2}}] \textbf{Compression of $\Omega_c$.} In this Step we apply a compression Lemma that allows us to apply the simultaneous controllability of Theorem \ref{TH2}.
 Given a set $\omega\subset\mathbb{R}^d$, let us define
$$\mathrm{diam}(\omega):=\max_{x_1,x_2\in\omega} |x_1-x_2|$$
$$\mathrm{diam}_{(k)}(\omega):=\max_{x_1,x_2\in\omega} |x_1^{(k)}-x_2^{(k)}|.$$
Furthermore, let us denote by $(\omega)^{(k)}$  the $k$-th component of the set $\omega$, i.e. $$\omega^{(k)}:=\{x\in\mathbb{R}: \exists y\in\omega\quad \text{ such that } y^{(k)}=x\}.$$

\textcolor{black}{As done in Theorem \ref{TH2}, in step 1, since the backward NODE is a solution for the forward NODE for the appropiate controls, we can find an equivalent set of targets $\{\alpha_m\}_{m=1}^M$ such that $\alpha_m^{(1)}\neq\alpha_{m'}^{(1)}$ if $m\neq m'$. Without loss of generality, we will assume that the set of targets fulfills such contidition.}

Let us summarize the properties we have built in the following
\begin{equation}\tag{\textbf{H}}\label{H}
 \begin{cases}
  \text{(a) Consider hyperplanes $\{x^{(k)}=c_{k,n}\}$ in the form of \eqref{hyperplanes}}.\\
  \text{(b) Let $\zeta>0$ be the width of the strips $\mathscr{S}_{n,k}$ in \eqref{varsigmastrip}}.\\
  \text{(c) Consider the sets {\small $\{\mathcal{H}_l\}_{l=1}^N$} defined by the connected components of {\small$\Omega\setminus\cup_{n,k}\mathscr{S}_{n,k}$}}.\\
  \text{(d) Assume that each set $\mathcal{H}_l$ has a single associated target $\alpha_{m(l)}$}.
  \end{cases}
\end{equation}

 \begin{lemma}[Fundamental Compression Lemma]\label{FUNLEMMA}
 
  Let $\Omega\subset\mathbb{R}^d$ be bounded, and consider hyperplanes, strips and sets as in \eqref{H}.  Assume that the set of targets $\{\alpha_m\}_{m=1}^M\subset \mathbb{R}^d$ fulfil $\alpha^{(1)}_{m}\neq\alpha^{(1)}_{m'}$ if $m\neq m'$.
  %usepackage{xfrac} sfrac
    %´in \ref{UAStep1}.
   For any $T>0$ and any $0<\eta<\sfrac{1}{N}$ small enough, there exist piecewise-constant controls $A,W\in L^\infty((0,T);\mathbb{R}^{d\times d})$ and $b\in L^\infty((0,T);\mathbb{R}^{d})$ such that
  \begin{subequations}
   \begin{align*}
     &\phi_{1}(\mathcal{H}_l;A,W,b)^{(1)}\subset[\alpha_{m(l)}^{(1)}-C\eta,\alpha_{m(l)}^{(1)}+C\eta]\hspace{0.35cm}\qquad \text{for every }1\leq l\leq N,\\
     %&\phi_{1}(\mathcal{H}_l;A,W,b))^{(k)}\subset\left[-\mathcal{O}(\eta),\mathcal{O}(\eta)\right]\hspace{2.3cm}\qquad \text{for every }l=1,...,N\text{ and any }2\leq k\leq d,\\
     &\mathrm{diam}_{(k)}\phi_{1}(\Omega;A,W,b)<\eta\hspace{3.2cm}\qquad \text{for every }2\leq k\leq d,\\
     & \phi_{1}(\Omega;A,W,b)\subset \mathbb{B}(0,K)
   \end{align*}
  \end{subequations}
  with $K$ and $C$ independent of $N,N_\Gamma$ and $\eta$. Furthermore, fixing $T=1$, one has that:
    \begin{subequations}\label{eqsFunlema}
   \begin{align*}
     &\|A\|_{L^\infty((0,1);\mathbb{R}^{d\times d})}=1,\qquad\|b\|_{L^\infty((0,1);\mathbb{R}^{d})}\lesssim_{\Omega} N ,\\
     &\|W\|_{L^\infty((0,1);\mathbb{R}^{d\times d})}\lesssim_{\Omega,\alpha,d} N\left(\frac{1}{\eta}+\frac{1}{\zeta}+\log\left(\frac{1}{\eta\zeta}\right)\right).\\
     &\text{The number of discontinuities of the controls }A,W,b\text{ is of the order of }N.
   \end{align*}
  \end{subequations}
  Here, the subscript $\lesssim_\alpha$ denotes the dependence on the set $\{\alpha_m\}_{m=1}^M$.
 \end{lemma}

 For the sake of readability, the proof is postponed to Appendix \ref{PROOF}. The proof uses a concatenation of flows as the ones described in Subsection \ref{funop} to achieve the desired result.
 
 %Lemma \ref{FUNLEMMA}, for being useful later on in the simultaneous controllability, we will need to impose that $C\eta<\min_{m,m'}|\alpha_m^{(1)}-\alpha_{m'}^{(1)}|$.
 
  Note that, Lemma \ref{FUNLEMMA}, aside of compressing, for $\eta$ small enough, also contains a generalization of Theorem \ref{TH1} for the hyperrectangles.\footnote{observe that, now, the sets are classified by strips, this is analogous to the result of Theorem \ref{TH1}.}

    \medskip
 \item[\namedlabel{UAStep3}{\textbf{Step 3}}] \textbf{Simultaneous Control.}
        For $\eta$ small enough, one can find a collection of $M$ sets, $\{Q_m\}_{m=1}^M$ such that:
        $$\mathrm{diam}(Q_m)\leq \max\{1,2C\}\eta, \quad 1\leq m\leq M$$
        $$\phi_T(\mathcal{H}_l)\subset Q_m  \qquad \text{ if the target associated to $\mathcal{H}_l$ is $\alpha_m$}.$$
        
        Therefore we can apply Theorem \ref{TH2}. Let us denote the flow from Theorem \ref{TH2} by $\psi_T$. The controllability time of Theorem \ref{TH2}, $T_M$ does not depend on $h$  and by Grönwall inequality to obtain that 
        $$\nu:=\max_{1\leq m\leq M}\sup_{x\in \psi_{T_M}(Q_m)}|x-\alpha_m|\leq C_M\eta$$
        where $C_M$ is a constant that depends on $M$ and on the set $\{\alpha_m\}_{m=1}^M$ but not on $h$. Hence $\nu\sim _{M,\alpha}\eta$. Furthermore, the extra number of switches will depend on $M$ as in Theorem \ref{TH2} but not on $N$.
        {\color{black}
        
        We summarize the above discussion in a lemma that will be used in the next section.
        
        \begin{lemma}\label{controllingCharacteristics}
          Let $\Omega\subset\mathbb{R}^d$ be bounded, and consider hyperplanes, strips and sets as in \eqref{H}. Let $f$ be a simple function as in \eqref{simple}.
        For all $\nu >0$ small enough, there exist piecewise-constant controls $A,W\in L^\infty((0,T);\mathbb{R}^{d\times d})$ and $b\in L^\infty((0,T);\mathbb{R}^{d})$ such that the associated flow of \eqref{CTNN} satisfies
         \begin{align*}
\begin{cases}
 |\phi_T(x;A,W,b)-f(x)|<\nu\quad &x\in\Omega_c\\
|\phi_T(x;A,W,b)|<K\quad &x\in\Omega_h.
\end{cases}
\end{align*}
        Moreover, control norms are         with $K$ and $C$ independent of $N,N_\Gamma$ and $\eta$. Furthermore, fixing $T=1$,  one has that:
    \begin{subequations}%\label{eqsFunlema}
   \begin{align*}
     \|A\|_{L^\infty((0,1);\mathbb{R}^{d\times d})}&=1,\qquad\|b\|_{L^\infty((0,1);\mathbb{R}^{d})}\lesssim_{\Omega,\alpha,d} N,\\
     \|W\|_{L^\infty((0,1);\mathbb{R}^{d\times d})}&\lesssim_{\Omega,\alpha,d,M} N\left(\frac{1}{\nu}+\frac{1}{\zeta}+\log\left(\frac{1}{\eta\zeta}\right)\right).\\
     &\text{The number of discontinuities of the controls }A,W,b\text{ is of the order of }N.
   \end{align*}
  \end{subequations}
  
          \end{lemma}

        }
  \item[\namedlabel{UAStep4}{\textbf{Step 4}}] \textbf{Estimates.}       
        Recalling \eqref{enh}\eqref{ennu}\eqref{enH} plus the requirement that $\eta<N^{-1}$ applied to \eqref{eqsFunlema} we obtain that:
    \begin{subequations}%\label{eqsFunlema2}
   \begin{align*}
     \|A\|_{L^\infty((0,1);\mathbb{R}^{d\times d})}=1,\qquad\|b\|_{L^\infty((0,1);\mathbb{R}^{d})}\lesssim_{\Omega,d} \epsilon^{-2d(d-1)}  ,\qquad\|W\|_{L^\infty((0,1);\mathbb{R}^{d\times d})}\lesssim _{\Omega,f,d}\epsilon^{-4d(d-1)}.\\
     \text{The number of switches of the controls }A,W,b\text{ is of the order of }\epsilon^{-2d(d-1)}.
   \end{align*}
  \end{subequations}

\end{description}

\end{proof}

{\color{black}

\begin{remark}[General simple functions]\label{fractalremark}
 In the previous theorem we have assumed that the perimeter of the characteristic sets is finite. However, the same proof works for certain sets with infinite perimeter.
 
 Let $D$ be the upper box-counting dimension of $\Gamma$ (see \cite[Chapter 2]{falconer2004fractal}, see also \cite[Chapter 1]{bishop2017fractals}). As in the proof above, let $N_\Gamma(h)$ be the number of hypercubes of side $h$ needed to cover the boundary $\Gamma$, the upper box-counting dimension is defined as 
 $$ D:=\limsup_{h\to0}\frac{\log N_{\Gamma}(h)}{\log\left(\frac{1}{h}\right)}.$$
 Hence, by definition one has that:
 $$N_\Gamma \lesssim h^{-D}\text{ as }h\to 0. $$
 Therefore, the estimates  \eqref{FinalT}\eqref{Finalb} and \eqref{FinalS} are replaced by
 \begin{equation*}
  \|W\|_{L^\infty}\lesssim \epsilon^{-\frac{4Dd}{d-D}},\qquad \|b\|_{L^\infty}\lesssim \epsilon^{-\frac{-2dD}{d-D}} \qquad\text{ as }\epsilon\to 0
 \end{equation*}
  \begin{equation*}
   \text{The number of switches of $A,W,b$ will be of the order of $\epsilon^{-\frac{2dD}{d-D}} $.}
 \end{equation*}
 Note that the estimates blow up when $D=d$, in such case, we cannot provide an estimate. %the reason comes from a technicality in the proof. If $D=d$, then  the measure of $\Omega_h$ will tend to a constant as $h\to 0$. In such case, we cannot provide an estimate.
 Boundaries with dimension $D=d$ exist, for instance, the Mandelbrot set \cite{shishikura1998hausdorff}. Generally speaking, one can have a Jordan curve that has dimension equals to the ambient dimension \cite{osgood1903jordan}.%,buff2012quadratic}.
 
 Taking the box-counting dimension as a measure of complexity of the simple function we see that as the characteristic sets are more complex the cost of the control grows.

\end{remark}

 }

\section{Control and simultaneous control of Neural transport equations}\label{Transport}

In this section we use the  techniques of the proof of Theorem \ref{THsimple} to control a Neural transport equation (NTE) of the form
\begin{equation}\label{NTEq}
\begin{cases}
  \partial_t\rho+\mathrm{div}_x\big[\left(W(t)\boldsymbol{\sigma}(A(t)x+b(t))\right)\rho\big]=0\qquad (x,t)\in\mathbb{R}^d\times(0,T)\\
 \rho(0)=\rho^0\in C_c(\mathbb{R}^d,\mathbb{R}^+)
\end{cases}
\end{equation}%C_b^0(\Omega;\mathbb{R}^+)
where $C_c(\mathbb{R}^d,\mathbb{R}^+)$ stands for the space of compactly supported continuous nonnegative functions. Our proof of the Universal Approximation Theorem can be interpreted in the context of the simultaneous control of these NTEs. This will also give rise to a transport formulation of the classification problem.

In this section the approximation will be measured in the sense of the  Wasserstein-1 distance:
\begin{deff}
 Let $\mu,\nu\in \mathcal{P}_c(\mathbb{R}^d)$ be probability measures. The Wasserstein-1 distance $\mathcal{W}_1(\mu,\nu)$ is defined by as:
  \begin{equation*}
   \mathcal{W}_1(\mu,\nu)=\sup_{Lip(g)\leq 1} \left\{\int_{\mathbb{R}^d} gd\mu-\int_{\mathbb{R}^d} gd\nu\right\}
  \end{equation*}
where $Lip(g)\leq 1$ stands for the class of Lipschitz functions with Lipschitz constant less or equal than $1$.
\end{deff}

 Wasserstein distances play a central role in the theory of  optimal transport (see  %\cite[Chapter 1 and 3]{santambrogio2015optimal},
 \cite[Chapter 5]{villani2008optimal} 
  for a general reference in this subject).

\subsection{Control of the Neural Transport equation.}
\textcolor{white}{.}\newline
Since the vector field $W\boldsymbol{\sigma}(Ax+b)$ of the  Neural Transport equation \eqref{NTEq} is Lipschitz, it
%\begin{equation}\label{NTEq}
%\begin{cases}
%  \partial_t\rho+\mathrm{div}_x\big[\left(w(t)\boldsymbol{\sigma}(A(t)x+b(t))\right)\rho\big]=0\\
% \rho(0)=\rho^0\in C_b^0(\Omega;\mathbb{R}^+).
%\end{cases}
%\end{equation}
preserves the mass for each set $\mathcal{H}$ along the characteristic flow $\phi_T(\mathcal{H})$
 $$\int_{\phi_T(\mathcal{H})}\rho(T)dx=\int_\mathcal{H}\rho^0dx.$$
The control of the \eqref{NTEq} can be done by controlling its characteristics plus a compression argument as done in Lemma \ref{FUNLEMMA} in Theorem \ref{THsimple}.

We will consider target measures $\rho^*$ that are finite combinations of Dirac masses:
\begin{equation}\label{targetrho}
 \rho^*=\sum_{m=1}^M \beta_m\delta_{\alpha_m}
\end{equation}
with $\beta_m>0$ and $\alpha_m\in\mathbb{R}^d$.
 \begin{theorem}\label{TH5}
  Let $T>0$, $d\geq 2$, $\sigma$ be as \eqref{relu} and $\rho^*$ as in \eqref{targetrho} be a target positive measure satisfying:
  \begin{align*}
   \int_{\mathbb{R}^d} d\rho^*=\int_{\mathbb{R}^d}\rho^0dx=1.
  \end{align*}
  Then, for every $\epsilon>0$, there exist piecewise-constant control functions $W,A\in L^\infty((0,T);\mathbb{R}^{d\times d})$ and $b\in L^\infty((0,T);\mathbb{R}^d)$ such that the solution of \eqref{NTEq} satisfies:
  $$ \mathcal{W}_1(\rho(T),\rho^*)<\epsilon$$
  where $\mathcal{W}_1$ is the Wasserstein-1 distance. 
  Furthermore, one has
  $$\|W\|_{L^\infty((0,T),\mathbb{R}^{d\times d})}\lesssim_{\rho^0,\rho^*} \epsilon^{-1}\qquad \text{ as }\epsilon\to0.$$
  Moreover, $\|b\|_{L^\infty((0,T),\mathbb{R}^{d})}$ , $\|A\|_{L^\infty((0,T),\mathbb{R}^{d\times d})}$ and the number of discontinuities of $W,A$ and $b$ are bounded independently of $\epsilon$.
 \end{theorem}
\begin{proof}
We proceed in several steps:

\begin{enumerate}[leftmargin=0.65cm,font=\bfseries]
 \item \textbf{Assigning targets to the initial datum.}
 Consider $F:\mathbb{R}\to\mathbb{R}$ be defined as:
 \begin{equation}\label{distrfun}
 F(s)=\int_{-\infty}^s\left(\int_{\mathbb{R}^{d-1}}\rho^0dx^{(2)}...dx^{(d)}\right)dx^{(1)}.
 \end{equation}
 Using the fact that $\rho^0\in C_c(\mathbb{R}^d,\mathbb{R}^+)$, there exist two  real numbers, name them $c_0$ and $c_M$, such that
 $$\forall s\leq c_0,\quad F(s)=0,\qquad \forall s\geq c_M,\qquad F(s)=1.$$
 Since $\rho^0\in C_c(\mathbb{R}^d,\mathbb{R}^+)$, $F$ is continuous and monotone increasing, i.e. we can find $M-1$ real numbers such that:
 $$F(c_m)-F(c_{m-1})=\beta_m\qquad 1\leq m\leq M.$$
 Hence, we will  divide $\rho^0$ with $M-1$ hyperplanes of the form:
 $$\left\{x^{(1)}=c_m\right\},\quad 1\leq m\leq M-1$$
 (see Figure \ref{ts2}).
   \begin{figure}%[h!]
  \includegraphics[scale=0.4]{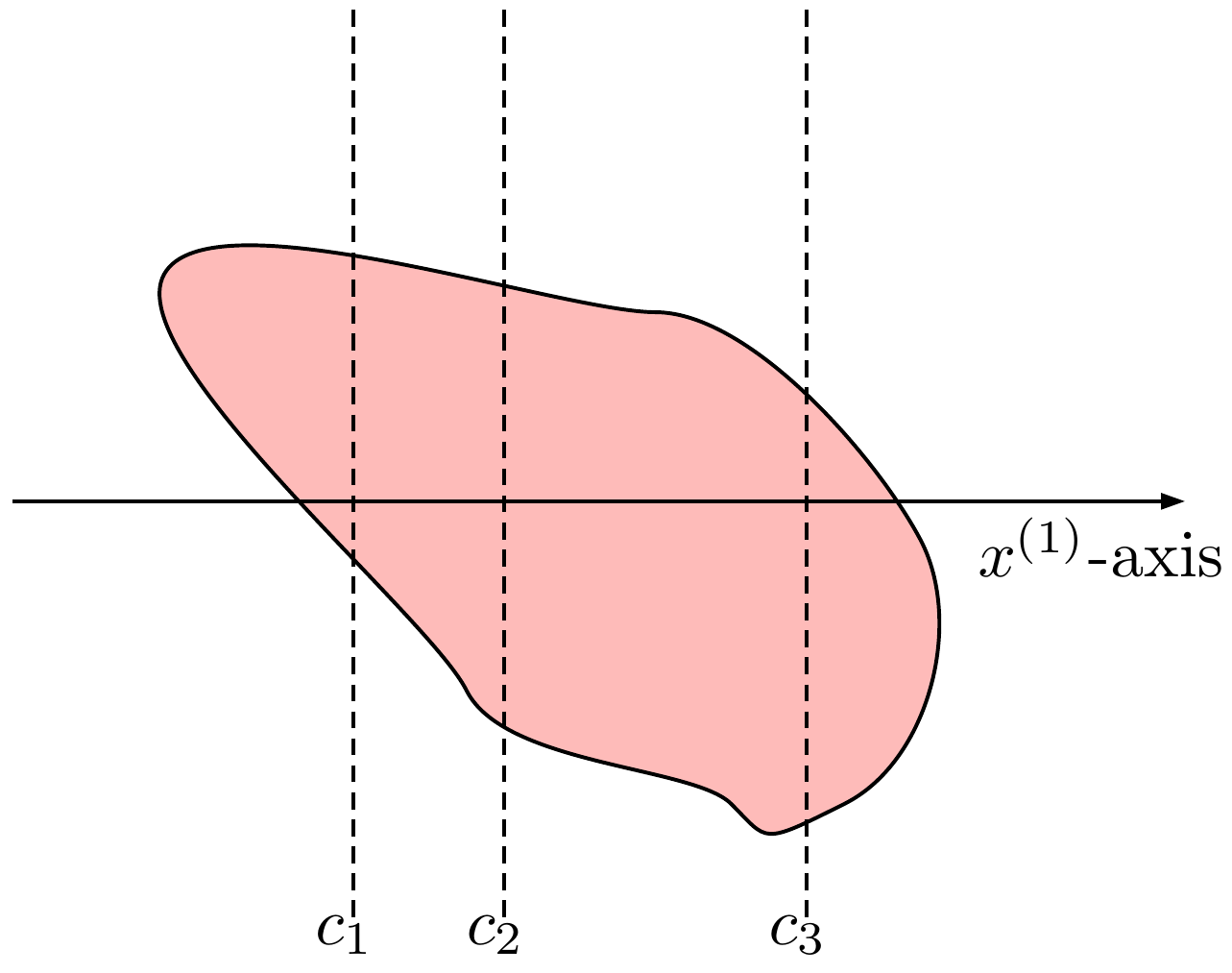}
  \caption{The support of the initial data $\rho^0$ and its division by hyperplanes.}\label{ts2}
 \end{figure}
  The intention is to find a flow such that, by means of the characteristics, can transport all the mass $\beta_m$ contained between the hyperplanes $\{x^{(1)}=c_m\}$ and $\{x^{(1)}=c_{m-1}\}$ to the target $\alpha_m$. As in the previous section, due to the continuity with respect of the initial data of the characteristics, this task will have to be carried out in an approximate manner.
  
 \item \textbf{Compression and simultaneous control.}
 Similarly as in \ref{UAStep1} of Theorem \ref{THsimple} ,  for every $\epsilon>0$, we will consider a $d$-dimensional strip around each hyperplane $\{x^{(1)}=c_m\}$ for $m=1,...,M-1$ of the form
 $$[c_m-\zeta,c_m+\zeta]\times\mathbb{R}^{d-1}$$
 for some $\zeta>0$ in a way that:
 $$ F(c_m+\zeta)-F(c_m-\zeta)\leq\frac{\epsilon}{M}\qquad 1\leq m\leq M-1.$$
 In this way, we ensure that the mass lying in the strips is of the order of $\epsilon$.

 %The mass lying at the $d$-dimensional strips will not be approximately controlled to thieir  target. However, U
 Using Lemma \ref{controllingCharacteristics}, we can ensure that the mass lying in  the $d$-dimensional strips will lie in a bounded set independently of $\zeta$ and consequently independently of $\epsilon$.
 
 Let $\mathcal{H}\subset\mathbb{R}^d$ be a subset, and let $\rho(t)$ be the solution of \eqref{NTEq} for certain $A,W,b$, let $\phi_T$ be the solution of \eqref{CTNN} for the same controls. Due to the divergence structure of Equation \eqref{NTEq} and that the field $W\boldsymbol{\sigma}(Ax+b)$ is Lipschitz, one has that the mass is preserved along the characteristics and
 $$\int_{\phi_T(\mathcal{H})}\rho(T)dx=\int_\mathcal{H}\rho^0dx.$$

 {\color{black}
 
 Define $\mathcal{H}_m=\mathrm{supp}(\rho^0)\cap ([c_{m-1}+\zeta,c_m-\zeta]\times\mathbb{R}^{d-1})$. Now we may apply Lemma \ref{controllingCharacteristics} to obtain the approximate controllability of all characteristics in every $\mathcal{H}_m$ while having the mass allocated in the strips has been transported in a bounded set $\mathbb{B}(0,K)$ with $K$ independent of $\zeta$ and $\nu$. Moreover, choose $\nu$ small enough so that 
 $$\mathbb{B}(\alpha_m,\nu)\cap\mathbb{B}(\alpha_{m'},\nu)=\varnothing\qquad \text{if }m\neq m'.$$
  %($\nu$ PROU PETIT perquè les boles siguin disjuntes)

 %We compress our dataset as in the previous theorem, and we control most of the characteristics simultaneously (Theorem \ref{TH3} Lemma \ref{FUNLEMMA})  to be on the setting to apply the simultaneous controllability Theorem \ref{TH2}. 

%Apply Lemma \ref{controllingCharacteristics}

Now, the objective is to quantify, in terms of the Wasserstein distance, the proximity of $\rho(T)$ to $\rho^*$.  The conclusion will follow from a straightforward computation using the triangular inequality with a suitable $\rho^\epsilon$. We will find a $\rho^\epsilon$ of the form
$$\rho^\epsilon=\sum_{m=1}^M\overline{\beta}_m\delta_{\alpha_m}+\sum_{m=1}^M\underline{\beta}_m\delta_{z_m}.$$
Let us proceed with the construction of $\rho^\epsilon$.
\begin{enumerate}[leftmargin=0.6cm,font=\bfseries]
\item \textbf{Choice of the masses $\boldsymbol{\{(\overline{\beta}_m,\underline{\beta}_m)\}_{m=1}^M}$.} Consider $\xi_m$ to be the minimizer of the following quantity
\begin{equation}\label{minimizer}
\xi_m:=\min_{0<\xi\leq \nu}\left|\int_{\mathbb{B}(\alpha_m,\xi)}\rho(T)dx-\beta_m\right|,\qquad 1\leq m\leq M.
\end{equation}
Observe that the function $g_m:\mathbb{R}^+\to[0,1]$ defined as
$$g_m(\xi)=\int_{\mathbb{B}(\alpha_m,\xi)}\rho(T)dx$$
is an increasing function. Therefore, if the minimizer \eqref{minimizer} is unique and if \eqref{minimizer} satisfies $\xi_m<\nu$, then we have that $g_m(\xi_m)=\beta_m.$
 We, then, define
 $$\overline{\beta}_m:=g(\xi_m)=\int_{\mathbb{B}(\alpha_m,\xi_m)}\rho(T)dx.$$
 By definition, $\overline{\beta}_m>0$. Moreover, taking into consideration that the mass transported, through the application of Lemma \ref{controllingCharacteristics}, into $\mathbb{B}(\alpha_m,\nu)$ is at least $\beta_m-\sfrac{2\epsilon}{M}$,  one has that
 $$\beta_m-\frac{2\epsilon}{M}\leq \overline{\beta}_m\leq \beta_m\qquad 1\leq m\leq M$$
 and
 $$1-2\epsilon\leq\sum_{m=1}^M\overline{\beta}_m\leq 1.$$
 Now define the remainders of mass $\underline{\beta}_m$ as follows
 $$ \underline{\beta}_m=\beta_m-\overline{\beta}_m,\qquad 1\leq m\leq M$$
 note that, by construction, one has $$0\leq\underline{\beta}_m\leq\frac{2\epsilon}{M},\quad 1\leq m \leq M.$$
\item \textbf{Choice of the locations $\boldsymbol{\{z_m\}_{m=1}^M}$.} Let $\mathcal{B}=\bigcup_{m=1}^M\mathbb{B}(\alpha_m,\xi_m)$ and define $G:\mathbb{R}\to \mathbb{R}$ as
  $$G(s)=\int_{-\infty}^s\left(\int_{\mathbb{R}^{d-1}}(1-\chi_{\mathcal{B}})\rho(T)dx^{(2)}...dx^{(d)}\right) dx^{(1)}.$$
  As in step 1 of this proof, there exist real numbers $c_0'$ and $c_M'$ such that:
  $$\forall s\leq c_0',\quad G(s)=0,\qquad \forall s\geq c_M',\qquad G(s)=1$$
  and as before, we find $M-1$ real numbers such that  
  $$ G(c_{m}')-G(c_{m-1}')=\underline{\beta}_m,\qquad\mathcal{T}_m= [c_{m-1}',c_{m}']\times \mathbb{R}^{d-1}.$$
  Take a collection of $\{z_m\}_{m=1}^M$ such that for every $m\in\{1,...,M\}$, satisfies $z_m\in \mathcal{T}_m \cap \mathbb{B}(0,K).$
\end{enumerate}
  }
{\color{black}
  \noindent Then, choosing $\nu=\epsilon$ the result follows by the triangular inequality.
  \begin{align*}
   \mathcal{W}_1(\rho(T),\rho^\epsilon)=&\sup_{Lip(g)\leq 1}\left|\int_{\mathbb{R}^d}g(x)\rho(T)dx-\sum_{m=1}^M\overline{\beta}_mg(\alpha_m)-\sum_{m=1}^M\underline{\beta}_mg(z_m)\right|\\
   \leq&\sup_{Lip(g)\leq 1}\left|\sum_{m=1}^M\int_{\mathbb{B}(\alpha_m,\xi_m)}(g(x)-g(\alpha_m))\rho(T)dx\right|+\sup_{Lip(g)\leq 1}\left|\sum_{m=1}^M\int_{\mathcal{T}_m}(g(x)-g(z_m))\rho(T)dx\right|\\
   \leq&2\nu+2K\epsilon\\
   \mathcal{W}_1(\rho^*,\rho^\epsilon)=&\sup_{Lip(g)\leq 1}\left|\sum_{m=1}^M\overline{\beta}_mg(\alpha_m)+\sum_{m=1}^M\underline{\beta}_mg(\alpha_m)-\sum_{m=1}^M\overline{\beta}_mg(\alpha_m)-\sum_{m=1}^M\underline{\beta}_mg(z_m)\right|\\
   =&\sup_{Lip(g)\leq 1}\left|\sum_{m=1}^M\underline{\beta}_mg(\alpha_m)-\underline{\beta}_mg(z_m)\right|\leq 4K\epsilon.
  \end{align*}
  
  \item \textbf{Estimates.}
  The control cost, also depends on the initial density function $\rho^0$. Define the map $\epsilon_c:\mathbb{R}^+\to\mathbb{R}^+$ as
  $$\epsilon_c(\zeta)=\int_{c-\zeta}^{c+\zeta}\left(\int_{\mathbb{R}^{d-1}}\rho^0dx^{(2)}...dx^{(d)}\right)dx^{(1)}.$$
  Note that $\epsilon_c(0)=0$, differentiating with respect to $\zeta$ and evaluating at $0$ we obtain that:
  $$\frac{d}{d\zeta}\epsilon_c(0)=2\int_{\mathbb{R}^{d-1}}\rho^0(c,x^{(2)},...,x^{(d)})dx^{(2)}...dx^{(d)}.$$
  Formally, the derivative of inverse map of $\epsilon_c$, which we denote by $\zeta_c:\mathbb{R}^+\to\mathbb{R}^+$, at zero fulfills:
  $$\frac{d}{d\epsilon}\zeta_c(0)=\left(2\int_{\mathbb{R}^{d-1}}\rho^0(c,x^{(2)},...,x^{(d)})dx^{(2)}...dx^{(d)}\right)^{-1}.$$
  Therefore, it suffices to require that
  \begin{equation}\label{constant}
   \zeta\leq\min_{m\in\{1,...,M-1\}}\left\{\left(2\int_{\mathbb{R}^{d-1}}\rho^0(c_m,x^{(2)},...,x^{(d)})dx^{(2)}...dx^{(d)}\right)^{-1},1\right\} \epsilon\qquad \text{as }\epsilon \to 0. 
  \end{equation}
  From Lemma \ref{controllingCharacteristics} we obtain that:
  $$\|W\|_{L^\infty}\lesssim_{\rho^*} \epsilon^{-1} \quad\text{ as }\epsilon\to 0.$$
  Furthermore,  $\|A\|_{L^\infty}=1, \|b\|_{L^\infty}\leq C $ and the number of discontinuities of the controls is bounded independently of $\epsilon$ but dependent on the target $\rho^*$.
  }
 \end{enumerate}
 \end{proof}
 \begin{remark}[More general target configuarations]
 We stated the theorem by setting targets that are Dirac masses, but, it is well known that, in particular, one can approximate any compactly supported probability measure by a finite number of Dirac masses.% Moreover, this result can also be understood and extended to measure valued solutions of \eqref{NTEq}.
 \end{remark}

 {\color{black}

 \begin{remark}[Influence of the initial condition on the control cost]
  
  The proof is valid for more general settings, namely for $\rho^0\in L^1_c(\mathbb{R}^d;\mathbb{R}^+)$, where by $ L^1_c(\mathbb{R}^d;\mathbb{R}^+)$ we understand the $L^1$ compactly supported nonnegative functions. However, the quantification of the control cost depends on the singularities of $\rho^0$. 
  
  In \eqref{constant} one can observe that as the density at a point $c_m$ increases, $\zeta$ should be smaller. If $\rho^0\in L^1_c(\mathbb{R}^d)$ the cost will depend on the ``\texttt{strength}'' of the singularities of $\rho^0$. The precise quantification of the control cost depending on the regularity of $F$ (or singularity of $\rho^0$) will not be treated in this article. For more details and a fine analysis of singularities we refer to \cite[Chapter I and II]{jaffard1996wavelet}.
 \end{remark}
  \begin{remark}
  %We have assumed the initial data to be in $C_c(\mathbb{R}^d;\mathbb{R}^+)$.
  Due to the uniqueness of the solution of the underlying characteristic system, we cannot consider, in general, initial data that are Dirac deltas. More precisely, by the uniqueness of the characteristics, it is impossible to bring the initial datum $\rho^0=\delta_0$ to any configuration of the type $\rho^*=\lambda \delta{\alpha_1}+(1-\lambda)\delta{\alpha_2}$ for any $\lambda\in(0,1)$ and any $\alpha_1,\alpha_2\in \mathbb{R}^d$.
 \end{remark}
 
 \begin{remark}
 In Theorems \ref{THsimple} and \ref{TH5}, to object of study was the  input output map given by the flow $\phi_T$ associated to \eqref{CTNN}.  In Theorem \ref{TH5} we use the same principles for another objective, to control in approximate manner from a given density function to another. 
 
 In Theorem \ref{THsimple} and Remark \ref{fractalremark} we observed how the geometry of the supports plays a crucial role in the control cost. However, in Theorem \ref{TH5}, since we are controlling a single probability density, we do not observe such dependence. In contrast, in \eqref{constant} we observe how high concentrations of mass can increase the cost of control.

 \end{remark}

 }

 \begin{remark}
  The restriction of $d\geq 2$ comes from the limitation pointed out in Remark \ref{d=1}. However, the transport equation 
  $$\partial_t \rho+\mathrm{div}_x[V(x,t)\rho]=0 $$
  with $V$ as a control, can be controlled in the one-dimensional case approximately with respect to the Wasserstein-1 distance. It would be enough to design appropriate locations of attractors and repulsors depending on the mass distribution of $\rho^0$. This would allow concentrating the mass in a finite number of points. Later, one needs to design a dynamical system that brings each mass approximately on the approximation of the target.
  Neural Transport Equations can achieve this in dimension $d\geq 2$ just with the controls $A,W$ and $b$.
 \end{remark}
 
 \begin{remark}
  If the target Dirac masses have the same mass, one does not need the simultaneous controllability. Given $M$ initial data $\{x_m\}_{m=1}^M$ and $M$ targets $\{\alpha_m\}_{m=1}^M$, it is enough that for every $\epsilon>0$ one can find controls $A,W$ and $b$ such that
  $$\forall j\in\{1,...,M\}\quad \exists!m:\quad \phi_T(x_m;A,b,W)\in \mathbb{B}(\alpha_j,\epsilon).$$
  %where $\mathbb{B}(\alpha,\epsilon)$ is the ball of radius $\epsilon$ centered in $\alpha\in\mathbb{R}^d$.
  Note that, one can also approximate any continuous compactly supported function by a finite number of Dirac masses whose support is disjoint and all of them have the same mass. 
 \end{remark}
 
 \begin{remark}[Other Wasserstein metrics]\label{villani}
 The definition we gave of the Wasserstein-1 metric relies on the Kantorovich-Rubinstein Theorem \cite[Theorem 1.14]{villani2003topics}. In the Euclidean space, a Wasserstein-p distance is defined as:
 \begin{equation*}
  \mathcal{W}_p(\mu_1,\mu_2)=\inf_{\gamma\in\Gamma(\mu_1,\mu_2)}\left\{\int_{\mathbb{R}^d\times\mathbb{R}^d}|x-y|^pd\gamma(x,y)\right\}
 \end{equation*}
 where $\Gamma(\mu_1,\mu_2)$ denotes the set of all measures $\gamma$ that satisfy $\mu_1=\int_{\mathbb{R}^d}\gamma(x,y) dx$ and $\mu_2=\int_{\mathbb{R}^d}\gamma(x,y) dy$. Theorem \ref{TH5} states the result in the Wasserstein-1 metric, however, since the supports of the final datum $\rho(T)$ and the target are compact, one has equivalence of all $\mathcal{W}_p$ distances \cite[Chapter 7, Section 7.1.2]{villani2003topics}.
 \end{remark}

 \begin{remark}\label{impUA}
  For the ReLU case, and for any Lipschitz nonlinearity, fixing controls $A,W$ and $b$, we see that the characteristics, at most, have an exponential decay. This is the impediment for Theorem \ref{TH3} to be a universal approximation theorem for functions in $L^2(\mathbb{R}^d,\mathbb{R}^d)$.
 \end{remark}

\subsection{Simultaneous control of Neural transport equations and classification}
 \textcolor{white}{.}\newline
 
 The scalar transport equation does not allow distinguishing among labels. One should take a vectorial structure in order to have a transport formulation for classification.
 
 Let us consider $M$ classes, and $M$ compactly supported probability densities $\rho_m$ for  $m=1,...,M$. Moreover, we assume, as in Section \ref{Smatset}, that for every $x$ there exists, at most, a unique label $y$. This implies that the supports of the probability measures $\rho_m$ are disjoint:
 $$ \mathrm{supp}(\rho_m)\cap \mathrm{supp}(\rho_{m'})=\varnothing, \qquad \text{if }m\neq m'.$$
 Theorem \ref{THsimple} is based on approximating a simple function. Therefore, in Theorem \ref{THsimple} we have already treated the problem of sending uniform disjoint compactly supported distributions to precisely located Dirac deltas (in an approximate manner). 
 Combining the ideas of Theorem \ref{THsimple} and \ref{TH5} one has the following:  
Let $M\in\mathbb{N}$ be a natural number and $\Omega\subset\mathbb{R}^d$ for $d\geq 2$ and consider the system:
\begin{equation}\label{contsys}
\begin{cases}
   \partial_t\rho_m+\mathrm{div}_x\left[\left(W(t)\boldsymbol{\sigma}(A(t)x+b(t))\rho_m\right)\right]=0,\quad &m\in\{1,...,M\}\\
   \rho_m(0)=\rho_m^0\in C^0_c(\mathbb{R}^d),\quad &m\in\{1,...,M\}.
\end{cases}
\end{equation}
%C_b^0(\Omega;\mathbb{R}^+)
\begin{theorem}\label{TH6}
 Let $T>0$ and $\rho_m^*$ be as in \eqref{targetrho} the target measure for the the $m$-th equation.
 Assume that the initial conditions satisfy:
 \begin{align*}
 \mathrm{supp}(\rho_m^0)\cap \mathrm{supp}(\rho_{m'}^0)=\varnothing\quad \text{if } m\neq m',\qquad\mathrm{Per}\left(\partial \mathrm{supp}(\rho^0_m)\right)<+\infty,\quad m\in\{1,...,M\}
   \end{align*}
 where by $\mathrm{Per}$ we understand the perimeter. In addition, assume that the target functions satisfy:
 \begin{align*}
   \mathrm{supp}(\rho_m^*)\cap \mathrm{supp}(\rho_{m'}^*)=\varnothing\quad \quad\text{if }m\neq m',\qquad\int \rho_m^*=\int \rho_m^0dx=1\quad\qquad m\in\{1,...,M\}.
 \end{align*}
 Then, for any $\epsilon>0$, there exist controls $W,A\in L^\infty((0,T),\mathbb{R}^{d\times d})$ and $b\in L^\infty((0,T),\mathbb{R}^{d})$ such that the solution of \eqref{contsys} satisfies:
 \begin{equation}
  \mathcal{W}_1(\rho_m(T),\rho_m^*)<\epsilon\quad  m\in\{1,...,M\}
 \end{equation}
 where $\mathcal{W}_1$ is the Wasserstein-1 distance. 

\end{theorem}
   The proof follows combining the arguments of Theorem \ref{THsimple} and \ref{TH5}. The only difference with respect to Theorem \ref{TH5} is that all initial data have to be separated by several extra hyperrectangles as done in \ref{UAStep1} in Theorem \ref{THsimple}. This step was not needed for controlling just one transport equation, but it had to be done in the universal approximation Theorem \ref{THsimple}. This implies that, in this case, the control cost will depend both, on the geometry of the supports of each density function, and on the concentration of mass of such density functions.

   The classification interpretation is the following. Let $\mathcal{S}=\{S_m\}_{m=1}^M$ be a partition of $\mathbb{R}^d$ and consider system \eqref{contsys}. The classification problem consists of: for every $\epsilon>0$, being able to find controls (that depend on $\epsilon$) $W,A\in L^\infty((0,T),\mathbb{R}^{d\times d})$ and $b\in L^\infty((0,T),\mathbb{R}^{d})$ such that the solution of \eqref{contsys} satisfies:
   $$ \int_{S_m}\rho_m(T)dx=1-\epsilon,\qquad \int_{\mathbb{R}^d}x\rho_m(T)dx\in S_m\qquad  m\in\{1,...,M\}$$
   Roughly speaking, the classification problem, in terms of transport equations, is equivalent to send most of the mass of each probability measure to prefixed sets and to control its expectation to the prefixed set.
   
   \begin{remark}
    Note that, since the supports are disjoint, if we consider $\rho=\sum_{m=1}^N \rho_m$, the controls $A,W$ and $b$ of Theorem \ref{TH6} control also the scalar equation \eqref{NTEq}.
   \end{remark}
   {\color{black}
   \begin{remark}[Robustness with respect to large number finite dimensional samples]
    The controls obtained from Theorem \ref{TH6} are robust with respect to large number finite dimensional samples that follow the probabilities distributions of the initial data of system \eqref{contsys}. This implies that, given an error $\epsilon$, the control cost of classification for finite dimensional samples as in Theorem \ref{TH1} does not blow up when the cardinality of sample tends to infinity. 
   \end{remark}
   
   \begin{remark}
    Note that Remark \ref{fractalremark} also applies in this case for having more irregular boundaries. However, if the boundaries of the supports of each distribution are at positive distance, one does not need to consider a fine meshing of the boundary. 
   \end{remark}

   }

\section{Conclusions and Perspectives}\label{Sconcl}

We now present a number of conclusions and open problems.

\begin{enumerate}[font=\bfseries]

 \item \textbf{The role of the activation function.}
 The key property that allows to prove all the results  in this paper is that the activation functions are able to leave one half-space invariant while moving the other half. As mentioned in the Introduction, the simultaneous controllability is a rare property for dynamical systems. In particular, it can never occur in linear systems when all trajectories solve the same system. Our methods  exploit this nonlinear feature of activation functions to obtain  simultaneous control properties of NODEs that lead to the desired results in the Machine Learning context.

 \medskip

  \item \textbf{Algorithmic complexity.} All proofs developed in this paper are algorithmic. 

 %\smallskip
 The proofs in Theorem \ref{TH1} and \ref{TH2} require a number of steps of the order of $\mathcal{O}(N)$, where $N$ is the number of points to be classified (see Remarks \ref{complexity1} and \ref{complexity2}). In the classification Theorem \ref{TH1}  each point to be classified is handled in an iterative manner. Thus, our estimate on the number of needed iterations  corresponds to the worst-case scenario. In practice, the optimal number of required switches will be lower, depending on the structure of the data to be classified (the initial data of the Neural ODE) and also on the partition chosen (or target in the case of Theorem \ref{TH2}). In particular,  when the initial data are structured into clusters our proof can take advantage of that to classify the data in fewer iterations, keeping them grouped in packages during the evolution.

%\smallskip

 Note that, even in the context of Theorem \ref{TH1} and \ref{TH2}, where the number of switches is of order of $\mathcal{O}(N)$, the control cost depends on the initial data and target destination. In particular, the cost of control increases when the distance between data corresponding to different labels decreases or their mixing increases.

%\smallskip

  The analysis is more intricate for the universal approximation theorem (Section \ref{SUA}) and for the control of  Neural transport equations (Section \ref{Transport}). Even if the proofs are also algorithmic, the $L^\infty$-norm of the control depends on the precision required when approximating the target as well as on the complexity of initial and final data, see Remark \ref{fractalremark}.
  \medskip

%\item \textbf{Computational cost.}

%The effective numerical approximation of the controls we build will require a fine approximation of the NODE under consideration with a number of time-steps higher than the number of switches of the controls in Theorems \ref{TH1} and \ref{TH2}. 

%Note also that the computational implementation of Theorem \ref{TH1} and \ref{TH2} requires to sort the data according to their target. Any  implementation of a sorting algorithm that makes comparisons has an algorithmic complexity of the order of $\mathcal{O}(N\log(N))$ \cite[Chapter 8]{cormen2009introduction}. Thus the cost of the sorting phase is  higher than the complexity of control process itself.
\medskip

     \item \textbf{Optimal control strategies.}
      As we have explained above, by time-scaling, all results are valid for an arbitrary $T>0$. For the analysis of the complexity of the control dynamics and the cost of control it is convenient to normalize the time to $T=1$. The complexity of the controls can be then estimated in terms of the number of  switches and their $TV$-norm. 
      
      %\smallskip
      
       The controls we build are not necessarily the optimal ones. In fact, as we have mentioned above, when the data to be classified present clustering phenomena, the complexity of the needed control diminishes.
      
      %\smallskip
     
      Optimal controls could be defined setting the time horizon to be $T=1$ and minimizing their $L^\infty$-norm. This would lead to an optimality system or Pontryagin maximum principle (\cite{pontryagin}) characterizing optimal controls that could be expected to present a bang-bang structure. The analysis of the complexity of optimal controls through such optimality system, which will strongly depend on the configuration  and structure of the data to be classified,   is an interesting and complex open problem. Of course such characterization does not lead to any explicit expression, it is of a purely implicit nature, making its posterior use rather complex (\cite{sussmann1979bang}). Note that this issue is even more complex in the context of the ReLU activation function, because of its lack of regularity.
     
     %\smallskip

    Other norms and cost criteria can also be used to define optimal controls. For instance, other than penalizing solely  the norm of the control, one could also penalize the controlled trajectories, enhancing the stability and the turnpike properties of the control processes, as shown in  \cite{esteve2020large} or in \cite{yague2021sparse} where the $L^1$ penalization of the state gives rise to sparse bang-bang controls as the ones shown in this article.
   
   %\smallskip
   
   In summary, the development of computational methodologies to derive control strategies of  minimal complexity is an interesting and challenging topic.
    \medskip

     \item \textbf{Optimal activation functions.}
 In the previous point, we were proposing to find optimal controls with prescribed initial and final data. One general open question would be the following. How to construct a vector field such that the norm of the control or the number of switches is minimal? or more precisely, what kind of vector fields can classify with minimal cost for any target function? 
  This question is analogous to the optimal observation and location of sensors in linear systems, addressed for instance in \cite{privat2015optimal,privat2016optimal}, but optimizing with respect to the nonlinearity.
 
 %\smallskip
  This question is not typically addressed in the control literature since, normally, in mechanics, we are not allowed to choose the dynamics that we want to control. 
 
 \medskip
 
 {\color{black}
  
  \item \textbf{Varying dimension.} The scaling of the control cost with respect to the precision $\epsilon$ and the number of switches in Theorem \ref{THsimple} depend on the dimension $d$.  The number of switches and the $L^\infty$ norm of the control suffer from the curse of dimensionality. However, the whole control process is made keeping $d$ invariant.
  
  Typically in ML, one can embed the system into a larger space, allowing more hyperplanes and benefiting from the extra directions. An open question would be: can we decrease the control cost by embedding the system into a larger dimension? How can one formulate a dimension varying Neural ODE? What is the impact on the control cost of this variability?
}
 \medskip
 
     \item \textbf{NODEs vs ResNets.} 
     The continuous setting of NODEs allows for a better understanding and mastery of the exponential growth and decay  of solutions with respect to time. These qualitative properties play an important role in the proofs of Theorems \ref{THsimple},\ref{TH5} and \ref{TH6}. On the other hand, the transport interpretation made in Section  \ref{Transport} allows for a more synthetic understanding of supervised learning. The time-continuous perspective can also be more adequate to analyze optimal control strategies.

      %One can see that the key elements for obtaining the approximate control of transport equations, as well as for the universal approximation are: simultaneous control of the ODE system, and compression.
     %The Neural ODEs approach to deep learning helps understanding processes that in the discrete version are less trivial such as compression of data points by setting attracting hyperplanes, exponential growths etc. 
     
     %\smallskip
     
 Once the NODE setting is well understood, a posteriori, the time-discrete setting allows achieving  similar results by means of fine discretizations of the Neural ODE, i.e., with sufficiently many deep layers. 

%\smallskip

 On the other hand, as mentioned in Remark \ref{d=1}, when the time step is large enough, the discrete dynamics is richer in the sense that they are able to express features that the continuous NODEs cannot. 
 
 %A further limitation that NODEs show is the fact that the they cannot change the topology of a given
     \medskip

 \item \textbf{On the topology approximation.} Our proof of universal approximation guarantees approximation in $L^2$. But it does not allow to obtain density results in spaces of higher regularity since it is fundamentally based on the approximation of simple functions. The approximation in $W^{k,p}$, if possible, would require more sophisticated techniques.   In the discrete setting, such results have been obtained  (\cite{daubechies2019nonlinear,burger2001error,guhring2020error}).
 \medskip
     
     \item \textbf{Other Neural ODEs.} In our poofs we have considered controls $W(t), A(t), b(t)$. But as we have seen, in each step of the control iteration only a few of the available control components  were activated. This suggests that similar results can be achieved with fewer controls.
For instance,
     \begin{itemize}
      \item $W=Id$. We could consider the model 
      $$\dot{x}=\boldsymbol{\sigma}(A(t)x+b(t))$$
      on which the vector field $\boldsymbol{\sigma}$ is always pointing towards the first quadrant, which makes it impossible to achieve purposes such as the Universal Approximation (Section \ref{SUA}) or the approximate control of NTEs (Section \ref{Transport}).
      \item We could also consider the variant
      $$\dot{x}=W(t)\boldsymbol{\sigma}(x)+b(t)$$
      Our results could be extended for this type of Neural ODEs,  using in an essential way the fact that  a quadrant of the phase space remains invariant under the dynamics. However, this structure of the NODE introduces further rigidity in the dynamics and the proofs require more complex arguments.
     \end{itemize}

     Another perspective is to generalize the results for second-order Neural ODEs such as continuous versions of Momentum Residual Neural Networks \cite{sander2021momentum}.
     \medskip

 \item \textbf{Optimal Transport.} Theorems \ref{TH5} and \ref{TH6} naturally establish a path towards optimal transport theory. 
 
 %\smallskip
 
  Optimal transport can be formulated in a dynamic framework. Namely, given two positive measures $\rho^0$ and $\rho^1$ with the same mass, the goal is to  find an optimal vector field $V$ and a solution $\rho$ of the continuity equation
 \begin{equation}\label{optrans}
 \begin{cases}
    \partial_t\rho+\mathrm{div}_x\left[V(x,t)\rho\right]=0\\
  \rho(t=0)=\rho^0\\
  \rho(t=T)=\rho^1
 \end{cases}
 \end{equation}
 so that $\rho$ transports $\rho^0$ into $\rho^1$ (\cite{benamou2000computational,chen2017matrix,benamou1998optimal,peyre2019computational}). This can be interpreted as an optimal controllability problem, the vector field $V$ being aimed to optimize the energy
 \begin{equation*}
  \mathcal{W}_2(\rho^0,\rho^1)=T\inf\int_0^T\int_{\mathbb{R}^d} \|V(x,t)\|^2d\rho(t) dt
 \end{equation*}{\color{black}
 subject to \eqref{optrans}. Here $\mathcal{W}_2$ stands for  the Wasserstein-2 distance. The control problem \eqref{optrans} is equivalent on finding geodesics in the measure space. }
 
 %\smallskip
 
 Theorems \ref{TH5} and \ref{TH6} guarantee that Neural transport equations \eqref{NTEq} and \eqref{contsys} allow to transport approximately $\rho^0$ into $\rho^1$. Analyzing how close are the optimal flows of Neural transport equations from the actual optimal solution of the dynamic optimal transport is an interesting and challenging problem. 
 \medskip

\end{enumerate}

\section*{Acknowledgments}
The authors acknowledge Borjan Geshkovski and Carlos Esteve Yagüe for their valuable comments.

\appendix

\section{Proof of the Fundamental Compression Lemma}\label{PROOF}
\setcounter{lemma}{0}

 \begin{lemma}[Fundamental Compression Lemma]%\label{FUNLEMMA}
 
  Let $\Omega\subset\mathbb{R}^d$ be bounded, and consider hyperplanes, strips and sets as in \eqref{H}.  Assume that the set of targets $\{\alpha_m\}_{m=1}^M\subset \mathbb{R}^d$ fulfil $\alpha^{(1)}_{m}\neq\alpha^{(1)}_{m'}$ if $m\neq m'$.
  %usepackage{xfrac} sfrac
    %´in \ref{UAStep1}.
   For any $T>0$ and any $0<\eta<\sfrac{1}{N}$ small enough, there exist piecewise-constant controls $A,W\in L^\infty((0,T);\mathbb{R}^{d\times d})$ and $b\in L^\infty((0,T);\mathbb{R}^{d})$ such that
  \begin{subequations}
   \begin{align*}
     &\phi_{1}(\mathcal{H}_l;A,W,b)^{(1)}\subset[\alpha_{m(l)}^{(1)}-C\eta,\alpha_{m(l)}^{(1)}+C\eta]\hspace{0.35cm}\qquad \text{for every }1\leq l\leq N,\\
     %&\phi_{1}(\mathcal{H}_l;A,W,b))^{(k)}\subset\left[-\mathcal{O}(\eta),\mathcal{O}(\eta)\right]\hspace{2.3cm}\qquad \text{for every }l=1,...,N\text{ and any }2\leq k\leq d,\\
     &\mathrm{diam}_{(k)}\phi_{1}(\Omega;A,W,b)<\eta\hspace{3.2cm}\qquad \text{for every }2\leq k\leq d,\\
     & \phi_{1}(\Omega;A,W,b)\subset \mathbb{B}(0,K)
   \end{align*}
  \end{subequations}
  with $K$ and $C$ independent of $N,N_\Gamma$ and $\eta$. Furthermore, fixing $T=1$, one has that:
    \begin{subequations}
   \begin{align*}
     &\|A\|_{L^\infty((0,1);\mathbb{R}^{d\times d})}=1,\qquad\|b\|_{L^\infty((0,1);\mathbb{R}^{d})}\lesssim_{\Omega} N ,\\
     &\|W\|_{L^\infty((0,1);\mathbb{R}^{d\times d})}\lesssim_{\Omega,\alpha,d} N\left(\frac{1}{\eta}+\frac{1}{\zeta}+\log\left(\frac{1}{\eta\zeta}\right)\right).\\
     &\text{The number of discontinuities of the controls }A,W,b\text{ is of the order of }N.
   \end{align*}
  \end{subequations}
  Here, the subscript $\lesssim_\alpha$ denotes the dependence on the set $\{\alpha_m\}_{m=1}^M$.
 \end{lemma}

   \begin{proof}
   In the whole proof we fix the $L^\infty$ norm of $W$ to be $1$, and we study the time needed for each step to be completed. By rescaling, this will give us the $L^\infty$ norm of $W$.
   
 \noindent\textbf{1. Compression of $\Omega_c$.}
 
  \noindent In this step we compress $\Omega_c$. Notice that in all this process we will be transforming also $\Omega$, being able to bound the transformation of $\Omega_h$ whose diameter will remain very close to the diameter of $\Omega_c$.
 We will compress the sets $\mathcal{H}_l$ so that:
 \begin{itemize}
  \item All hyperrectangles, after applying the controls, will have a diameter of the order of $\eta$ as small as we want.
  \item The compressed sets will be separated by a fix constant independent of $\eta$.
 \end{itemize}
  This will allow us to apply the simultaneous controllability after a refinement of Theorem \ref{TH1} made in the next step of the proof.
  %\smallskip
  
   There will be two phases, firstly we compress the $x^{(k)}$-coordinates for $k=1,...,d-1$. 
   As we will see later on, with this compression the sets $\mathcal{H}_l$ will have been transformed and moved in a particular configuration that will allow the the compression in the $x^{(d)}$-coordinate. Moreover, while compressing the the $x^{(d)}$-coordinate we will create a separation between the compressed sets.
  %FF\smallskip
  
     %{\color{pink}Let us define
%$$\mathrm{diam}(\Omega):=\max_{x_1,x_2\in\Omega} |x_1-x_2|$$
%$$\mathrm{diam}_{(k)}(\Omega):=\max_{x_1,x_2\in\Omega} |x_1^{(k)}-%x_2^{(k)}|.$$}

{\color{black}For each $k$ from $1$ to $d-1$ apply successively the following steps:} 
  \medskip
    \begin{enumerate}[leftmargin=0.5cm]
     \item[1.1]\label{2.1} \textbf{Compression of the $x^{(k)}$-coordinates $1,...,d-1$.}
        
        Here below we describe a procedure that will be recursively applied using successively all hyperplanes $\{x^{(k)}=c_{k,n}\}$ starting from $n=1$ until $n=N_{\Gamma,k}$.
        \medskip
        \begin{enumerate}
            \item[1.1.1]\label{2.1.1} \textbf{First Compression.}
                We set the hyperplane $\{x^{(k)}=c_{k,n}\}$ and the controls
                $$ A_{ij}=-\delta_{ki}\delta_{kj},\quad b_{i}=c_{k,n}\delta_{k,i}.$$
                We choose $W$ so that the field points towards the hyperplane, generating a contraction.
                
                 We solve \eqref{CTNN} for $T$ large enough so that the hyperrectangles that have been affected by the nonzero vector field have a $\mathrm{diam}_{(k)}$ lower than $\delta$  (see Figure \ref{precond0}), i.e.
                 $$\mathrm{diam}_{(k)}\phi_T(\mathcal{H}_l)<\delta\quad \text{ if } \max_{x\in\mathcal{H}_l}x^{(k)}<c_{k,n}.$$
                 
                    %$$ \max_{x\in\phi_T(\mathcal{H}_i)} x^{(d)}<\min_{\eta\in \mathcal{H}_j} \eta^{(d)}\quad\text{ if }                        \max_{x\in\mathcal{H}_i} x^{(k)}<c_{k,1},\quad                        \max_{x\in\mathcal{H}_j} x^{(k)}>c_{k,1}$$
                    By $\phi_T(\mathcal{H}_l)$ we understand the flow applied to the set $\mathcal{H}_l$ until time $T$ with the aforementioned controls.
     
                We make an abuse of notation, and we redefine $\mathcal{H}_l:=\phi_T(\mathcal{H}_l)$. This will be done after each step in the proof to avoid the long notation of applying sequentially different flows with different controls.
                
                The time needed for making this compression is of the order of:
                \begin{equation}\label{T.2.1.1}
                 T_{\text{step } 1.1.1}\sim_{\Omega} \log\left(\frac{1}{\delta}\right).
                \end{equation}
                %where by $\sim_\Omega$ we denote that there exists a constant $C$ depending on $\Omega$ such that $T=C\log(1/\delta)$.

                \begin{figure}[h!]
                    \hspace{-0cm}\begin{subfigure}[b]{0.1\textwidth}
                         \includegraphics[scale=0.8]{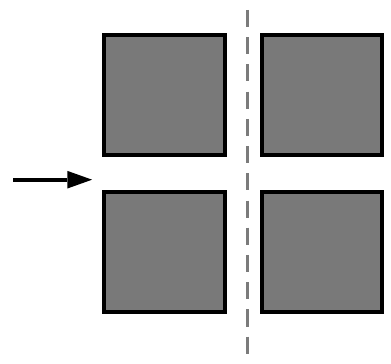}
                         \caption{}
                            \end{subfigure}
                            \hspace{6cm}\begin{subfigure}[b]{0.1\textwidth}
                            \includegraphics[scale=0.8]{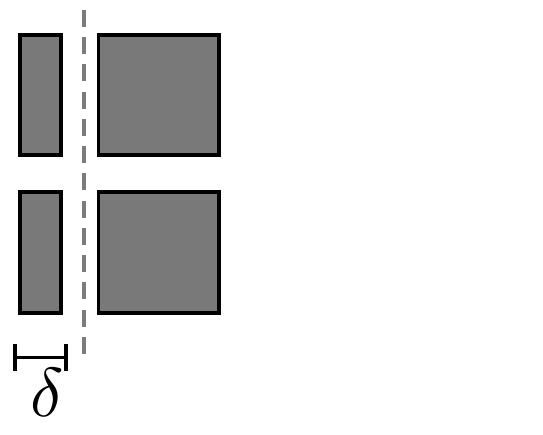}
                                           \caption{}
                            \end{subfigure}
                        \caption{Qualitative representation of Step 1.1.1 (A) Initial meshing of hypercubes where a vector field is applied to compress, (B) qualitative representation of the resulting transformation.}\label{precond0}
                \end{figure}

            \item[1.1.2] \textbf{Separation.}\label{2.1.2}
                We choose a parallel hyperplane to the one in the step above $\{x^{(k)}=c_{k,n}+(\sfrac{1}{4})\zeta)\}$ and the matrix
                $$ A_{ij}=-\delta_{ki}\delta_{kj},\quad b_{i}=\left(c_{k,n}+\frac{1}{4}\zeta\right)\delta_{k,i}.$$
                We choose $W$ so that the field points to $-\infty$ in the $x^{(d)}$-axis. We solve \eqref{CTNN} for $T$ large enough so that the sets that have been affected by the nonzero vectorfield, have lower values in the $x^{(d)}$-coordinate than the sets that have not been affected by the vectorfield (in the region in which the field has been $\boldsymbol{0}$) (see Figure \ref{precond1}), i.e.
                $$ \max_{x\in\phi_T(\mathcal{H}_i)} x^{(d)}<\min_{\xi\in \mathcal{H}_j} \xi^{(d)}\quad\text{ if }
                \max_{x\in\mathcal{H}_i} x^{(k)}<c_{k,n},\quad
                \max_{x\in\mathcal{H}_j} x^{(k)}>c_{k,n}.$$

                Given that the side of the removed strip is of the order of $\zeta$, the field will be bounded by below by a quantity of the order of $\zeta$. The  space needed for the translation is a constannt depending on $\Omega$. Therefore, the time needed for achieving the separation is of the order of
                \begin{equation}\label{T.2.1.2}
                    T_{\text{step } 1.1.2}\sim_{\Omega} \frac{1}{\zeta}.
                \end{equation}
                \begin{figure}[h!]
                    \hspace{-0cm}\begin{subfigure}[b]{0.1\textwidth}
                         \includegraphics[scale=0.8]{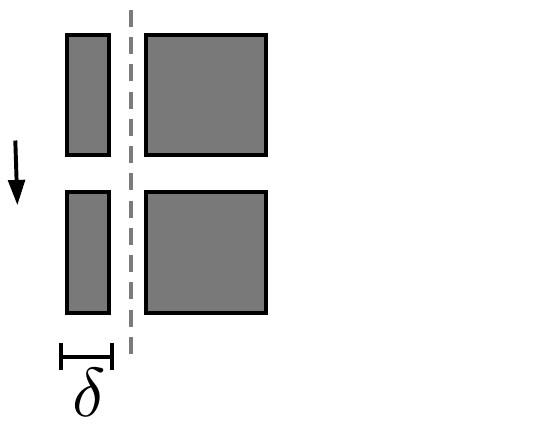}
                         \caption{}
                        \end{subfigure}
                        \hspace{6cm}\begin{subfigure}[b]{0.1\textwidth}
                        \includegraphics[scale=0.4]{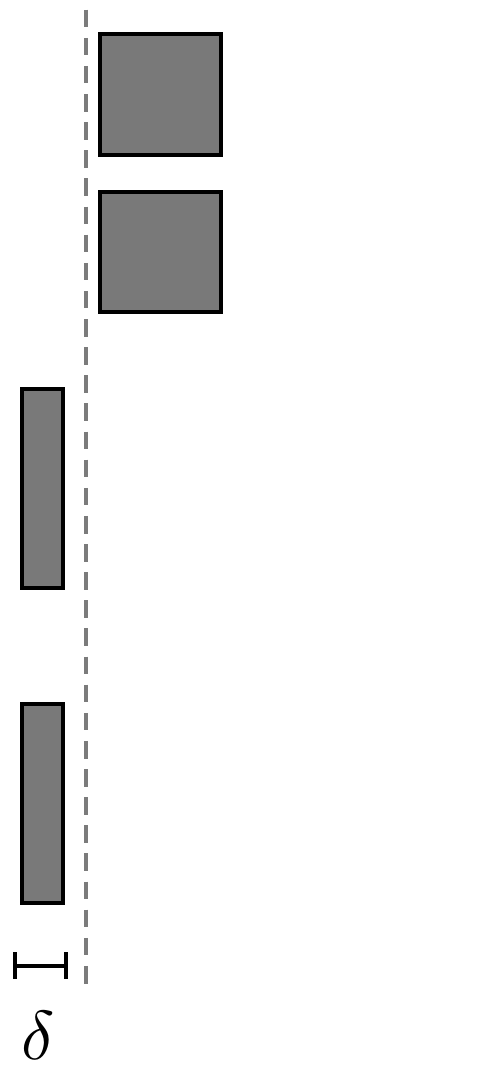}
                                           \caption{}
                               \end{subfigure}
                    \caption{Qualitative representation of step 1.1.2 (A) vector field  that is going to be applied to separate, (B) qualitative representation of the resulting transformation.}\label{precond1}
                \end{figure}
            \item[1.1.3] \textbf{Second Compression.}\label{2.1.3}
                We apply a compression in the $x^{(k)}$ coordinate, by setting a hyperplane $\{x^{(k)}=c_{k,n+1}\}$. %for $r$ large enough so that all sets are at one side of the hyperplane (see Figure \ref{precond2}). %Note that, now we are compressing more sets
                Setting
                $$ A=-\delta_{ki}\delta_{kj},\quad b_i=c_{k,n+1}\delta_{k,i}$$
                and $W$ so that the field points towards the hyperplane, making it attractive (see Figure \ref{precond2}). We solve \eqref{CTNN} for $T$ large enough so that:
                $$\sup_{x\in \phi_T(\mathcal{H}_l)} |x^{(k)}-c_{k,n+1}|<\delta,\quad \text{ if } \max_{x\in\mathcal{H}_l}x^{(k)}<c_{k,n+1}.$$
                The needed time is:     
                \begin{equation}\label{T.2.1.3}
                    T_{\text{step } 1.1.3}\sim_{\Omega} \log\left(\frac{1}{\delta}\right).
                \end{equation}
                
                \begin{figure}[h!]
                    \hspace{-0cm}\begin{subfigure}[b]{0.1\textwidth}
                    \includegraphics[scale=0.4]{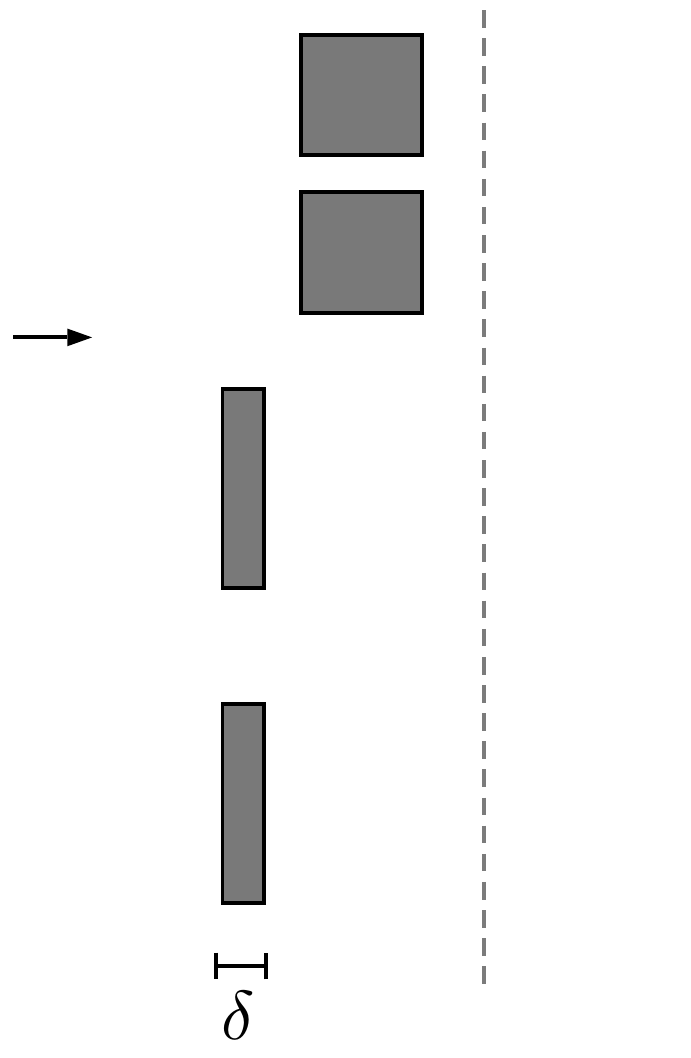}
                         \caption{}
                    \end{subfigure}
                    \hspace{6cm}\begin{subfigure}[b]{0.1\textwidth}
                    \includegraphics[scale=0.4]{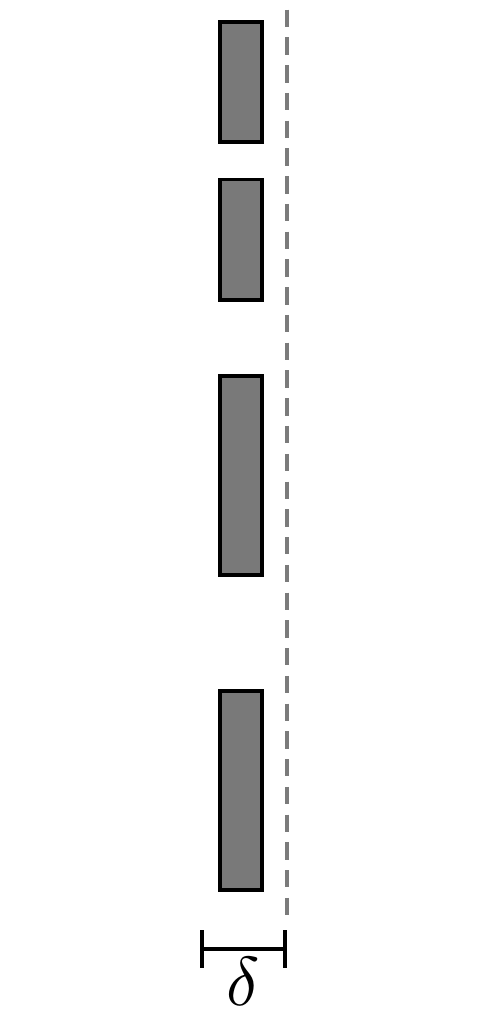}
                                           \caption{}
                               \end{subfigure}
                    \caption{Second compression phase of step 1.1.3 (A) a vector field is generated towards an attractive hyperplane, (B) the qualitative result}\label{precond2}
                \end{figure}

        \end{enumerate}
        
        As said before, this procedure is applied sequentially for every $1\leq k\leq d-1$ and every $ 1\leq n\leq N_{\Gamma,k}-1$ obtaining that the resulting sets satisfy
        $$\sup_{x\in \phi_T(\mathcal{H}_l)} |x^{(k)}-c_{k,N_{\Gamma,k}}|<\delta,\quad \forall k\in\{1,...,d-1\},\quad\forall l\in\{1,...,N\}.$$
    
        The set $\Omega$ has transformed, and $\mathrm{diam}_{(k)}(\phi_T(\Omega))<\delta $  for $1\leq k\leq d-1$.
        
                 {\color{black}
        Before continuing, we will ``\textit{move the sets away from the targets}''. The reason of this particular action will be better understood in Step 2 of this proof. 
        
        Let $c=\mathrm{diam}(\{0\}\cup\{\alpha_m\}_{m=1}^M)+1$,  consider the ball $\mathbb{B}(0,c)$ and pick the hyperplane
        \begin{equation}\label{allocationhyp}
         \left\{x^{(1)}=c\right\}.
        \end{equation}
        Choose $A$ and $W$ so that the vector field, in its nonzero region in the phase space, follows the Cartesian direction of the first element of the canonical basis and makes the hyperplane \eqref{allocationhyp} attractive. We solve \eqref{CTNN} for $T$ long enough so that:
        $$\sup_{x\in \phi_T(\mathcal{H}_l)} |x^{(1)}-c|<\delta,\quad \forall l\in\{1,...,N\}.$$
        This transformation does not affect the smallness of the $k$ coordinates from $k=1,...,d-1$.

        }
        
        The overall required time of this step is the following
        \begin{equation*}\label{T.2.1}
         T_{\text{step } 1.1}\sim_{\Omega,d,\alpha} N_\Gamma \left(\log\left(\frac{1}{\delta}\right)+\frac{1}{\zeta}\right)
        \end{equation*}
        where by $\sim_\alpha$ we denote the dependence on the targets.
        %\begin{equation}\label{T.2.1}
        %    T\sim h^{-(d-1)} \left(\log\left(\frac{1}{\delta}\right)+\frac{1}{h^d}\right)
        %\end{equation}
        
        In the separation step, Step 1.1.2, the set deforms since the field is not homogeneous along the set. We must ensure that a hyperplane with normal vector equal to an element of the canonical basis can fit between the transformed sets. This is the main reason of the first compression step, Step 1.1.1. This step reduces the $\mathrm{diam}_{(k)}$ of the set making that the field in the next step, Step 1.1.2 is almost constant on the hyperrectagle.  Let us quantify the present discussion.
        
        Before Step 1.1.2 the distance between two neighbor rectangles is $\zeta$. After applying Step 1.1.2 $(d-1)N_\Gamma$ times the separation is:
        $$\zeta-(d-1)N_\Gamma T_{\text{step 1.1.2}} \zeta\delta $$
        %and
        %$$N_\Gamma T_{\text{step 1.1.2}}h^{d}\sim h^{d-1} h^{-d}h^{d}\sim h^{d-1} $$
        so it will be enough if
        \begin{equation*}
                \delta \leq \mathcal{C} N_{\Gamma}^{-1}\zeta
        \end{equation*}
        for a certain constant $\mathcal{C}$ to guarantee that after the whole process the distance between two neighbor rectangles is of the order of $\zeta$.
        
        \medskip
        Notice that $A$ has been chosen to be unitary, also $b$ is bounded since we have chosen only among the original hyperplanes defined in \ref{UAStep1}. Furthermore, the number of switches we employed are explicit:
         \begin{equation}\label{S2.1}
          \text{The number of switches in the }A\text{ control is }d-1.
         \end{equation}
         \begin{equation}\label{S2.1b}
          \text{The number of switches in the }b\text{ control and }W\text{ control are of the order of }N_{\Gamma}.%h^{-(d-1)}.
         \end{equation}

        \medskip
     \item[1.2] \textbf{Compression of the $d$-coordinate.}
     
                  Now, the task is to compress the $x^{(d)}$- coordinate and separate enough each set while keeping a small diameter. After the transformations made above, the sets can be separated by $N-1$ hyperplanes of the form $$\left\{x^{(d)}=C_{d,l}\right\}\qquad 1\leq l\leq N$$%(more hyperplanes than the ones defined for the $x^{(d)}$-coordinate in \ref{UAStep1}).
                  For doing so, for $l$ from $1$ until $N$ we proceed as follows:
             
             \begin{enumerate}
             \medskip
            \item[1.2.1] \textbf{Compression.}
                 Take $\{x^{(d)}=C_{d,l}\}$ and $A_{ij}=\delta_{di}\delta_{dj}$, $b_i=-C_{d,l}\delta_{di}$ and we choose $W$ so that the hyperplane $\{x^{(d)}=C_{d,l}\}$ is attractive. We choose $T>0$ so that 
                $$\sup_{x\in \phi_T(\mathcal{H}_{l})}| x^{(d)}-C_{l,d}|<\delta\qquad \text{$\forall l$ such that} \inf_{x\in\mathcal{H}_l}x^{(d)}>C_{d,l}$$
                see Figure \ref{precond3}.
      
                \begin{figure}%[h!]
                    \hspace{-0cm}\begin{subfigure}[b]{0.1\textwidth}
                    \includegraphics[scale=0.4]{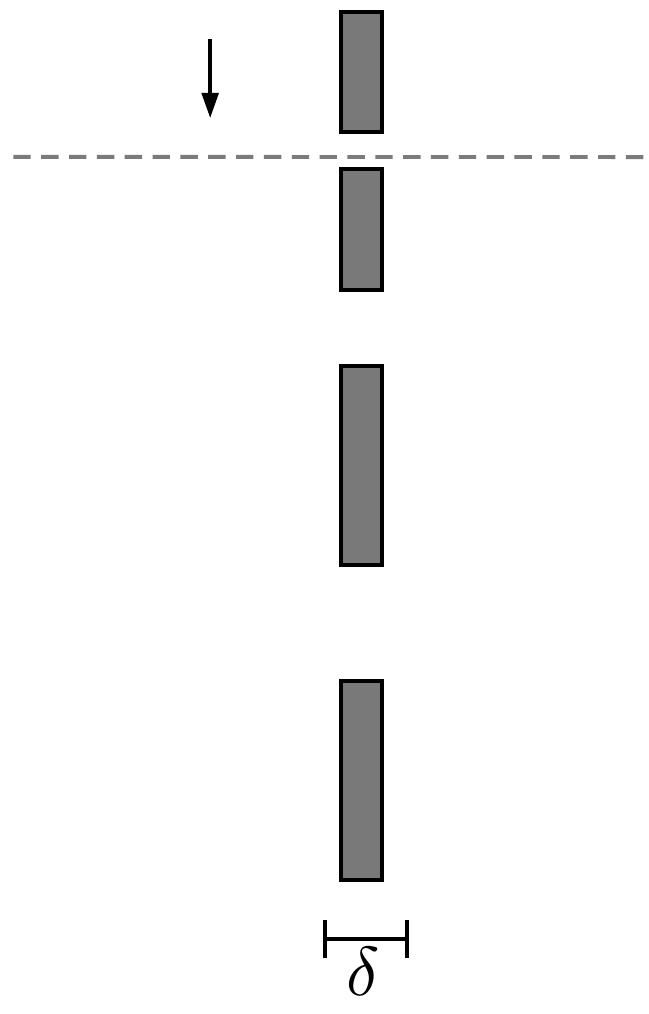}
                         \caption{}
                    \end{subfigure}
                       \hspace{6cm}\begin{subfigure}[b]{0.1\textwidth}
                        \includegraphics[scale=0.4]{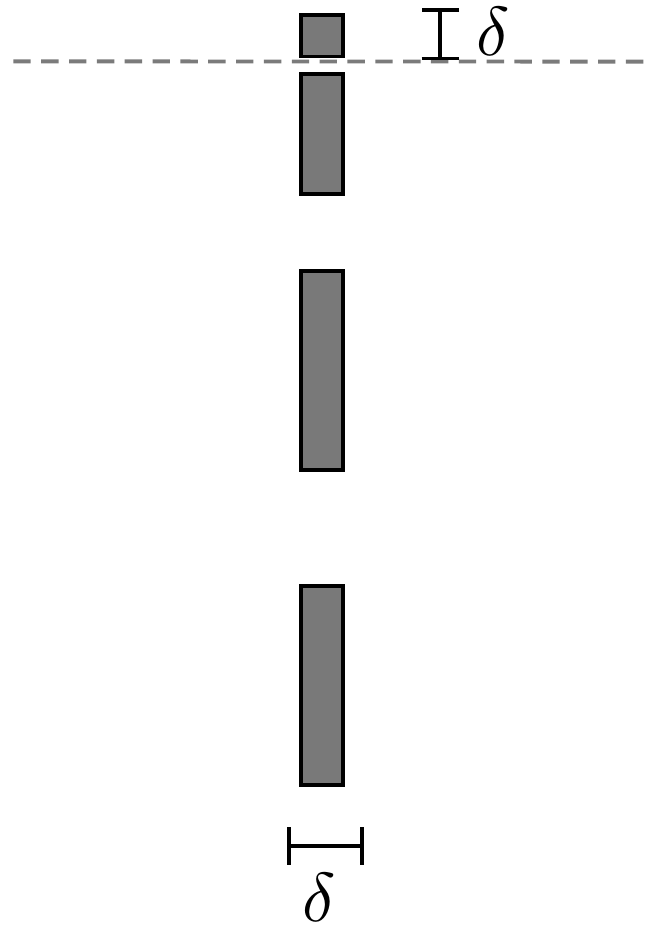}
                         \caption{}
                  \end{subfigure}
                    \caption{Compression the $x^{(d)}$ coordinate, step 1.2.1 (A) a vector field is generated towards an attractive hyperplane. (B) The resulting effect.}\label{precond3}
                \end{figure}
                The diameter of $\phi_{T}(\mathcal{H}_l)$ is lower than $\mathrm{diam}(\phi_{T}(\mathcal{H}_l))<\sqrt{d}\delta$.

            \medskip
            \item[1.2.2] \textbf{Separation.} We fix a parallel hyperplane to the one before. We set it slightly below at a distance of the order of the removed band,  $\zeta$,
                $\{x^{(d)}=C_{d,l}-(\sfrac{1}{4})\zeta\}$, and we choose $A_{ij}=\delta_{di}\delta_{dj}$, $b_i=(-C_{d,l}+(\sfrac{1}{4})\zeta)\delta_{di}$ and $W$ such that the field points to the $x^{(1)}$-coordinate. We solve \eqref{CTNN} for $T^*$ large enough so that 
                $$\inf_{x\in\phi(\mathcal{H}_l)} \left|c_{1,N_\Gamma}-x^{(1)}\right|= \frac{2}{N}. $$
                This would require $T^*=8/(N\zeta)$.
     
                When separating the sets one should ensure that at the end of the process we stay in a bounded set independently of $N$. %In this way, we will be able to adapt the simultaneous controllability that will be done in the next step of the proof. This is the reason of the separation choice made right above.
     
                \begin{figure}%[h!]
                        \hspace{-0cm}\begin{subfigure}[b]{0.1\textwidth}
                        \includegraphics[scale=0.4]{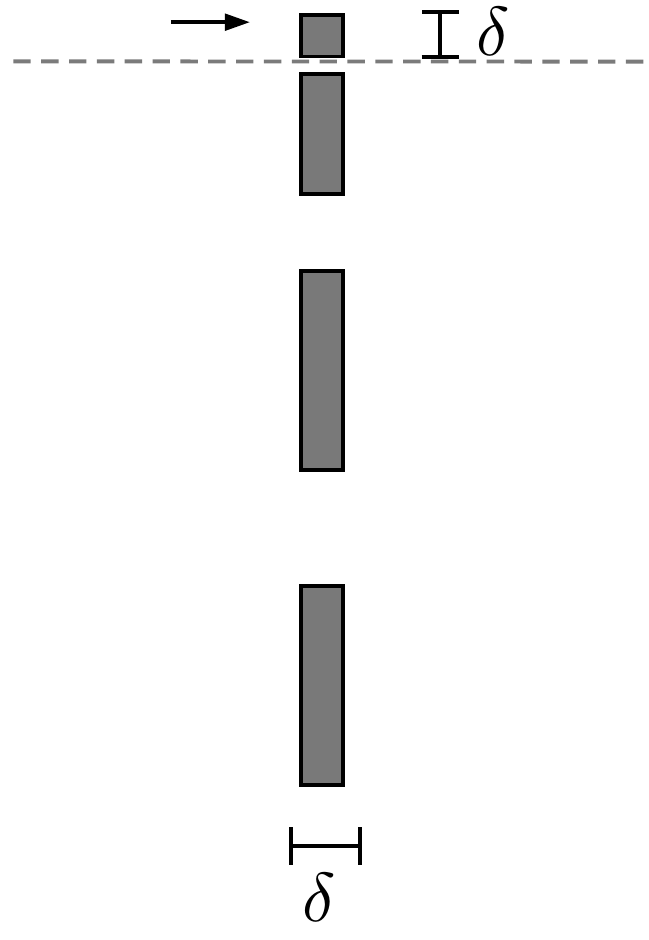}
                         \caption{}
                  \end{subfigure}
                       \hspace{6cm}\begin{subfigure}[b]{0.1\textwidth}
                    \includegraphics[scale=0.4]{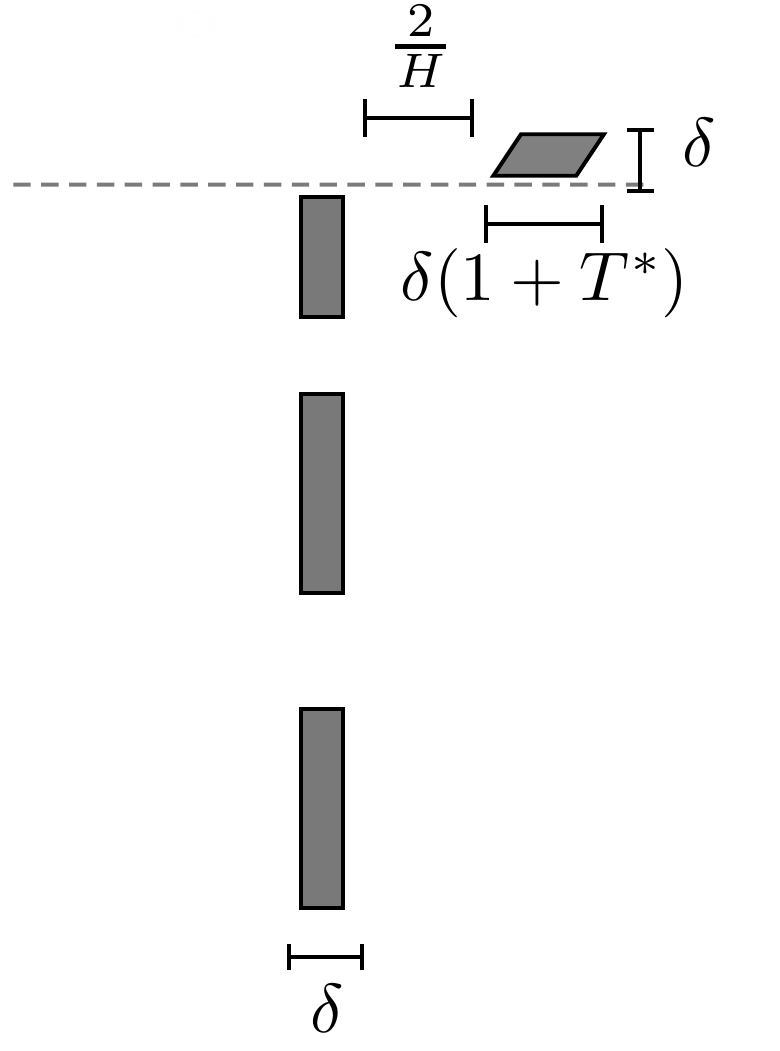}
                         \caption{}
                  \end{subfigure}
                    \caption{Displacement in the $x^{(1)}$ direction, Step 1.2.2 (A) A vector field is generated in order to push the set in the $x^{(1)}$-coordinate (B) Qualitative representation of the resulting effect.}\label{precond4}
                \end{figure}
                
                After applying this movement along the $x^{(1)}$-coordinate, the  diameter of the set $\mathcal{H}_l$ has grown being:
                \begin{align*}
                    \mathrm{diam}(\phi(\mathcal{H}_l))^2\leq& \left((d-1)+\left(1+T^*\right)^2\right)\delta^2\\
                    &\left((d-1)+\left(1+\frac{8}{N\zeta}\right)^2\right)\delta^2.
                \end{align*}

            \end{enumerate}
            
            \begin{figure}%[h!]
                                      \hspace{-0cm}\begin{subfigure}[b]{0.1\textwidth}
                    \includegraphics[scale=0.4]{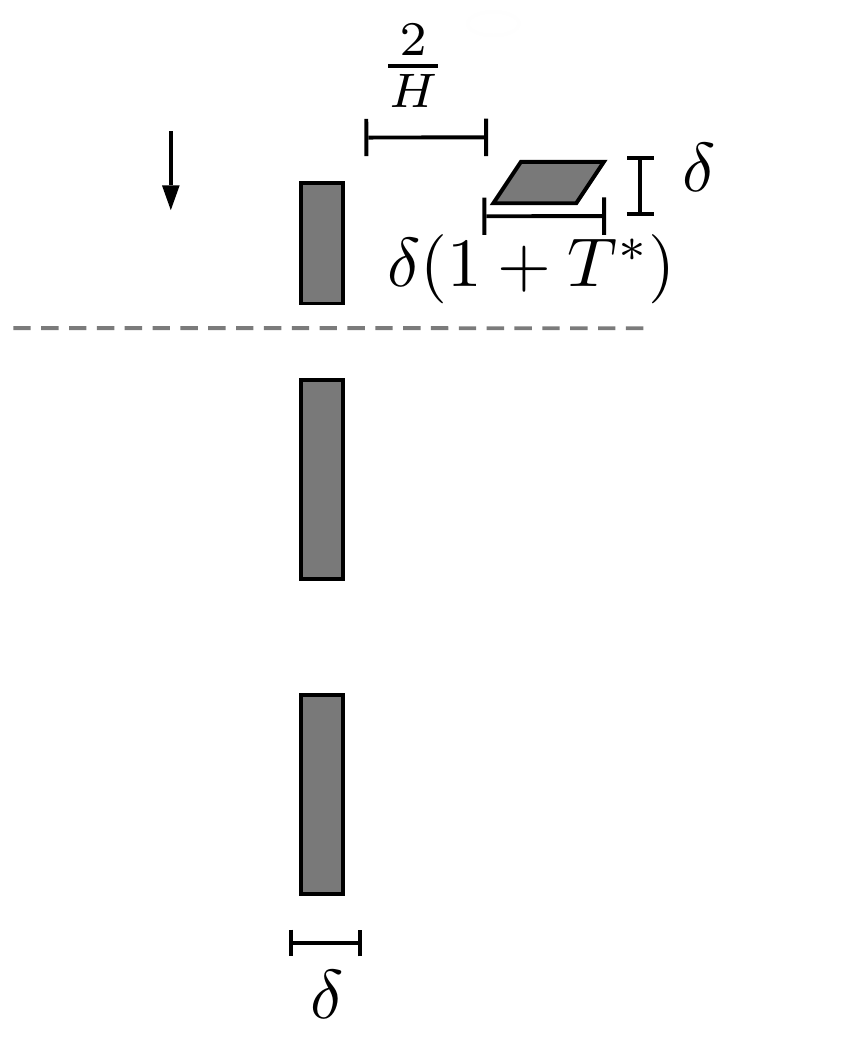}
                         \caption{}
                  \end{subfigure}
                       \hspace{6cm}\begin{subfigure}[b]{0.1\textwidth}
                    \includegraphics[scale=0.4]{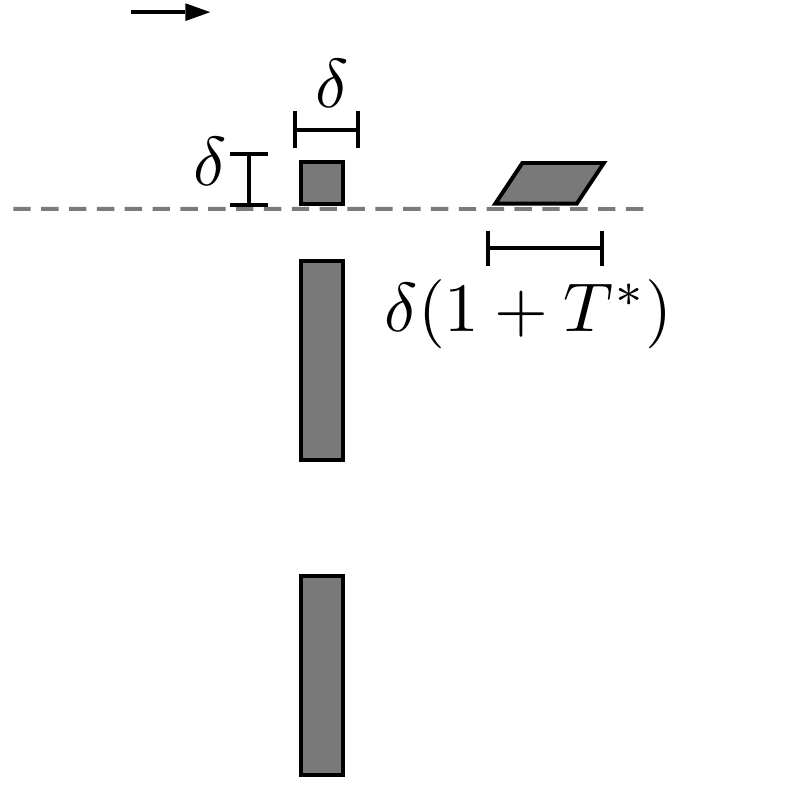}
                         \caption{}
                  \end{subfigure}
                    \caption{Compression phase for the $x^{(d)}$ direction, (A) a vector field is generated towards an attractive hyperplane. (B) Qualitative representation of the resulting effect.}\label{precond5}
                \end{figure}
                
                The time needed to perform this steps, the number of switches and the norm of $b$ are estimated by the following:
                \begin{equation}\label{T2.2}
                    T_{\text{step } 1.2}\sim_{\Omega} N\left(\log\left(\frac{1}{\delta}\right)+\frac{1}{\zeta}\right),\quad                     \|b\|_\infty\sim N
                \end{equation}

                %\begin{equation}\label{N2.2}
                %    \|b\|_\infty\sim N
                %\end{equation}
     
                \begin{equation}\label{S2.2}
                    \begin{array}{c}
                     \text{The number of switches of the $A$ control has been $1$.}\\
                     \text{The number of switches of the $b$ and $W$ control is of the order of $N$.}
                    \end{array}                    
                \end{equation}
     
                %\begin{equation}\label{S2.2b}
                %    \text{The number of switches of the $b$ and $W$ control is of the order of $N$}
                %\end{equation}

    \end{enumerate}
    \medskip
    
                    We apply iteratively this process (see Figure \ref{precond5}) for all hyperplanes $\{x^{(d)}=C_{d,l}\}$ (obtaining at the end Figure \ref{finalsec}).  Since the velocity field has different strength at different parts of the set, the distance between two $\mathcal{H}_l$ can diminish.
                
     In order to guarantee that after the whole process we have enough distance between the sets, we have to require $\delta$ to be small enough. When we apply the horizontal translation, the difference in the vector field inside each set is $\delta$. Hence, we know that the distance in their $x^{(1)}$-coordinate between the sets reduces as:
     $$\min_{\substack{x\in\phi_t(\mathcal{H}_l)\\z\in\phi_t(\mathcal{H}_{l+1})}}|x^{(1)}-z^{(1)}|\geq\frac{2}{N}-\left(\delta t\right).$$
     If we impose that, after $T^*N$ units of time, the distance in the $x^{(1)}$-coordinate is bigger than $1/N$, this implies that the requirement is the following:
     \begin{equation}\label{delta1}
      \delta<\frac{\zeta}{8N}%\lesssim h^{-d^2}.
     \end{equation}
     The requirement \eqref{delta1} will be not the only one.
     Now the task is to guarantee that $\mathrm{diam}(\phi_T(\Omega_h))$ is uniformly bounded for $h$ small enough and that the maximum diameter of all the sets $\phi_T(\mathcal{H}_l)$ is as small as we desire.
     
     \begin{itemize}
            
            \item The requirement for the diameter of each set follows from
            \begin{equation*}
                \mathrm{diam}(\phi(\mathcal{H}_l))^2\leq \left((d-1)+(1+NT^*)^2\right)\delta^2
            \end{equation*}
            \begin{equation}\label{diambound}
                \mathrm{diam}(\phi(\mathcal{H}_l))^2\leq \left((d-1)+\left(1+\frac{8}{\zeta}\right)^2\right)\delta^2\leq \eta^2
            \end{equation}
            \item For the diameter of $\phi_T(\Omega_h)$ one has that
            \begin{align*}
            \mathrm{diam}(\phi_T(\Omega_h))^2\leq& \delta^2(d-1)+N^2\left(\frac{2}{N}+\max_{l} \mathrm{diam}(\mathcal{H}_l))\right)^2\\
            \leq&\delta^2(d-1)+4+N^2\eta^2+4N\eta.
            \end{align*}
     \end{itemize}
     Choosing any $\eta$ such that
     \begin{equation}\label{criteta}
        \eta<\frac{1}{N}
     \end{equation}
     we would guarantee that $\phi_T(\Omega)$ remains bounded.%, i.e.
     %\begin{equation*}
     % \eta\lesssim h^{-d(d-1)}
     %\end{equation*}

      With all the process described, we have moved the sets downwards in the $x^{(d)}$-component, this translation depends on $N$. However, by putting an attractive hyperplane at the of the form $x^{(d)}=c$ at the desired place, we can move the set without having a dependence on $N$.
     \smallskip
     
         Now the sets are at distances bigger than $1/N$ between them and each one of them has a bound in its diameter  \eqref{diambound}. We can choose $\delta$ as small as we want and we can guarantee that:
        \begin{equation*}
         \mathrm{diam}(\phi(\mathcal{H}_l))<\eta \quad \iff\quad\delta^2 <\eta^2\left((d-1)+\left(1+\frac{8}{\zeta}\right)^2\right)^{-1}. 
        \end{equation*}

     \begin{figure}%[h!]
      \includegraphics[scale=0.7]{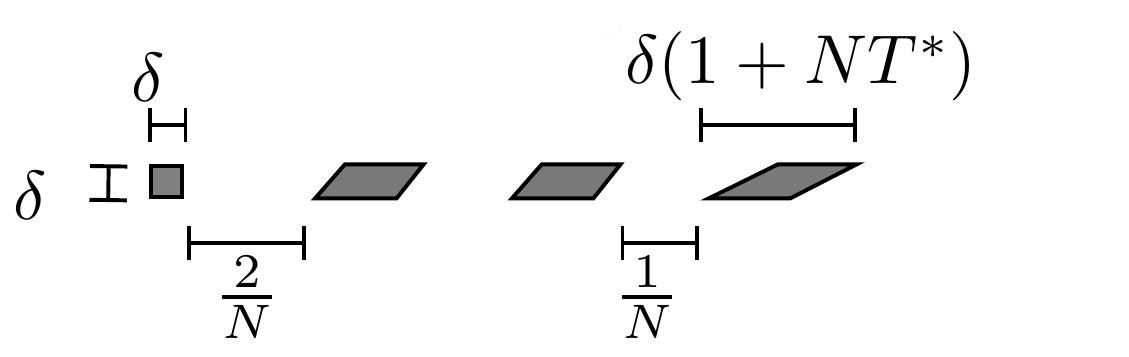}
      \caption{Qualitative representation of the final configuration of the compression process. The sets $\phi_T(\mathcal{H}_l)$ are uniformly separated by $1/N$ and their diameter is of the order of $\delta$.}\label{finalsec}
     \end{figure}

      Using the expressions \eqref{T2.2} and \eqref{T.2.1}, we find that the time horizon needed in Step 1 of this proof scales in the following manner:
     \begin{equation}\label{T2}
      T_{\text{step }1}\sim_{\Omega,d,\alpha} N\left(\log\left(\frac{1}{\delta}\right)+\frac{1}{\zeta}\right)+N_{\Gamma}\left(\log\left(\frac{1}{\delta}\right)+\frac{1}{\zeta}\right)
     \end{equation}

      The norm of the control $b$ depends on $N$ as seen in the treatment of the $x^{(d)}$-coordinate (Step 1.2) \eqref{T2.2}
     \begin{equation}\label{N2}
      \|b\|_\infty\sim N
     \end{equation}
     Considering \eqref{S2.2}\eqref{S2.1} the total number of switches on the control $A$ is 
     \begin{equation}\label{S2}
      \text{The number of switches of the $A$ control has been $d$},
     \end{equation}
     while for $b$ and $W$ we use \eqref{S2.1b} and \eqref{S2.2}
    \begin{equation}\label{S2b}
      \text{The number of switches of the $b$ and $W$ control is of the order of $N$}
     \end{equation}

    \medskip
 \noindent\textbf{2. Ordering and Grouping.}

         Let $\phi_T$ be the flow associated to the compression phase, and let $\Omega$ and $\mathcal{H}_l$ be defined as in \ref{UAStep1} of this proof. We redefine $\Omega$ and $\mathcal{H}_l$ as $\Omega:=\phi_T(\Omega)$ and $\mathcal{H}_l:=\phi_T(\mathcal{H}_l)$. Without loss of generality, we make a change of coordinates, for simplifying notation, in a way that:
        $$ \Omega \subset \{x:\quad |x^{(1)}|\leq 2,|x^{(k)}|<\eta\quad k\geq 2\}.$$
        Moreover, let $K=2\textcolor{white}{.}\mathrm{diam}\left(\Omega\cup \{\alpha_m\}_{m=1}^M\cup \{0\}\right)$ and consider the ball $ \mathbb{B}(0,K).$
        %for some fixed $x^*\in\Omega\cup \{\alpha_m\}_{m=1}^M$.
        Naturally one has that:
        \begin{equation*}
            \Omega \subset \mathbb{B}(0,K),\qquad
            \{\alpha_m\}_{m=1}^M \subset \mathbb{B}(0,K).
        \end{equation*}
 
         %Let us denote by $\psi_T$ the flow associated with the simultaneous control of Theorem \ref{TH2}. If we apply it to the centers of each $\mathcal{H}_l$ we can get approximately to the targets. However, note that we need to ensure that $\psi_T(\Omega_h)$ remains bounded independently of $h$. This needs a more subtle treatment since $\mathrm{diam}(\Omega_h)$ is not small and the $L^\infty$-norm of the control (or the controllability time) depend on the number of sets to control as noticed in Remarks \ref{complexity1} and \ref{complexity2} (see also Remark \ref{fractalremark}).
  
         Note that by Step 1.1, all the sets are \textit{away} from the targets and in particular, the following condition holds:
        \begin{equation}\label{reccondition}
            \alpha_{m'}^{(1)}\notin \mathcal{H}_l^{(1)} \quad\text{ if }\alpha_{m(l)}\neq \alpha_{m'}.
        \end{equation}
        Remind that by $(\omega)^{(k)}$ we denote the $k$-th component of the set $\omega$, i.e. $$\omega^{(k)}:=\{x\in\mathbb{R}: \exists y\in\omega\quad \text{ such that } y^{(k)}=x\}.$$
        Now we build a control that will control the first component of $\mathcal{H}_l$ to their corresponding target while ensuring that after this control process, \eqref{reccondition} still holds. This will allow us to proceed recursively.

         The procedure we are going to show is, in particular, an improvement of Theorem \ref{TH1}  for keeping track properly of $\Omega_h$ while applying our controls.
         
         The time horizon will depend on the number of elements $N$. The mission is to make use of the explicit controlled flows to obtain sharper bounds and independent on the number of elements. Since the time horizon depends on $N$, the direct application of Grönwall's inequality is not of practical use for bounding $\Omega_h$ independently of $N$. This is the reason why we examine the dynamics in more detail.%In order to not depend on the number of hyperrectangles we will use the refinement of Theorem \ref{TH1} to first order the hyperrectangles by strips according to their target for later compressing each strip so that they can be treated as $M$ elements. 

         {\color{black}This Lemma that we are proving, leaves the system in a configuration in which the simultaneous controllability of Theorem \ref{TH2} can be applied. In the application of Theorem \ref{TH2}, after applying Theorem \ref{TH1}, we only set controls that endow flows that follow directions in the orthogonal subspace of $\mathrm{Span}((1,0,...,0))$.   Due to the Cartesian structure of our flows see that  we only have to ensure that after the refinement of Theorem \ref{TH1} the condition $\mathrm{diam}_{(k)}(\Omega)=\mathcal{O}(\eta)$ for any $k\geq 2$ still holds.}

    %\begin{enumerate}
    
    \medskip
     %\item[2.1] \textbf{Ordering and Grouping}
     
            For every $l\in\{1,...,N\}$ consider the center of $\mathcal{H}_l$,
            $$d_l=\frac{1}{|\mathcal{H}_l|}\int_{\mathcal{H}_l}xdx $$
            and we proceed sequentially.
            \begin{enumerate}[leftmargin=0cm]
                \item  Observe that $\Omega$ is bounded in the second coordinate by:
                    $$ x^{(2)}=\eta,\qquad x^{(2)}=-\eta.$$
                    Fix the following bounds:
                    \begin{subequations}\label{fin}
                        \begin{align}
                            &x^{(1)}=d_l^{(1)}+2\eta\label{fin1}\\
                            &x^{(1)}=d_l^{(1)}-2\eta\label{fin2}
                        \end{align}
                        \end{subequations}
                    and the hyperplane\eqref{fin3}.
                    \begin{equation}\label{fin3}
                        \left\{x^{(2)}=-1\right\}
                    \end{equation}
                    as Figure \ref{Lastfigure}.A shows.
  
                    \begin{figure}%[h!]
                    \begin{subfigure}[b]{0.37\textwidth}
                        \includegraphics[width=\textwidth]{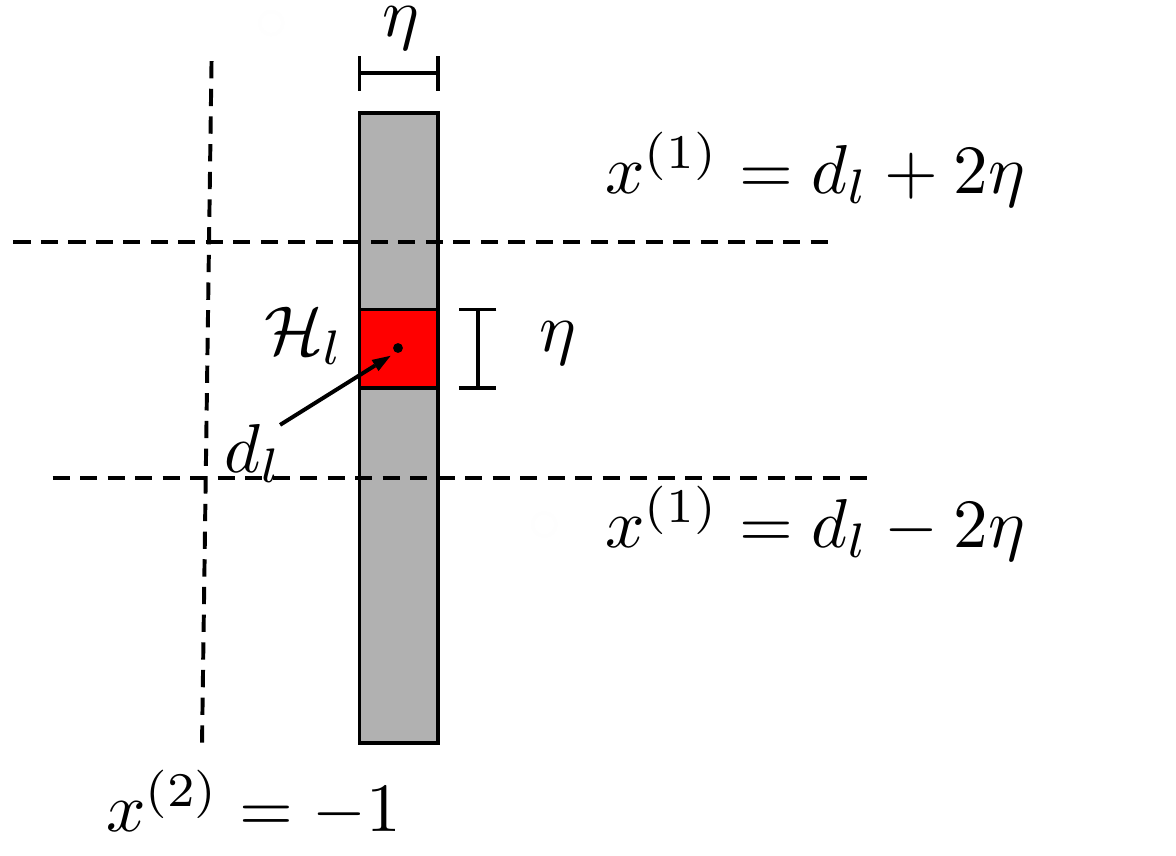}
                        \caption{}
                        \end{subfigure}
                        \\
                        \begin{subfigure}[b]{0.37\textwidth}
                            \includegraphics[width=\textwidth]{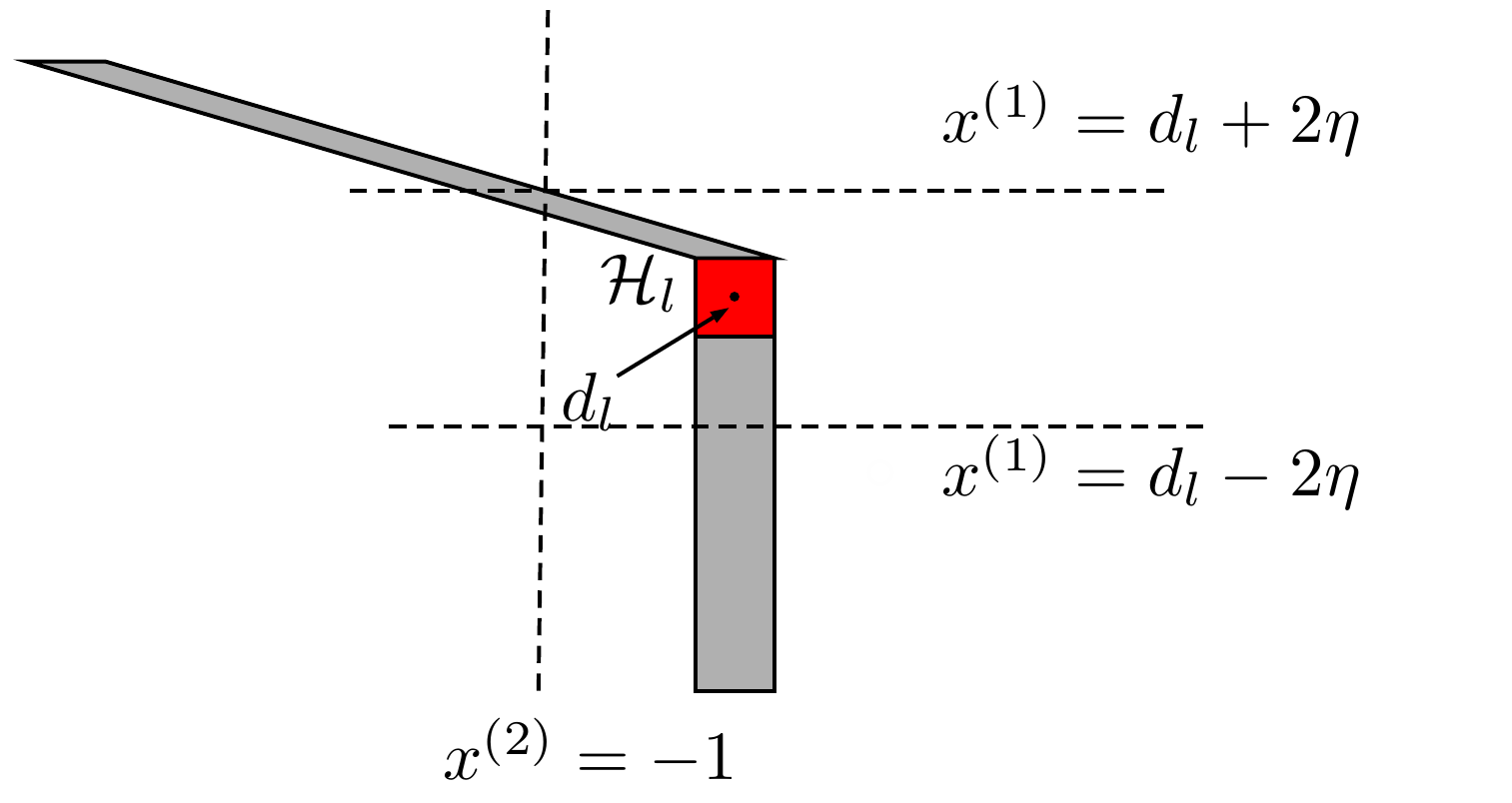}
                        \caption{}
                            \end{subfigure}
                        \hfill
                        \begin{subfigure}[b]{0.37\textwidth}
                            \includegraphics[width=\textwidth]{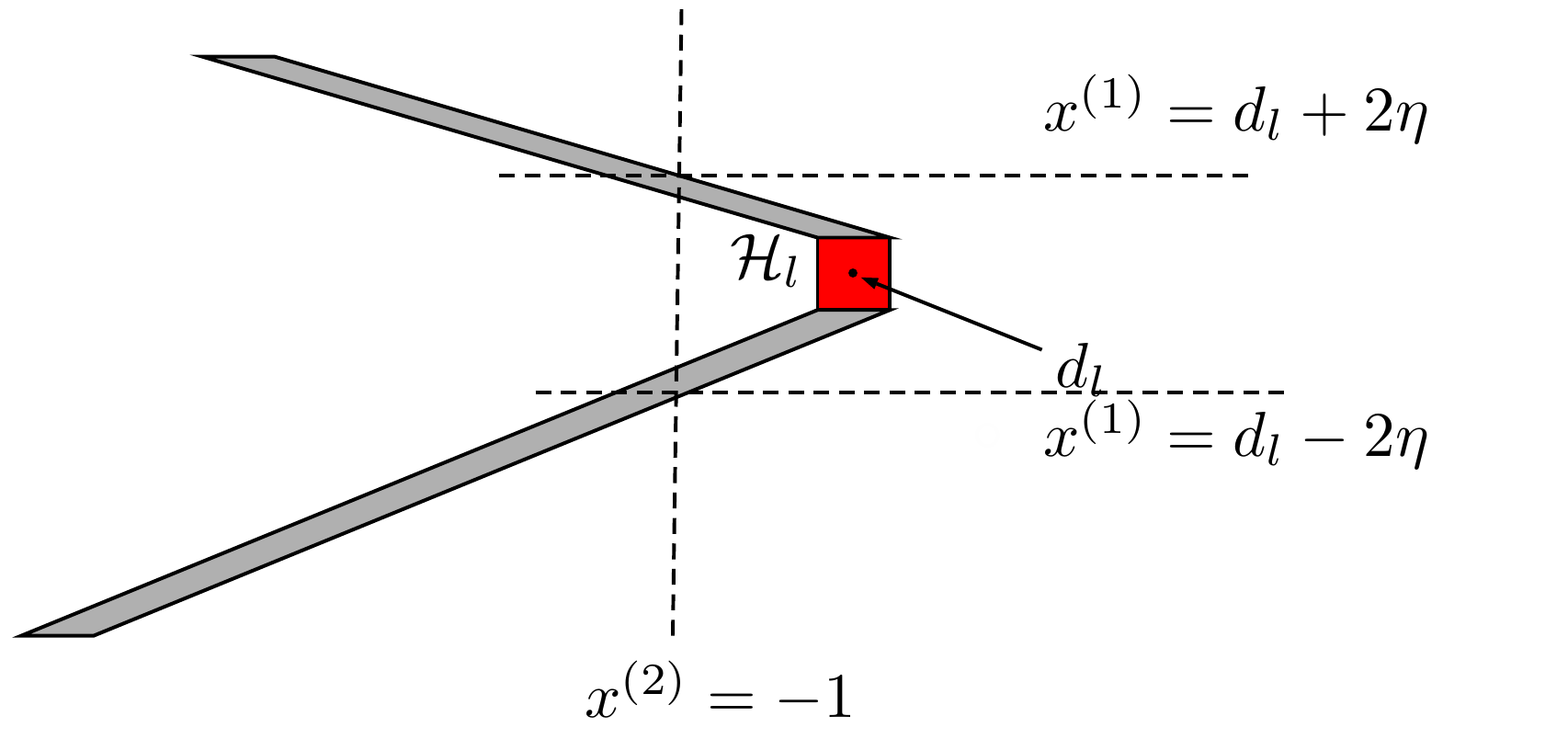}
                            \caption{}
                            \end{subfigure}
                        \caption{(A) Representation of the bounds \eqref{fin} and the hyperplane \eqref{fin3}, $\Omega$ is contained in the hyperrectangle, the set $\mathcal{H}_l$ in red and $\Omega\setminus\mathcal{H}_l$ in grey. (B) The resulting set after applying the flow $\psi_{T_{Sep1}}^{Sep1}$ to $\Omega$. (C) The resulting set after applying $\psi_{T_{Sep2}}^{Sep2}$ to $\psi_{T_{Sep1}}^{Sep1}(\Omega)$.}\label{Lastfigure}
                    \end{figure}

            \item We select the hyperplane
                $\{x^{(1)}=d_l+\eta\}.$
                We apply the flow as in Theorem \ref{TH1}, pointing at the negative side in the $x^{(2)}$-component (as done in Theorem \ref{TH1} in equation \eqref{tria2} and selecting the control $A$ as in equation \eqref{tria1}). Then, the bound in $x^{(2)}$ will evolve as
                \begin{equation}\label{boundevolution}
                    x^{(2)}(t)=-\max(x^{(1)}-d_l^{(1)}-\eta,0)t+\eta.
                \end{equation}
                \item Let us denote by $\psi_{T}^{Sep1}$ the associated flow with such controls.  We want to find $T_{Sep1}$ such that:
                $$\mathrm{diam}_{(1)}\left(\psi^{Sep1}_{T_{Sep1}}(\Omega)\cap \{x^{(2)}\geq -1\}\cap \{x_1\geq d_l+\eta\}\right)=\eta$$
                    For doing so, we compute the intersection between \eqref{boundevolution} and \eqref{fin3}, taking $x^{(1)}$ as in \eqref{fin1} and we find a lower bound for the time in which the translation has to be applied (see Figure \ref{Lastfigure}.B):
                \begin{equation}\label{boundont}
                        T_{Sep1}\sim 1+\frac{1}{\eta}.
                \end{equation}
                %The bound \eqref{boundont} has to be considered together with the bound \eqref{criteriaT1} in Theorem \ref{TH1} with the other centers of the other sets $\psi(\mathcal{H}_n)$ $n\neq l$.
            \item We can do an analogous procedure for the hyperplane $x^{(1)}=d_l-\eta$, and we obtain also the same lower bound for the time  \eqref{boundont}. Let us denote by $\psi_{T_{Sep2}}^{Sep2}$ the resulting flow. After this point, we obtain a structure similar to Figure \ref{Lastfigure}.C. Let us define, by convenience,   
                    $$\psi^{Sep}:=\psi_{T_{Sep2}}^{Sep2}\circ \psi_{T_{Sep1}}^{Sep1}.$$
                Now we have that:
            \begin{equation*}
                \mathrm{diam}_{(1)}(\psi^{Sep}(\Omega)\cap \{x_2\geq -1\})\leq 4\eta,\qquad\qquad
                (\psi^{Sep}(\Omega))^{(1)}=(\Omega)^{(1)}.
            \end{equation*}
            
            \item Now we choose the hyperplane $\{x^{(2)}=-1\}$ and we apply a translation movement. Denote by $\psi_{T_{Mov}}^{Mov}$ the translation movement flow. We choose a vector field in the $x^{(1)}$ direction pointing to the first component of the target of $d_l$, $\alpha_{m(l)}^{(1)}$, during the required time $T_{Mov}$ to reach the target. Since the distance between the hyperplane $\{x^{(1)}=-1\}$ and $d_l$ is independent of $\eta$, $T_{Mov}$ does not depend on the parameter $\eta$. 
            $$T_{Mov}=\frac{|d_l^{(1)}-\alpha_{m(l)}^{(1)}|}{|d_l^{(2)}+1|}. $$
            Then we have that, defining $a_1:=\min\{\alpha_{m(l)}^{(1)},d_l^{(1)}\}$ and $a_2:=\max\{\alpha_{m(l)}^{(1)},d_l^{(1)}\}$, there exists a constant $C$ independent of $\eta$ such that:   
            \begin{equation*}
                \left(\psi_{T_{Mov}}^{Mov}\left(\psi^{Sep}(\Omega)\cap\{x_2\geq -1\}\right)   \right)^{(1)}\subset \left[a_1-C\eta,a_2+C\eta\right]\subset \mathbb{B}(0,K)^{(1)}%\\
                %&\subset \mathbb{B}(x^*,D)^{(1)}
            \end{equation*}
            for $\eta$ small enough. Moreover
            \begin{align*}
                &\psi_{T_{Mov}}^{Mov}\left(\psi^{Sep}(\Omega)\cap\{x_2< -1\}\right)=\psi^{Sep}(\Omega)\cap\{x_2< -1\}\\
                &\max_{x\in\psi_{T_{Mov}}^{Mov}\left(\psi^{Sep}(\mathcal{H}_l)\right)}|x-\alpha_{m(l)}|\leq C_2\eta%&\left(\psi_{T_{Mov}}^{Mov}\left(\psi^{Sep}(\Omega)\cap\{x_2< -1\}\right)\right)^{(1)}\subset (\Omega)^{(1)}.
            \end{align*}
            where $C_2=((d-1)+C\eta)^{1/2}$.
            
            For guaranteeing that after this transformation we still fulfill condition \eqref{reccondition}, we have to choose $\eta$ small enough, depending only on the target configuration, i.e. $\eta$ small so that $$\alpha_{m'}^{(1)}\notin[\alpha_{m(l)}^{(1)}-C\eta,\alpha_{m(l)}^{(1)}-C\eta]\qquad \text{if }\alpha_{m'}\neq \alpha_{m(l)}.$$

        \item We have deformed substantially the set in the $x^{(2)}$-component. Now, in order to apply the argument for the next set $\mathcal{H}_l$ we need to compress around that coordinate. So we choose the hyperplane $\{x^{(2)}=0\}$ and we apply a contraction, denote it by $\psi^{Cont}_{T}$, until time $T_{Cont}$ for which the following is satisfied
        $$\mathrm{diam}_{(k)}\left(\psi_{T_{Cont}}^{Cont}\left(\psi_{T_{Mov}}^{Mov}\left(\psi^{Sep}(\Omega)\right)\right)\right)<\eta.$$
        Furthermore, one has that:
            \begin{align*}
                \left(\psi_{T_{Cont}}^{Cont}\left(\psi_{T_{Mov}}^{Mov}\left(\psi^{Sep}(\Omega)\cap \{x_2<-1\}\right)\right)\right)^{(1)}\subset(\Omega)^{(1)}.
            \end{align*}

        \item Redefining $\Omega:=\psi_{T_{Cont}}^{Cont}\left(\psi_{T_{Mov}}^{Mov}\left(\psi^{Sep}(\Omega)\right)\right)$, one can apply the argument recursively for the next $\mathcal{H}_l$ until all $\mathcal{H}_l$'s have its first component approximately controlled.

        \end{enumerate}
        
        \noindent Now there are $M$ sets of diameter of the order of $\eta$.
        The control time of the whole process $T_{\text{step }2}=N(T_{Sep1}+T_{Sep2}+T_{Mov}+T_{Cont})$ is of the order of
     \begin{equation}\label{T3}
        T_{\text{step }2}\sim_{\alpha} N\eta^{-1}+N\log\left(\frac{1}{\eta}\right)+N,
     \end{equation}
     while the norms of the control $b$ remained uniformly bounded.
     \begin{equation}\label{S3}
      \text{The number of switches of the controls $A,W,b$ has been of the order of $N$}
     \end{equation}

     %\end{enumerate}
     
     \medskip
     %\end{enumerate}
     {\color{black}
     Comparing the two steps, the final time horizon is of the order of
%The control time of the whole process $T=T_1+T_2+T_3+T_4$ is of the order of

     \begin{equation*}%\label{T3}
        T\lesssim_{\Omega,\alpha,d} N\left(\frac{1}{\eta}+\frac{1}{\zeta}+\log\left(\frac{1}{\eta\zeta}\right)\right).%+h^{-d(d-1)}\log\left(\frac{1}{\eta}\right)+h^{-d(d-1)},
     \end{equation*}
     The norms of the control $b$ is dependent on $N$ from Step 1.2
          \begin{equation*}%\label{T3}
        \|b\|_{L^\infty}\lesssim_{\Omega} N.%+h^{-d(d-1)}\log\left(\frac{1}{\eta}\right)+h^{-d(d-1)},
     \end{equation*}
     Finally the number of switches is of the order of $N$
     \begin{equation*}%\label{S3}
      \text{The number of switches of the controls $A,W,b$ has been of the order of $N$.}
     \end{equation*}
}
     \end{proof}
   
\begin{remark}[The Bottleneck]\label{bottleneck}
 The compression process done in the Lemma above is constituted by several stages, combining contractions and separations in different manner. However, there is a control stage which is giving the bound on the control cost which is the Ordering and Grouping step (Step 2 of the proof). In producing certain flows, such as separations, NODEs require to have higher controls than when producing contractions.
\end{remark}

\bibliography{CBIBfull16feb.bib}{}
\bibliographystyle{abbrv}

\end{document}